\documentclass[reqno,11pt]{amsart}
\usepackage{amsmath,amssymb,mathrsfs,amsthm,amsfonts,comment}
\usepackage[inline]{enumitem}
\usepackage[usenames,dvipsnames]{xcolor}
\usepackage{hyperref}
\usepackage{caption}
\usepackage{pgfplots}
\usepackage{tikz,overpic}
\usepackage{subcaption}
\usepackage{stmaryrd}
\usetikzlibrary{calc}
\usetikzlibrary{arrows.meta}
\usetikzlibrary{arrows}
\DeclareMathAlphabet{\mathpzc}{OT1}{pzc}{m}{it}

\hypersetup{%
	colorlinks=true, linkcolor=blue,
	citecolor=ForestGreen
}
\usepackage{bbm}
\usepackage[paper=letterpaper,margin=1in]{geometry}

\newtheorem{theorem}{Theorem}[section]

\newtheorem{lemma}[theorem]{Lemma}
\newtheorem{proposition}[theorem]{Proposition}
\newtheorem{corollary}[theorem]{Corollary}
\theoremstyle{definition}
\newtheorem{definition}[theorem]{Definition}
\newtheorem{remark}[theorem]{Remark}
\newtheorem{innercustomprop}{Proposition}
  {\endinnercustomprop}

\numberwithin{equation}{section}
\allowdisplaybreaks
\usepackage{acronym}

\renewcommand{\Pr}{\mathbb{P}}	
\newcommand{\Ex}{\mathbb{E}}	
\newcommand{\ind}{\mathbf{1}}	
\newcommand{\til}{\widetilde}
\newcommand{\diag}{\operatorname{diag}}
\newcommand{\hslg}{\mathpzc{HSLG}}
\newcommand{\irw}{\mathsf{IRW}}

\newcommand{\m}{\mathsf}
\newcommand{\zh}{Z}
\newcommand{\zt}{\mathpzc{Z}_N^{\mathsf{PL}}}

\newcommand{\sd}{\textcolor{black}}
\newcommand{\bl}{\textcolor{black}}

\newcommand{\R}{\mathbb{R}} 
\newcommand{\Z}{\mathbb{Z}} 

\newcommand{\e}{\varepsilon}

\newcommand{\zs}{Z_{\operatorname{sym}}}

\renewcommand{\ll}{\llbracket}
\newcommand{\rr}{\rrbracket}

\newcommand{\wb}{W_{\operatorname{br}}}
\newcommand{\pc}{\Pr_{\operatorname{block}}}
\newcommand{\ec}{\Ex_{\operatorname{block}}}
\newcommand{\pg}{\Pr_{\operatorname{gibbs}}}

\usepackage{graphicx}

\title[HSLG polymer in the bound phase]{The half-space log-gamma polymer in the bound phase}

\author[S.\ Das]{Sayan Das}
\address{S.\ Das,
	Department of Mathematics, University of Chicago,
	\newline\hphantom{\quad \ \ S. Das}
	5734 S.~University Avenue, Chicago, Illinois 60637 USA
}
\email{sayan.das@columbia.edu}

\author[W.\ Zhu]{Weitao Zhu}
\address{W.\ Zhu,
	Department of Mathematics, Columbia University,
	\newline\hphantom{\quad \ \ S. Das}
	2990 Broadway, New York, NY 10027 USA
}
\email{weitao.zhu@columbia.edu}


\begin{document}

\begin{abstract} We consider the log-gamma polymer in the half-space with bulk weights distributed as $\operatorname{Gamma}^{-1}(2\theta)$ and diagonal weights as $\operatorname{Gamma}^{-1}(\alpha+\theta)$ for $\theta>0$ and $\alpha>-\theta$. We show that in the bound phase, i.e., when $\alpha\in (-\theta,0)$, the endpoint of the polymer lies within an $O(1)$ stochastic window of the diagonal. This result gives the first rigorous proof of the pinned phenomena for the half-space polymers in the bound phase conjectured by Kardar \cite{kar2}. We also show that the limiting quenched endpoint distribution of the polymer around the diagonal is given by a random probability mass function proportional to the exponential of a random walk with log-gamma type increments.
\end{abstract}

\maketitle
{
		\hypersetup{linkcolor=black}
		\setcounter{tocdepth}{1}
	}

\begin{figure}[h!]
    \centering
    \vspace{-0.6cm}
    \includegraphics[height=6.5cm]{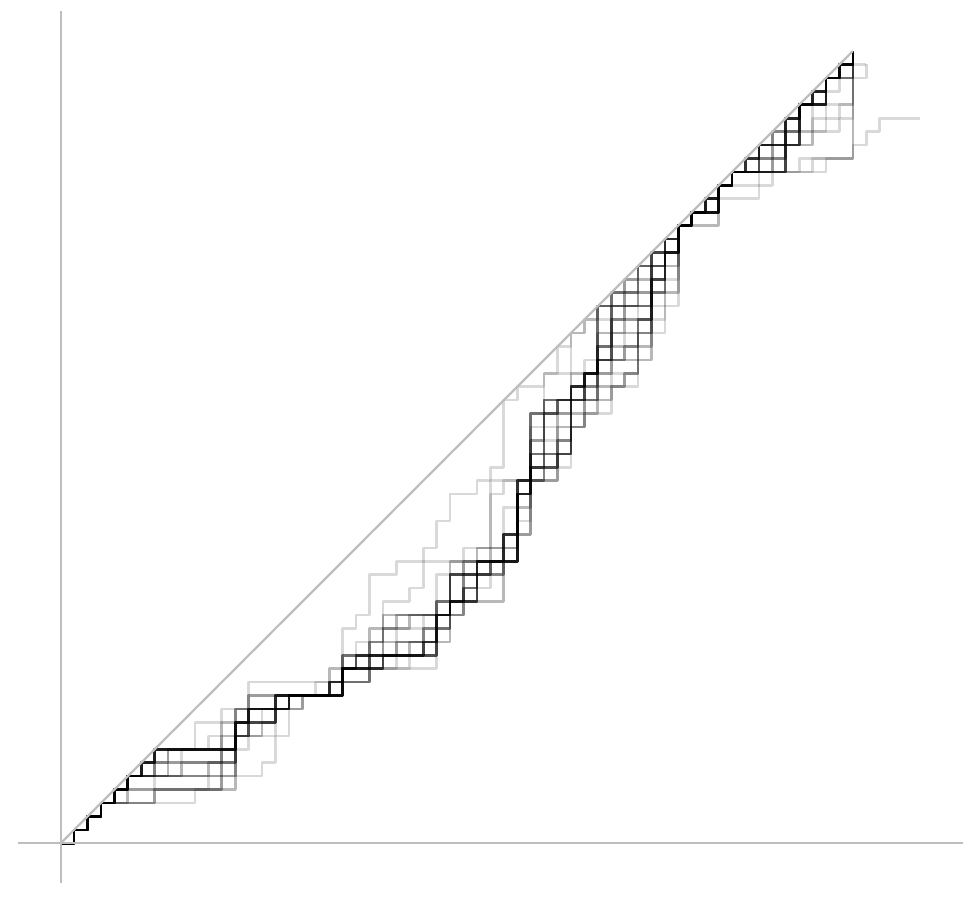} \includegraphics[height=6.5cm]{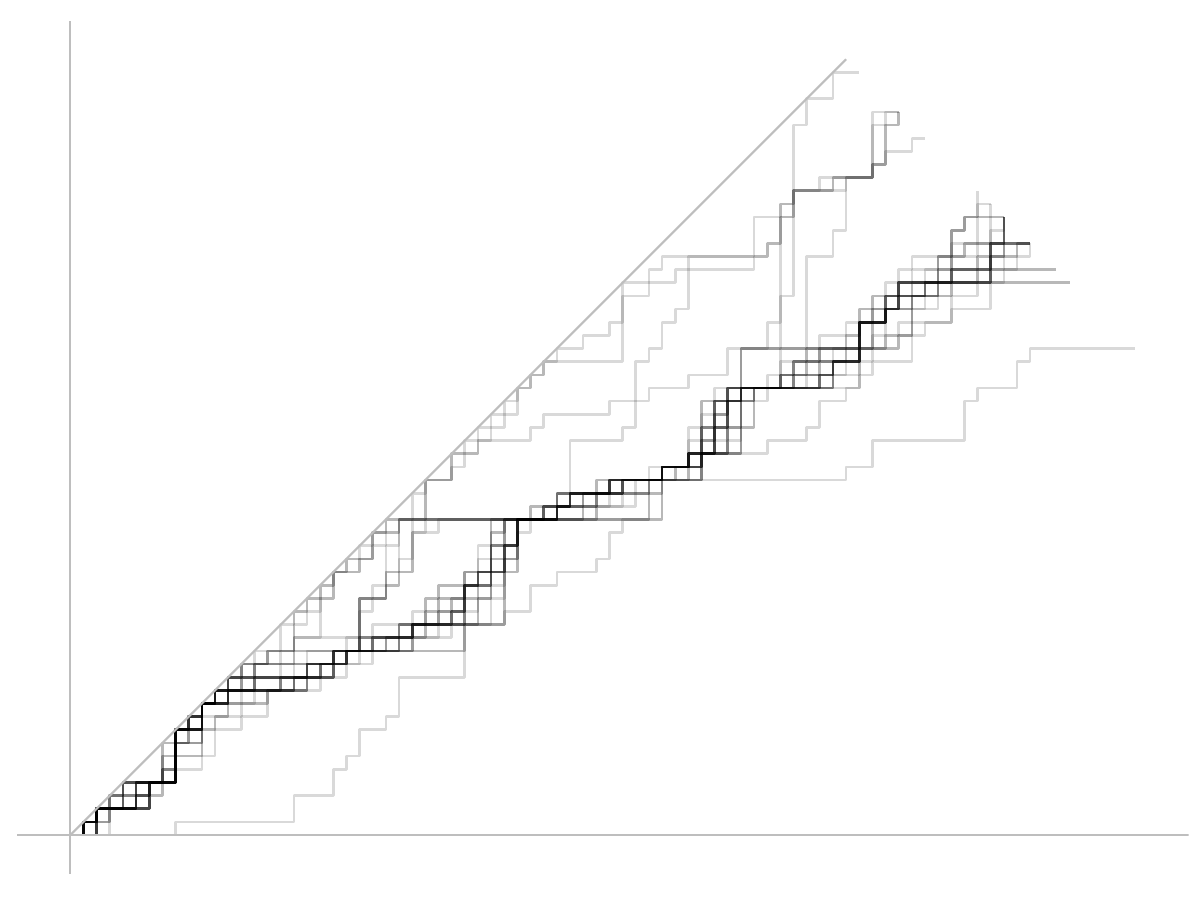}
    \caption{The bound and the unbound phase.}
    \label{f.simulation1}
\end{figure}

\section{Introduction}

 Directed polymers in random environments, first appeared in \cite{hhf, imb, bol}, are a rich class of mathematical physics models that have been extensively studied over the last several decades (see ~books \cite{bk1,gia,den,comets} and the references therein). More recently, a particular variant of the polymer models, the half-space polymers, has garnered considerable attention. The structure of the half-space polymers resembles the behavior of an interface in the presence of an attractive wall and their understanding renders importance to the studies of the wetting phenomena (\cite{abr80, psw82, bhl83}). Depending on the attraction force of the wall, it was conjectured in \cite{kar2} that these models exhibit a ``depinning" phase transition. When the attraction force exceeds a certain critical threshold (colloquially known as the bound phase), \cite{kar2} conjectured that the endpoint of the polymer stays within a $O(1)$ window around the wall, i.e., it gets pinned to the wall. In this paper, we focus on the half-space polymers with log-gamma weights which make the model integrable and demonstrate a $O(1)$ pinning phenomena of the endpoint in the bound phase. Our work is the first rigorous proof of such $O(1)$ endpoint pinning phenomena.
 
 Presently, we begin with an introduction to the model and the statements of our main results. 

\subsection{The model and the main results} Fix any $\theta>0$ and $\alpha>-\theta$ and define the half-space index set: $\mathcal{I}^{-}=\{(i,j) \in \Z_{>0}^2 \mid  j\le i\}$. We consider a family of independent variables $(W_{i,j})_{(i,j)\in \mathcal{I}^{-}}$:
	\begin{align}\label{eq:wt}
		W_{i,i}\sim 
			\operatorname{Gamma}^{-1}(\alpha+\theta) \qquad
			W_{i,j} \sim \operatorname{Gamma}^{-1}(2\theta) \ \mbox{ for } i>j,
	\end{align}
	where $\operatorname{Gamma}(\beta)$ denotes a random variable with density $\ind\{x>0\}[\Gamma(\beta)]^{-1}x^{\beta-1}e^{-x}$.
Let $\Pi_{N}^{\operatorname{half}}$ be the set of all upright lattice paths of length $2N-2$ starting from $(1,1)$  that are confined to the half-space $\mathcal{I}^{-}$ (see Figure \ref{fig1}). Given the weights in \eqref{eq:wt}, the half-space log-gamma ($\hslg$) polymer is a random measure on $\Pi_{N}^{\operatorname{half}}$ defined as
\begin{align}\label{eq:Gibbs}
		\Pr^{W}(\pi)=\frac1{\zh(N)} \prod_{(i,j)\in \pi} W_{i,j} \cdot \ind_{\pi\in \Pi_{N}^{\operatorname{half}}},
	\end{align}
where $\zh(N)$ is the normalizing constant. 

\begin{figure}[h!]
		\centering
		\begin{tikzpicture}[line cap=round,line join=round,>=triangle 45,x=0.6cm,y=0.6cm]
			\foreach \x in {0,1,2,...,7}
			{        
				\coordinate (A\x) at ($(9,0)+(\x,0)$) {};
				\coordinate (B\x) at ($(9,0)+(0,\x)$) {};
				\draw[dotted] ($(A\x)+(0,\x)$) -- ($(A\x)$);
				\draw[dotted] ($(B\x)+(7,0)$) -- ($(B\x)+(\x,0)$);
				\coordinate (C\x) at ($(16,0)+(\x,0)$) {};
				\coordinate (D\x) at ($(16,0)+(0,\x)$) {};
				\draw[dotted] ($(C\x)+(0,7-\x)$) -- ($(C\x)$);
				\draw[dotted] ($(D\x)+(7-\x,0)$) -- ($(D\x)+(0,0)$);
			}
			\draw [line width=1.5pt,color=blue] (9,0)-- (11,0)--(11,2)--(14,2)--(14,4)--(19,4);
			\draw [line width=1.5pt] (9,0)-- (10,0)--(10,1)--(13,1)--(13,3)--(15,3)--(15,4)--(16,4)--(16,6)--(17,6);	
   \begin{scriptsize}
       \node at (8.4,0) {$(1,1)$};
	\node at (19.6,4) {$(11,5)$};
	\node at (17.6,6) {$(9,7)$};
   \node[] (A) at (12.5,5.5) {$\operatorname{Gamma}^{-1}(\alpha+\theta)$};
\node[] (B) at (16.2,2.5) {$\operatorname{Gamma}^{-1}(2\theta)$};
   \end{scriptsize}
   \draw[-stealth] (12,3) -- (A);
			\draw[-stealth] (13,4) -- (A);	
			\draw[-stealth] (15,1) -- (B);
			\draw[-stealth] (16,1) -- (B);	
			\draw[-stealth] (14,1) -- (B);	
		\end{tikzpicture}
		\caption{Two possible paths of length $14$ in $\Pi_{8}^{\operatorname{half}}$ are shown in the figure.}
		\label{fig1}
	\end{figure}
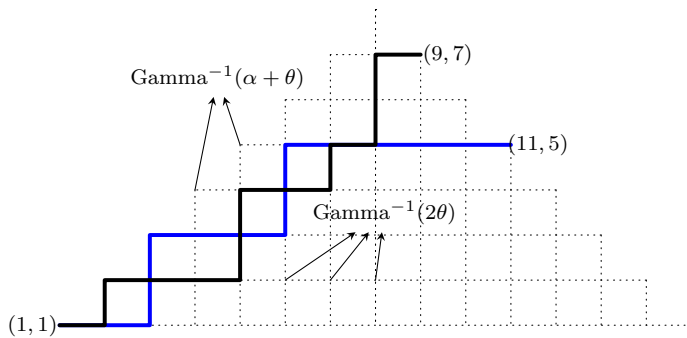



The parameter $\alpha$ controls the strength of the boundary weights, i.e. the attractiveness of the wall, and a ``depinning" phase transition occurs when $\alpha= 0$ (see \cite{kar2,ps02,bc20}). When $\alpha \ge 0$, \cite{bw, bcd23} showed that the polymer measure is unpinned and the endpoint lies in a $O(N^{2/3})$ window. For $\alpha <0,$ the conjecture is that the attraction is strong enough so that the polymer measure is pinned to the diagonal (see Figure \ref{f.simulation1}). Indeed, our first main result below confirms that in the bound phase, i.e., when $\alpha\in (-\theta,0)$, the endpoint of the $\hslg$ polymer is within $O(1)$ window of the diagonal and is the first such result to capture the ``pinning" phenomenon of the half-space polymer measure to the diagonal. 


\begin{theorem}[Bounded endpoint]\label{t:bdpt} Fix $\theta>0$ and $\alpha\in (-\theta,0)$ and consider the random measure $\Pr^W$ from \eqref{eq:Gibbs}. For a path $\pi\in \Pi_N^{\operatorname{half}}$, we denote $\pi(2N-2)$ as the height (i.e., $y$-coordinate) of the endpoint of the polymer. We have
    \begin{align}
        \lim_{k\to\infty}\limsup_{N\to\infty}\Pr^{W}(\pi(2N-2) \le N-k)=0, \quad \mbox{ in probability}.
    \end{align}
\end{theorem}

Theorem \ref{t:bdpt} is a quenched result and naturally implies its annealed version. Following the above theorem, our next point of inquiry is the limiting behavior of the quenched distribution of the endpoints around the diagonal.
We introduce and clarify a few more notations below before stating our results in this direction. Let $\Pi_{m,n}^{\operatorname{half}}$ is the set of all upright lattice paths starting from $(1,1)$ and ending at $(m, n)$ that reside solely in the half-space $\mathcal{I}^{-}$. We define the \textit{point-to-point} partition function as
 \begin{align}\label{eq:hpp}
 \zh(m,n):=\sum_{\pi\in \Pi_{m,n}^{\operatorname{half}}} \prod_{(i,j)\in \pi} W_{i,j}.
\end{align}	
Under the above definition, the normalizing constant $\zh(N)$ in \eqref{eq:Gibbs}, can also be viewed as the \textit{point-to-line} partition function, i.e.
\begin{align*}
    \zh(N)=\sum_{p=0}^{N-1} \zh(N+p,N-p).
\end{align*}
The natural logarithm of the partition function is termed as the free energy of the polymer. The aforementioned depinning phase transition can be observed by studying the fluctuations of the free energy of the polymer. In this context, \cite{bw} obtained precise one-point fluctuations for the point-to-line free energy $\log Z(N)$ in both the bound and unbound phases and observed the BBP phase transition. A similar fluctuation result and Baik-Rains phase transition were later shown in \cite{ims22} for the point-to-point free energy $\log Z(N,N)$ on the diagonal. For $\alpha \ge 0$, it was recently proven in \cite{bcd23} that the point-to-point free energy process $$\big(\log Z(N+pN^{2/3}, N-pN^{2/3})\big)_{p\in [0,r]}$$ after appropriate centering and scaling by $N^{1/3}$ is functionally tight. This result captures the characteristic KPZ $1/3$ fluctuation and $2/3$ transversal scaling exponents. In our present work, we study the point-to-point free energy process under $\alpha<0$ case. Our second main result below obtains precise fluctuations for the increments of the point-to-point free energy process when $\alpha <0$. To state the result, we introduce the definition of the \textit{log-gamma random walk}.

\begin{definition}\label{def:lgrw} Fix $\theta>0$ and $\alpha\in (-\theta,0]$. Let $Y_1 \sim \operatorname{Gamma}(\theta+\alpha)$ and $Y_2\sim \operatorname{Gamma}(\theta-\alpha)$ be independent random variables. We refer to $X:=\log Y_{2}-\log Y_1$ as a log-gamma random variable. It has a density given by
\begin{align}\label{lgrw:den}
    p(x):=\frac1{\Gamma(\theta+\alpha)\Gamma(\theta-\alpha)}\int_{\R} \exp\left((\theta-\alpha)y-e^y +(\theta+\alpha)(y-x)-e^{y-x}\right)dy. 
\end{align}
Let $(X_i)_{i\ge 0}$ be a sequence of such iid log-gamma random variables. Set $S_0=0$ and $S_k=\sum_{i=1}^k X_i$. We refer to $(S_k)_{k\ge 0}$ as a \textit{log-gamma random walk}. 
\end{definition}

Our next result states that in the bound phase, the above random walk is an \textit{attractor} for the increments of the half-space log-partition function.

\begin{theorem}\label{t:walk}  Fix $\theta>0$ and $\alpha\in (-\theta,0)$. For each $k\ge 1$, as $N\to \infty$, we have the following multi-point convergence in distribution 
\begin{align}
    \left(\frac{\zh(N+r,N-r)}{\zh(N,N)}\right)_{r\in \ll0,k\rr} \stackrel{d}{\to} \left(e^{-S_r}\right)_{r\in \ll0,k\rr},
\end{align}
where $(S_r)_{r\ge 0}$ is a log-gamma random walk from Definition \ref{def:lgrw}.
\end{theorem}

From the above result, we deduce the following limiting quenched distribution of the endpoint when viewed around the diagonal.

\begin{theorem}\label{t:qdistn}  Fix $\theta>0$ and $\alpha\in (-\theta,0)$ and consider the random measure $\Pr^W$ from \eqref{eq:Gibbs}. Let $(S_k)_{k\ge 0}$ be a log-gamma random walk from Definition \ref{def:lgrw}. Set $Q:=\sum_{p\ge 0}  e^{-S_p}$. For a path $\pi\in \Pi_N^{\operatorname{half}}$, we denote $\pi(2N-2)$ as the height (i.e., $y$-coordinate) of the endpoint of the polymer. Then for each $k\ge 1$, as $N\to \infty$, we have the following multi-point convergence in distribution 
\begin{align}\label{eq:tqd}
    \left(\Pr^{W}(\pi(2N-2)=N-r)\right)_{r\in \ll0,k\rr} \stackrel{d}{\to} \left(Q^{-1} \cdot  e^{-S_r}\right)_{r\in \ll0,k\rr}. 
\end{align}
\end{theorem}

\begin{remark} Lemma \ref{a1} ensures that $\Pr(Q\in [1,\infty))=1$ and makes the right-hand side of \eqref{eq:tqd} a valid (random) probability distribution.
In fact, by Theorem 1.2 in \cite{abo21}, once we multiply both the denominator and numerator of the r.h.s. of \eqref{eq:tqd} by $R_0 \sim \operatorname{Gamma}^{-1}(\theta-\alpha)$, then the new denominator $QR_0 \sim \operatorname{Gamma}^{-1}(-2\alpha).$
\end{remark}




\medskip

 Beyond proving the $O(1)$ transversal fluctuation around the point $(N, N)$ and pinning down the exact density within this region, our main theorems above also shed light on the attractive properties of half-space log-gamma stationary measures.  In \cite{bkld20} a stationary version of the half-space log-gamma polymer was considered for $\alpha\in (-\theta,\theta)$, where the horizontal weights along the first row are assumed to be distributed as $\operatorname{Gamma}^{-1}(\theta-\alpha)$ (i.e., $W_{i,1}\sim \operatorname{Gamma}^{-1}(\theta-\alpha)$). Let us denote $Z^{\operatorname{stat}}(n,m)$ to be the point-to-point $\hslg$ partition function computed using these weights. It was shown in \cite[Proposition 4.5]{bkld20}, that this model is stationary in the sense that for all $k\ge 1$, and $N\ge k+1$ 
 $$(\log Z^{\operatorname{stat}}(N,N)-\log Z^{\operatorname{stat}}(N+r,N-r))_{r\in \ll0,k\rr}\stackrel{d}{=} (S_r)_{r\in \ll0,k\rr}.$$ 
where $(S_r)_{r\ge 0}$ is a log-gamma random walk defined in Definition \ref{def:lgrw}. 
\begin{remark} Using the above stationary weights, one can define an associated polymer measure $\Pr_{\operatorname{stat}}^W$ in the spirit of \eqref{eq:Gibbs}. We remark that both Theorem \ref{t:bdpt} and Theorem \ref{t:qdistn} continue to hold under $\Pr_{\operatorname{stat}}^W$. This is not hard to check from our log-gamma random walk results presented in Appendix \ref{sec:app}.
\end{remark}
Theorem \ref{t:walk} shows that for $\alpha<0$ the above log-gamma random walk measure is an \textit{attractor} for the original polymer model in the sense that the increment of the log-partition function of the original model converges to the same log-gamma random walk measure. We believe that our broad techniques should also lead to a similar convergence result for $\alpha \ge 0$. 
We leave this for future consideration.

\medskip

We end this section by mentioning a recent work \cite{bc22} on the stationary measures for the $\hslg$ polymer. The point-to-point log-gamma polymer partition function $Z(n,m)$ satisfies a recurrence relation 
 \begin{align*}
     Z(n,m) & =W_{n,m} \cdot (Z(n-1,m)+Z(n,m-1)) \mbox{ for }n>m\ge 1, \\
     Z(n,n) & =W_{n,n} \cdot Z(n,n-1) \mbox{ for }n\ge 1.
 \end{align*}
We refer to a process $(h(k))_{k\ge 0}$ as horizontal-stationary for the $\hslg$ polymer if the solution to the above recurrence relation with initial data $z(\cdot,0)=e^{h(\cdot)}$ has stationary horizontal increments. For instance, the distribution of horizontal increments $(\log Z(N+k,N)-\log Z(N,N))_{k\ge 0}$ is same for all $N\ge 0$ (and equal to that of the initial data). Recently, \cite{bc22} posited a one-parameter family of horizontal-stationary measures for the $\hslg$ polymer model and conjectured that these stationary measures are attractors for a large class of initial data $(Z(n,0))_{n\ge 0}$ subject to the condition $\lim_{k\to\infty} \log Z(k,0)/k=d \in \R$.
However, the initial data relevant to our polymer model corresponds to $Z(k,0)=\ind_{k=1}$ and is not covered in \cite{bc22}. Nevertheless, our result is consistent with \cite{bc22}, in the sense that it
corresponds to the $v=\infty$ limit of their Conjecture 1.9. Our limiting distributon also matches with stationary measures for log-gamma polymer models on the strip obtained in \cite{bcy} after taking the right side of the strip to infinity.

\subsubsection{Implications of Gaussian fluctuations on the diagonal} In \cite{ims22}, the authors studied one point fluctuations of the $\hslg$ log-partition function on the diagonal, $\log Z(N,N)$, in both phases. In bound phase, they showed that \begin{align}\label{e:gaufl}
    \frac{\log Z(N,N)-RN}{\sigma\sqrt{N}} \to G,
\end{align}
where $G\sim \mathcal{N}(0,1)$ and
\begin{equation*}
    \begin{aligned}
    R(\theta, \alpha) & := -\Psi(\theta+\alpha) -\Psi(\theta - \alpha), \quad \sigma^2(\theta,\alpha) & : = \Psi'(\theta + \alpha) - \Psi'(\theta - \alpha).
\end{aligned}
\end{equation*}
Here $\Psi(\cdot)$ denote the digamma function defined on $\R_{>0}$ by
	\begin{align}\label{psidef}
		\Psi(z)=\partial_z\log \Gamma(z)=-\gamma+\sum_{n=0}^{\infty} \left(\frac1{n+1}-\frac1{n+z}\right),
	\end{align}
	where $\gamma$ is the Euler-Mascheroni constant. \bl{We also record here the expression for the $k$-th derivative of digamma function for later use.}
\begin{align}\label{deripsi}
    \bl{\Psi^{(k)}(z)=\partial^k_z \Psi(z) = (-1)^{k+1}k!\sum_{n=0}^{\infty} \frac{1}{(n+z)^{k+1}}.}
\end{align}
Combining the above result from \cite{ims22} with our results, we prove gaussianity away from the diagonal.

{\begin{theorem}\label{t:fluc} \bl{Fix $\theta>0$ and $\alpha\in (-\theta,0)$.} Fix any $k\in \Z_{> 0}$. For each $N>0$, fix $(a_{N,1},\ldots,a_{N,k}) \in \Z_{\ge 0}^k$. Suppose that as $N\to \infty$, $a_{N,i}/\sqrt{N} \to 0$ for each $i\in \{1,\ldots,k\}$. We have
\begin{align*}
    \bigg(\frac{\log Z(N+a_{N,i},N-a_{N,i})-RN}{\sigma\sqrt{N}}\bigg)_{i=1}^k \to (G,G,\ldots,G).
\end{align*}
 where $G\sim \mathcal{N}(0,1)$.   
\end{theorem}
The above theorem establishes gaussianity in the $o(\sqrt{N})$ window around the diagonal with trivial correlations. In fact, we expect the above theorem to hold even if $a_{N,i}/N\to 0$. When $a_{N,i}$ are precisely of the order $N$, we still expect to see gaussianity but with nontrivial correlations. The above result is proved using a strong coupling result (Proposition \ref{p:tech}) that we prove in Section \ref{sec:mainpf}. The gaussianity in the above theorem essentially comes from the \cite{ims22} input. However, we believe that it is possible to re-establish \eqref{e:gaufl} using the machinery developed in this paper. We leave this for future work.} 

\subsection{Proof Ideas}\label{sec:pfidea} In this section we sketch the key ideas behind the proofs of our main results. Our proof relies on inputs from the recently developed $\hslg$ Gibbsian line ensemble in \cite{bcd23}, one-point fluctuation results for point-to-(partial)line half-space log-partition functions from \cite{bw} and the localization techniques from \cite{dz1}. At the heart of our argument lies an innovative combinatorial argument that bridges the aforementioned inputs and enables our proof. 

\smallskip

\begin{figure}[h!]
    \centering
    	\begin{tikzpicture}[scale=.8,line cap=round,line join=round,>=triangle 45,x=4cm,y=1.5cm]
     
\draw [line width=1pt] (0,2.7361969564619426)-- (0.09968303324610706,2.9960253171531424);
			\draw [line width=1pt] (0.09968303324610706,5.9960253171531424)-- (0.19830773702771506,5.6884965795913667);
			\draw [line width=1pt] (0.19830773702771506,5.6884965795913667)-- (0.30006302893410286,5.889974785777928);
			\draw [line width=1pt] (0.30006302893410286,5.889974785777928)-- (0.4,6);
			\draw [line width=1pt] (0.4,6)-- (0.5,5.9);
			\draw [line width=1pt] (0.5,5.9)-- (0.5973644533669638,5.7832141676196814);
			\draw [line width=1pt] (0.5973644533669638,5.7832141676196814)-- (0.698588596048139,5.7266568810878273);
			\draw [line width=1pt] (0.698588596048139,5.7266568810878273)-- (0.8017593568577984,5.469758011699555);
			\draw [line width=1pt] (0.8017593568577984,5.469758011699555)-- (0.8990902632820055,5.0648171825842884);
			\draw [line width=1pt] (0.8990902632820055,5.0648171825842884)-- (0.9990053878836391,5.200795293991784);
			\draw [line width=1pt] (0.9990053878836391,5.200795293991784)-- (1.1079547104749525,5.1089745438145457);
			\draw [line width=1pt] (1.1079547104749525,5.1089745438145457)-- (1.2,5);
			\draw [line width=1pt] (1.2,5)-- (1.300093597749738,4.659193187756184);
			\draw [line width=1pt] (1.300093597749738,4.659193187756184)-- (1.4032643585593976,4.32597371537499);
			\draw [line width=1pt] (1.4032643585593976,4.32597371537499)-- (1.5,4);
			\draw [line width=1pt] (1.5,4)-- (1.5998727895362959,3.925290549662153);
			\draw [line width=1pt] (1.5998727895362959,3.925290549662153)-- (1.7,4);
			\draw [line width=1pt] (1.7,4)-- (1.8,3.5);
			\draw [line width=1pt] (1.8,3.5)-- (1.895758745065885,3.6867886653092736);
			\draw [line width=1pt] (1.895758745065885,3.6867886653092736)-- (2,3.5);

\draw[line width=1pt,dashed,<->] (0,0)--(2,0);
\draw[line width=1pt,dashed,<->] (0.5,3.1)--(0.5,5.8);
     
			\draw [line width=1pt] (0,2.7361969564619426)-- (0.09968303324610706,2.9960253171531424);
			\draw [line width=1pt] (0.09968303324610706,2.9960253171531424)-- (0.19830773702771506,2.6884965795913667);
			\draw [line width=1pt] (0.19830773702771506,2.6884965795913667)-- (0.30006302893410286,2.889974785777928);
			\draw [line width=1pt] (0.30006302893410286,2.889974785777928)-- (0.4,3);
			\draw [line width=1pt] (0.4,3)-- (0.5,3);
			\draw [line width=1pt] (0.5,3)-- (0.5973644533669638,2.3832141676196814);
			\draw [line width=1pt] (0.5973644533669638,2.3832141676196814)-- (0.698588596048139,2.7266568810878273);
			\draw [line width=1pt] (0.698588596048139,2.7266568810878273)-- (0.8017593568577984,2.869758011699555);
			\draw [line width=1pt] (0.8017593568577984,2.869758011699555)-- (0.8990902632820055,2.7648171825842884);
			\draw [line width=1pt] (0.8990902632820055,2.7648171825842884)-- (0.9990053878836391,2.600795293991784);
			\draw [line width=1pt] (0.9990053878836391,2.600795293991784)-- (1.1079547104749525,2.5089745438145457);
			\draw [line width=1pt] (1.1079547104749525,2.5089745438145457)-- (1.2,2.5);
			\draw [line width=1pt] (1.2,2.5)-- (1.300093597749738,2.259193187756184);
			\draw [line width=1pt] (1.300093597749738,2.259193187756184)-- (1.4032643585593976,2.32597371537499);
			\draw [line width=1pt] (1.4032643585593976,2.32597371537499)-- (1.5,2);
			\draw [line width=1pt] (1.5,2)-- (1.5998727895362959,1.925290549662153);
			\draw [line width=1pt] (1.5998727895362959,1.925290549662153)-- (1.7,2);
			\draw [line width=1pt] (1.7,2)-- (1.8,1.5);
			\draw [line width=1pt] (1.8,1.5)-- (1.895758745065885,1.6867886653092736);
			\draw [line width=1pt] (1.895758745065885,1.6867886653092736)-- (2,1.5);
			\draw [line width=1pt] (0,2.8983782378219005)-- (0.09903021247502392,2.812517559454864);
			\draw [line width=1pt] (0.09903021247502392,2.812517559454864)-- (0.20025435515619922,2.8315977102030945);
			\draw [line width=1pt] (0.20025435515619922,2.8315977102030945)-- (0.2975852615804062,2.745737031836058);
			\draw [line width=1pt] (0.2975852615804062,2.745737031836058)-- (0.40398128310420617,2.841199415267653);
			\draw [line width=1pt] (0.40398128310420617,2.841199415267653)-- (0.5000335469427568,2.698036654965482);
			\draw [line width=1pt] (0.5000335469427568,2.698036654965482)-- (0.5993110714954479,2.5740156751019847);
			\draw [line width=1pt] (0.5993110714954479,2.5740156751019847)-- (0.7,2.5);
			\draw [line width=1pt] (0.7,2.5)-- (0.8116756285443952,2.510788449351374);
			\draw [line width=1pt] (0.8116756285443952,2.510788449351374)-- (0.9,2.5);
			\draw [line width=1pt] (0.9,2.5)-- (1.0003144059631806,2.2114928108856082);
			\draw [line width=1pt] (1.0003144059631806,2.2114928108856082)-- (1.1073784030298084,2.2401130370079536);
			\draw [line width=1pt] (1.1073784030298084,2.2401130370079536)-- (1.200816073197047,1.8871302481656922);
			\draw [line width=1pt] (1.200816073197047,1.8871302481656922)-- (1.2981469796212541,1.6677085145610433);
			\draw [line width=1pt] (1.2981469796212541,1.6677085145610433)-- (1.3993711223024294,1.7344890421798496);
			\draw [line width=1pt] (1.3993711223024294,1.7344890421798496)-- (1.4947554105981522,1.3624261025893578);
			\draw [line width=1pt] (1.4947554105981522,1.3624261025893578)-- (1.5998727895362959,1.3528860272152425);
			\draw [line width=1pt] (1.5998727895362959,1.3528860272152425)-- (1.701096932217471,1.162084519732939);
			\draw [line width=1pt] (1.701096932217471,1.162084519732939)-- (1.8003744567701623,0.6278402987824895);
			\draw [line width=1pt] (1.8003744567701623,0.6278402987824895)-- (1.8938121269374009,0.6850807510271805);
			\draw [line width=1pt] (1.8938121269374009,0.6850807510271805)-- (2.0008761240040287,0.12221630395438526);
			\begin{scriptsize}
				\draw (0.93,4.6) node[anchor=west] {$H_N^{(1)}(\cdot)$};
				\draw (0.93,2.85) node[anchor=west] {$H_N^{(2)}(\cdot)$};
				\draw (0.93,1.5) node[anchor=west] {$H_N^{(3)}(\cdot)$};
                \draw (0.82,-0.2) node[anchor=west] {$M_1\sqrt{N}$};
                \draw (0.51,4.2) node[anchor=west] {$\sqrt{N}$};
			\end{scriptsize}
		\end{tikzpicture}
    \caption{First three curves of the $\hslg$ line ensemble. There is a high probability uniform separation of length $\sqrt{N}$ between the first two curves in the above $M_1\sqrt{N}$ window.}
    \label{fig:le}
\end{figure}
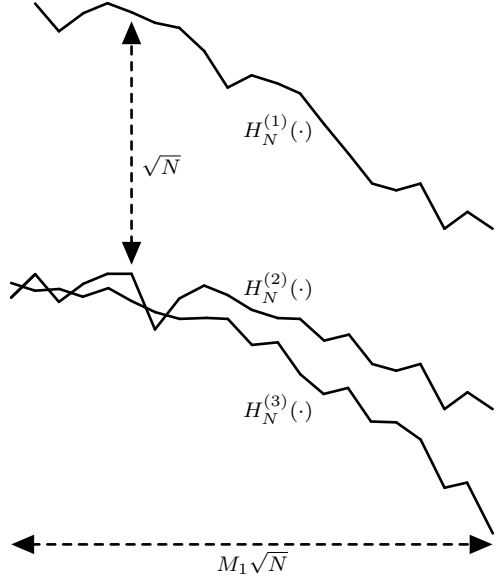

 The starting point of our analysis is the $\hslg$ Gibbsian line ensemble in \cite{bcd23}, which allows us to embed the free energy $\log Z(N+r, N-r)$ of the $\hslg$ polymer as the top curve of a Gibbsian line ensemble $(H_N^{(k)}(\cdot))_{k\in \ll1,N\rr}$ of log-gamma increment random walks interacting through a soft
version of non-intersection (Theorem \ref{t:order}) conditioning and subject to an energetic interaction at the left boundary
(where $r=0$) depending on the value of $\alpha$. This fact is due to the geometric RSK correspondence (\cite{cosz14, osz14, nz17,bz19}) and the half-space Whittaker process (\cite{bbc20}). The key idea of our proof is to show that with high probability, the first and the second curves in our line ensemble (see Figure \ref{fig:le}) are sufficiently uniformly separated. Then the separation allows us to conclude that the first curve indeed behaves similarly to a log-gamma random walk by a localization analysis. 

The existing literature contains some information about the locations of the top two curves. When $\alpha < 0$, one can deduce from the line ensemble description in \cite{bcd23} that the first and the second curves are repulsed from each other at the left boundary.  Results in \cite{bw} also supply information about the location for the first curve. However, one cannot deduce that the entire second curve lies uniformly much lower than the first curve from the above two inputs and line ensemble techniques alone.

\subsubsection{Intuition behind the separation} Before we proceed to further break down our argument about the separation, it is worth dwelling on the mathematical intuition behind the separation between the first and second curves, which originates from the definition of the line ensemble defined in Section \ref{sec2.1}. For simplicity, let us focus only on the left boundary. By Definition \ref{hslg}, we have $H_N^{(1)}(1)=\log Z(N,N)$, and 
\begin{align}\label{def2}
    H_N^{(1)}(1)+H_N^{(2)}(1):=\log \left[2\sum_{\pi_1,\pi_2} \prod_{(i,j)\in \pi_1\cup \pi_2}\til{W}_{i,j}\right],
\end{align}
where the above sum is over all pair of non-intersecting upright paths $\pi_1,\pi_2$ from $(1,1)$ to $(N,N-1)$ and from $(1,2)$ to $(N, N)$ confined in the entire quadrant $\Z_{> 0}^2$ (instead of octant). Here $\til{W}_{i,j}$ is the symmetrized version of the weights defined in \eqref{eq:wt} on the entire quadrant as:
 \begin{align}\label{eq:symwt}
			\til{W}_{i,i}=
				W_{i,i}/2 \ \mbox{ for } \ i\ge 1, \qquad
				\til{W}_{i,j}=\til{W}_{j,i}=W_{i,j} \ \mbox{ for } \ i>j.
		\end{align} 
Using point-to-(partial)line log-partition function fluctuation results from \cite{bw} and line ensemble techniques, it is not hard to deduce that 
    $\tfrac1N H_N^{(1)}(1) \to R:=-\Psi(\theta+\alpha)-\Psi(\theta-\alpha),$ where $\Psi$ is the digamma function defined in \eqref{psidef}. However, $H_N^{(2)}(1)$ should follow a different law of large numbers. This can be understood intuitively from \eqref{def2} as follows. For $\alpha$ close to $-\theta$, the weights on the diagonal are huge and stochastically dominate all the other weights. The sum in \eqref{def2} then concentrates on the pair of paths $\pi_1^*,\pi_2^*$ which jointly have the maximal numbers of diagonal points. This occurs when one of the paths carries all the diagonal weights and the other path has no diagonal weights. Thus we expect,
   \begin{align}\label{expect}
       \sum_{\pi_1,\pi_2} \prod_{(i,j)\in \pi_1\cup \pi_2}\til{W}_{i,j} \asymp \left[\sum_{\pi_1} \prod_{(i,j)\in \pi_1}\til{W}_{i,j}\right]   \cdot \left[\sum_{\pi_2 \mid \diag(\pi_2)=\varnothing} \prod_{(i,j)\in \pi_2}\til{W}_{i,j}\right]
   \end{align} 
  Upon taking logarithms and dividing by $N$, the first term goes to $R$. However, the second term does not feel  the effect of the diagonal and hence should follow the law of large numbers corresponding to the unbound phase, i.e., $\alpha>0$. The unbound phase law of large numbers is given by $-\Psi(\theta)$ noted in \cite{bw,bcd23}. Thus overall, we expect $\frac1N(H_N^{(1)}(1)+H_N^{(2)}(1)) \to R-\Psi(\theta)$ \bl{which forces  $\frac1NH_N^{(2)}(1) \to -\Psi(\theta)$. From the explicit formula of $\Psi^{(2)}$ in \eqref{deripsi}, we have that $\Psi$ is concave, which in turn implies $R>-\Psi(\theta)$.} Hence the above heuristics suggests $H_N^{(2)}(1)$ follow a lower law of large numbers. While our technical arguments to be presented later do not yield exactly \eqref{expect}, we utilize the above idea to obtain a large enough separation between the two curves, which turns out to be sufficient for proving our main theorems.



\smallskip

\subsubsection{The $U$ map and its consequences} We now describe the key idea that makes the above intuition work. All the statements mentioned in this subsection should be understood as high probability statements. The above idea of having one path having all diagonal weights is made precise in Section \ref{sec:2k}, where we develop a combinatorial map in Lemma \ref{l:umap}, referred to as the $U$ map. 

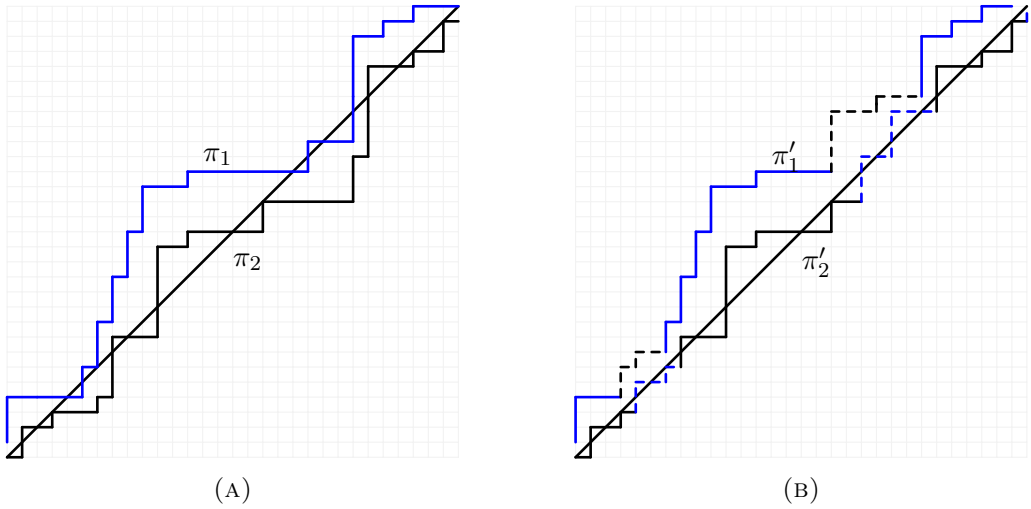
\begin{figure}[h!]
    \centering
    \begin{subfigure}[b]{0.45\textwidth}
        \centering
\begin{tikzpicture}[line cap=round,line join=round,>=triangle 45,x=0.2cm,y=0.2cm]
\draw [color=gray!10!white,, xstep=0.2cm,ystep=0.2cm] (1,1) grid (31,31);
\draw [line width=1pt] (1,1)-- (2,1);
\draw [line width=1pt] (2,1)-- (2,3);
\draw [line width=1pt] (2,3)-- (4,3);
\draw [line width=1pt] (4,3)-- (4,4);
\draw [line width=1pt] (1,1)-- (31,31);
\draw [line width=1pt] (4,4)-- (7,4);
\draw [line width=1pt] (7,4)-- (7,5);
\draw [line width=1pt] (7,5)-- (8,5);
\draw [line width=1pt] (8,5)-- (8,9);
\draw [line width=1pt] (8,9)-- (11,9);
\draw [line width=1pt] (11,9)-- (11,15);
\draw [line width=1pt] (11,15)-- (13,15);
\draw [line width=1pt] (13,15)-- (13,16);
\draw [line width=1pt] (13,16)-- (18,16);
\draw [line width=1pt] (18,16)-- (18,18);
\draw [line width=1pt] (18,18)-- (24,18);
\draw [line width=1pt] (24,18)-- (24,21);
\draw [line width=1pt] (24,21)-- (25,21);
\draw [line width=1pt] (25,21)-- (25,24);
\draw [line width=1pt] (25,24)-- (25,27);
\draw [line width=1pt] (25,27)-- (28,27);
\draw [line width=1pt] (28,27)-- (28,28);
\draw [line width=1pt] (28,28)-- (30,28);
\draw [line width=1pt] (30,28)-- (30,30);
\draw [line width=1pt] (30,30)-- (31,30);
\draw [line width=1pt,color=blue] (1,2)-- (1,5);
\draw [line width=1pt,color=blue] (1,5)-- (3,5);
\draw [line width=1pt,color=blue] (3,5)-- (6,5);
\draw [line width=1pt,color=blue] (6,5)-- (6,7);
\draw [line width=1pt,color=blue] (6,7)-- (7,7);
\draw [line width=1pt,color=blue] (7,7)-- (7,10);
\draw [line width=1pt,color=blue] (7,10)-- (8,10);
\draw [line width=1pt,color=blue] (8,10)-- (8,13);
\draw [line width=1pt,color=blue] (8,13)-- (9,13);
\draw [line width=1pt,color=blue] (9,13)-- (9,16);
\draw [line width=1pt,color=blue] (9,16)-- (10,16);
\draw [line width=1pt,color=blue] (10,16)-- (10,19);
\draw [line width=1pt,color=blue] (10,19)-- (13,19);
\draw [line width=1pt,color=blue] (13,19)-- (13,20);
\draw [line width=1pt,color=blue] (13,20)-- (20,20);
\draw [line width=1pt,color=blue] (20,20)-- (21,20);
\draw [line width=1pt,color=blue] (21,20)-- (21,22);
\draw [line width=1pt,color=blue] (21,22)-- (24,22);
\draw [line width=1pt,color=blue] (24,22)-- (24,25);
\draw [line width=1pt,color=blue] (24,25)-- (24,29);
\draw [line width=1pt,color=blue] (24,29)-- (26,29);
\draw [line width=1pt,color=blue] (26,29)-- (26,30);
\draw [line width=1pt,color=blue] (26,30)-- (28,30);
\draw [line width=1pt,color=blue] (28,30)-- (28,31);
\draw [line width=1pt,color=blue] (28,31)-- (31,31);
\node at (15,21) {$\pi_1$};
\node at (17,14) {$\pi_2$};
\end{tikzpicture}
\caption{}
    \end{subfigure}
\begin{subfigure}[b]{0.45\textwidth}
     \centering
\begin{tikzpicture}[line cap=round,line join=round,>=triangle 45,x=0.2cm,y=0.2cm]
\draw [color=gray!10!white,, xstep=0.2cm,ystep=0.2cm] (1,1) grid (31,31);
\draw [line width=1pt] (1,1)-- (2,1);
\draw [line width=1pt] (2,1)-- (2,3);
\draw [line width=1pt] (2,3)-- (4,3);
\draw [line width=1pt] (4,3)-- (4,4);
\draw [line width=1pt] (1,1)-- (31,31);
\draw [line width=1pt] (4,4)-- (5,4);
\draw [line width=1pt] (8,7)-- (8,9);
\draw [line width=1pt] (8,9)-- (11,9);
\draw [line width=1pt] (11,9)-- (11,15);
\draw [line width=1pt] (11,15)-- (13,15);
\draw [line width=1pt] (13,15)-- (13,16);
\draw [line width=1pt] (13,16)-- (18,16);
\draw [line width=1pt] (18,16)-- (18,18);
\draw [line width=1pt] (18,18)-- (20,18);
\draw [line width=1pt] (25,24)-- (25,27);
\draw [line width=1pt] (25,27)-- (28,27);
\draw [line width=1pt] (28,27)-- (28,28);
\draw [line width=1pt] (28,28)-- (30,28);
\draw [line width=1pt] (30,28)-- (30,30);
\draw [line width=1pt] (30,30)-- (31,30);
\draw [line width=1pt,color=blue] (1,2)-- (1,5);
\draw [line width=1pt,color=blue] (1,5)-- (4,5);
\draw [line width=1pt,color=blue] (7,8)-- (7,10);
\draw [line width=1pt,color=blue] (7,10)-- (8,10);
\draw [line width=1pt,color=blue] (8,10)-- (8,13);
\draw [line width=1pt,color=blue] (8,13)-- (9,13);
\draw [line width=1pt,color=blue] (9,13)-- (9,16);
\draw [line width=1pt,color=blue] (9,16)-- (10,16);
\draw [line width=1pt,color=blue] (10,16)-- (10,19);
\draw [line width=1pt,color=blue] (10,19)-- (13,19);
\draw [line width=1pt,color=blue] (13,19)-- (13,20);
\draw [line width=1pt,color=blue] (13,20)-- (18,20);
\draw [line width=1pt,color=blue] (24,25)-- (24,29);
\draw [line width=1pt,color=blue] (24,29)-- (26,29);
\draw [line width=1pt,color=blue] (26,29)-- (26,30);
\draw [line width=1pt,color=blue] (26,30)-- (28,30);
\draw [line width=1pt,color=blue] (28,30)-- (28,31);
\draw [line width=1pt,color=blue] (28,31)-- (30,31);
\draw [line width=1pt,dashed,color=blue] (5,4)-- (5,6);
\draw [line width=1pt,dashed,color=blue] (5,6)-- (7,6);
\draw [line width=1pt,dashed,color=blue] (7,6)-- (7,7);
\draw [line width=1pt,dashed,blue] (7,7)-- (8,7);
\draw [line width=1pt,dashed] (4,5)-- (4,7);
\draw [line width=1pt,dashed] (4,7)-- (5,7);
\draw [line width=1pt,dashed] (5,7)-- (5,8);
\draw [line width=1pt,dashed] (5,8)-- (7,8);
\draw [line width=1pt,dashed,color=blue] (20,18)-- (20,20);
\draw [line width=1pt,dashed,color=blue] (20,20)-- (20,21);
\draw [line width=1pt,dashed,color=blue] (20,21)-- (22,21);
\draw [line width=1pt,dashed,color=blue] (22,21)-- (22,24);
\draw [line width=1pt,dashed,color=blue] (22,24)-- (25,24);
\draw [line width=1pt,dashed] (18,20)-- (18,24);
\draw [line width=1pt,dashed] (18,24)-- (21,24);
\draw [line width=1pt,dashed] (21,24)-- (21,25);
\draw [line width=1pt,dashed] (21,25)-- (24,25);
\draw [line width=1pt,dashed] (21,25)-- (21,24);
\draw [line width=1pt,dashed,blue] (31,30)-- (31,31);
\node at (15,21) {$\pi_1'$};
\node at (17,14) {$\pi_2'$};
\end{tikzpicture}
\caption{}
     \end{subfigure}
\caption{The $U$ map takes $\pi_1,\pi_2$ from (A) and returns $\pi_1',\pi_2'$ in (B). The precise description of the map is given in the proof of Lemma \ref{l:umap}}
    \label{fig:d}
\end{figure}

The $U$ map takes every pair of paths $\pi_1,\pi_2$ in the sum in \eqref{def2} and returns a pair of non-intersecting paths $\pi_1',\pi_2'$ while preserving their shared weights up to reflections (see Figure \ref{fig:d}).  Moreover, the diagonal weights collectively carried by the pair will only rest on one of the paths among $\pi_1',\pi_2'$. The $U$ map is not injective but has at most $2^N$ many pre-images for each pair in its image (i.e., $|U^{-1}(\pi_1',\pi_2')|\le 2^N$).
When we apply the $U$ map to a single pair of adjacent paths, we get that $$\tfrac{1}{N}(H_N^{(1)}(1)+H_N^{(2)}(1)) \le \log 2+R-\Psi(\theta).$$ The $\log 2$ is an entropy factor that comes from overcounting the number of inverses of our $U$ map. To remove its influence, we rely on the definition of the lower curves of the line ensemble. Indeed, similar to \eqref{def2}, $\sum_{i=1}^{2k} H_N^{(i)}(1)$ admits a representation in terms of $2k$-many non-intersecting paths. When we apply the $U$ map to $k$ pairs of adjacent paths simultaneously, it leads to the following average law of large numbers of the top $2k$ curves: $$\frac{1}{2kN}\sum_{i = 1}^{2k} H_N^{(i)}(1)\le \tfrac1{2k}\log 2 -\tfrac12\Psi(\theta) - \tfrac12\Psi(\theta+\alpha)-\tfrac12\Psi(\theta-\alpha).$$
Taking $k$ large enough, one can ensure the right-hand side constant is strictly less than $R$. In fact, the above argument can be strengthened to conclude that for large enough $k$
\begin{align*}
    \sup_{p\in \ll1,2N-4k+2\rr} \frac1{2kN}\sum_{i=1}^{2k} H_N^{(i)}(p) \le R-\delta,
\end{align*}
for some $\delta>0$. This is obtained in Proposition \ref{p:alaw}.

As a consequence of this result, using soft non-intersection property of the line ensemble (Theorem \ref{t:order}), we derive that with high probability, the $(2k+2)$-th curve $H_N^{2k+2}(\cdot)$ is uniformly $\mbox{Const}\cdot N$ below $RN$ over $\ll 1, N\rr$ in Section \ref{sec:2nd}. Employing one-point results from \cite{bw}, one can ensures the point $H_N^{(1)}(M_1\sqrt{N})$ on the top curve is $(M_2+1)\sqrt{N}$ below $RN.$ Combining the last two results and line ensemble techniques we are able to benchmark the second curve from above:
\begin{align}\label{e.bench}
   \sup_{p \in \ll 1, M_1\sqrt{N} \rr} H_N^{(2)}(p) \le RN - M_2\sqrt{N} 
\end{align}
in Proposition \ref{p:benchmark}. 
The details of the argument are presented in Section \ref{sec:2nd}. While we are unable to obtain a mismatch in the laws of large numbers for the first two curves following the above procedures, the fact that the second curve is  below the diffusive regime of the first curve (since $M_2$ can be chosen as large as possible) over an interval of length $M_1\sqrt{N}$ is sufficient for our next step of the analysis. 


\subsubsection{Localization analysis}\label{sec1.2.4} The remaining piece of our proof of main theorems boils down to a localization analysis of the first curve in Section \ref{sec:mainpf}. Our proof roughly follows the techniques developed in our paper \cite{dz1}. First, to prove Theorem \ref{t:bdpt} we divide the tail into a deep and a shallow tail depending on the distance away from $(N, N)$, see Figure \ref{f.tri1}.

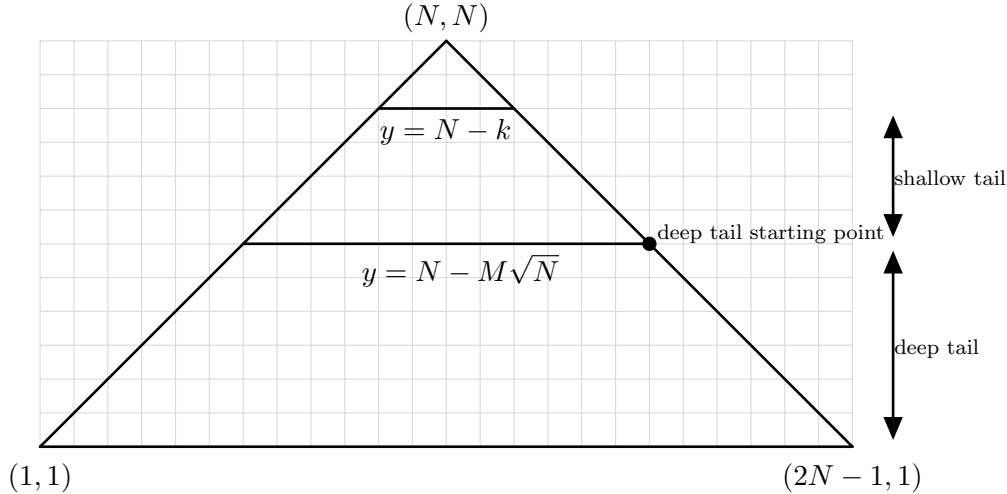
\begin{figure}[h!]
    \centering
    \begin{tikzpicture}[line cap=round,line join=round,>=triangle 45,x=0.9cm,y=0.9cm]
    \draw [color=gray!30!white,, xstep=0.45cm,ystep=0.45cm] (1,1) grid (13,7);
\draw [line width=1pt] (7,7)-- (13,1);
\draw [line width=1pt] (7,7)-- (1,1);
\draw [line width=1pt] (1,1)-- (13,1);
\draw [line width=1pt] (4,4)-- (10,4);
\draw [line width=1pt] (6,6)-- (8,6);
\draw [line width=1pt,<->] (13.6,5.9)-- (13.6,4.1);
\draw [line width=1pt,<->] (13.6,3.9)-- (13.6,1.1);
\draw (5.6,4) node[anchor=north west] {$y=N-M\sqrt{N}$};
\draw (5.87,6) node[anchor=north west] {$y=N-k$};
\draw (7,7.7) node[anchor=north] {$(N,N)$};
\draw (1,0.9) node[anchor=north] {$(1,1)$};
\draw (13,0.9) node[anchor=north] {$(2N-1,1)$};
\begin{scriptsize}
\draw [fill=black] (10,4) circle (2.5pt);
\draw (10,4.4) node[anchor=north west] {deep tail starting point};
\draw (13.5,2.7) node[anchor=north west] {deep tail};
\draw (13.5,5.2) node[anchor=north west] {shallow tail};
\end{scriptsize}
\end{tikzpicture}
    \caption{{If the height of the endpoint of the polymer is less than $N-k$, it either lies in the shallow tail or in the deep tail (illustrated above). Lemma \ref{l:deep} shows it is exponentially unlikely to lie in the deep tail.}}
    \label{f.tri1}
\end{figure}

Our argument in Lemma \ref{l:deep} uses one-point fluctuations results of point-to-(partial)line log-partition function from \cite{bw} as input and shows that the probability of the endpoint living in the deep tail region is exponentially small. To show that the shallow tail contribution is also small and to prove our remaining theorems, we establish the following strong convergence result in Proposition \ref{p:tech}:
\begin{enumerate}[label=(\alph*),leftmargin=18pt]
\setlength\itemsep{0.5em}
    \item \label{1} the law of the top curve within the $[1,M\sqrt{N}]$ window is arbitrarily close to that of a log-gamma random walk for large enough $N$ (Proposition \ref{p:tech}).
    \end{enumerate}
In light of \ref{1}, the conclusion that the shallow tail contribution is small follows from estimating the probability of the same event under the log-gamma random walk law. Theorem \ref{t:walk} is immediate from \ref{1} and Theorem \ref{t:qdistn} also follows from \ref{1} after some calculations. The details are presented in Section \ref{sec.maintech}. 

Finally, we briefly explain how we establish \ref{1}. A detailed discussion appears in the Step 1 of the proof of Proposition \ref{p:tech}. As $H_N^{(1)}(\cdot)$ is a log-gamma random walk subject to soft non-intersecting condition with $H_N^{(2)}(\cdot)$, it suffices to show that there's sufficient distance between the first and the second curves. Indeed, this will imply $H_N^{(1)}$ behaves like a true log-gamma random walk. As we have already benchmarked the second curve in \eqref{e.bench}, it remains to determine a suitable lower bound for the first curve. The key idea here is to find a point $p=O(\sqrt{N})$ on the first curve in the deep tail region such that with high probability $$H_N^{(1)}(p) \ge RN-M'\sqrt{N}$$ for some $M'$. This is achieved in Lemma \ref{l:high} using fluctuation results from \cite{bw}. Then using standard random walk tools such as Kolmogorov's maximal inequality, we derive that with high probability $H_N^{(1)}(q) \ge RN-(M'+1)\sqrt{N}$ for all $q\in \ll1,p\rr$. Choosing $M_2=M'+2$ in \eqref{e.bench} implies that with high probability the first curve is at least $\sqrt{N}$ above the second curve. This completes our deduction and consequently establishes \ref{1}. 

\subsection{Related works and future directions} \label{sec1.3}
 Our study of half-space polymers succeeds an extensive history of endeavors that attempt to unravel their full-space variant. These full-space polymer models have rich connections with symmetric functions, random matrices, stochastic PDEs and integrable systems and are believed to belong to the KPZ universality class (see \cite{comets, gia, bc20}). Yet in spite of intense efforts in the past decade, rigorous results proving either the $1/3$ fluctuation exponent or the $2/3$ transversal exponent for general polymers have been scarce outside a few integrable cases (see \cite{comets, timo, bc20, bcd21, dz1, dz2} and the references therein).

In the half-space geometry, a wealth of literature has focused on the phase diagram for limiting distributions based on the diagonal strength. One of the first mathematical works goes to the series of joint works \cite{br1, br2, br3} on the geometric last passage percolation (LPP), i.e. polymers with zero temperature. Their multi-point fluctuations were studied in \cite{si04} and similar results were later proved for exponential LPP in
\cite{bbcs18a, bbcs18b} using Pfaffian Schur processes. For further recent works on half-space LPP,
we refer to \cite{bbnv18, bfo20, bfo22, fo22} and the references therein.

For positive temperature models, i.e., polymers, as they are no longer directly related to the Pfaffian point processes, the first rigorous proof of the depinning transition appeared much later in \cite{bw}. Here the authors also included precise fluctuation results such as the BBP phase transition
\cite{bap05} for the point-to-line log-gamma free energy. For the point-to-point log-gamma free energy,
the limit theorem as well as the Baik-Rains phase transition were conjectured in \cite{bbc20} based on steepest descent analysis of half-space Macdonald
processes. This result has been recently proved in \cite{ims22} by relating the half-
space model to a free boundary version of the Schur process. 

Similar to their full-space counterparts, in addition to fluctuations, another dimension of interest to half-space polymers is their localization behaviors, which refer to the concentration of polymers in a very small region given the environment. Figure \ref{f.simulation1} is a simulation of $30$ samples of $\hslg$ polymers of length $120$ sampled from the same environment with $\theta=1$, $\alpha=-0.2$ and $\alpha=+0.2$. The simulation suggests that even in the unbound phase, we expect a localization phenomenon around a favorite site given by the environment. Localization is a unique behavior of the polymer path in the strong disorder regime. 
In the full space, various levels of localization results have been established for discrete and continuous polymers. The mathematical work began with the \textit{strong} localization result of \cite{carmonaHu1} that confirmed the existence of the favorite sites for the endpoint distributions of point-to-line polymers and has been upgraded to the notion of \textit{atomic} and \textit{geometric} localization for general reference random walks in a series of joint works \cite{bates, bc20, bak,bak22}. An even stronger notion, the ``favorite region conjecture", which conjectures the favorite corridor of a polymer to be stochastically bounded, has been proved for two integrable models: the stationary log-gamma polymer in the discrete case (\cite{comngu}) and the continuous directed random polymer (CDRP) in the continuous case (\cite{dz1}). In this direction, building up on \cite{dz1} work, recently  \cite{yugu22} 
have studied the localization distance of the CDRP. 

Investigating the geometry of the half-space CDRP is an interesting question to consider next. Recently, a number of new results have appeared on the half-space KPZ equation, which arises as the free energy of the half-space CDRP \cite{wu20, bc22}, in both the mathematics
\cite{cs18, bbcw18, bbc20, par19, par22, bc22, ims22} and the physics literature \cite{gld12, bbc16,
it18, dnkldt20, kld18, bkld20, bld21, bkld22}. 
These results on the free energy render the half-space continuous polymers amenable to analysis. However, the challenge with further studying the geometry of the half-space CDRP remains, due to the lack of an analogous half-space KPZ line ensemble.



 
\subsection*{Outline}
The rest of the paper is organized as follows. Section \ref{sec:bac} reviews some of the existing results related to $\hslg$ line ensemble and one-point fluctuations of point-to-(partial)line free energy of $\hslg$ polymer. In Section \ref{sec:2k} we prove our key combinatorial lemma and use it to control the average law of large numbers for the top curves of the line ensemble. In Section \ref{sec:2nd}, we establish control over the second curve of the line ensemble. Finally, in Section \ref{sec:mainpf}, we complete the proofs of our main theorems. Appendix \ref{sec:app} contains basic properties of log-gamma random walks.
\subsection*{Notation} Throughout this paper, we will assume $\theta>0$ and $\alpha\in (-\theta,0)$ are fixed parameters. We write $\ll a,b\rr:=[a,b]\cap \Z$ to denote the set of integers between $a$ and $b$. We will use serif fonts such as $\m{A}, \m{B}, \ldots$ to denote events. The complement of
an event $\m{A}$ will be denoted as $\neg\m{A}$.

\subsection*{Acknowledgements} We thank Guillaume Barraquand and Ivan Corwin for useful discussions, comments, and for their encouragement during the completion of this manuscript. SD's research was partially supported by Ivan Corwin's  W.M.~Keck Foundation Science and
Engineering Grant. We thank the anonymous
referees for their careful reading and useful comments on improving our manuscript.

\section{Basic framework and tools}
\label{sec:bac}

In this section, we present the necessary background on the half-space log-gamma ($\hslg$) line ensemble and point-to-(partial) line partition function. From  \cite{bcd23} and \cite{bw} we gather a few of the known results on these objects that are crucial in our proofs.

\subsection{The $\hslg$ line ensemble and its Gibbs property} \label{sec2.1}
We begin with the description of the $\hslg$ line ensemble and its Gibbs property. The definition of the $\hslg$ line ensemble is based on the point-to-point symmetrized partition function for multiple paths defined in \eqref{e:nz}. These are sum over multiple non-intersecting upright paths on the entire quadrant $\Z_{>0}^2$ of products of the symmetrized version defined in \eqref{eq:symwt} of the weights defined in \eqref{eq:wt}. Fix $m,n,r\in \Z_{>0}$ with $n\ge r$, let $\Pi_{m,n}^{(r)}$ be the set of all $r$-tuples of non-intersecting upright paths in $\Z_{>0}^2$ starting from $(1,r), (1,r-1),\ldots, (1,1)$ and going to $(m,n),(m,n-1), \ldots, (m,n-r+1)$ respectively. We define the point-to-point symmetrized partition function for $r$ paths as
		\begin{align}\label{e:nz}
			\zs^{(r)}(m,n):=\sum_{(\pi_1,\ldots,\pi_r)\in \Pi_{m,n}^{(r)}}  \prod_{(i,j)\in \pi_1\cup \cdots \cup \pi_r} \til{W}_{i,j}.
		\end{align}
 where $\til{W}_{i,j}$ are defined in \eqref{eq:symwt}. We write $\zs(m,n):=\zs^{(1)}(m,n)$ and use the convention that $\zs^{(0)}(m,n)\equiv 1$. One can recover $\hslg$ partition function from symmetrized partition function via the following identity. For each $(m,n)\in \mathcal{I}^{-}$ we have 
\begin{align}\label{l:iden}
   2\zs(m,n)=\zh(m,n). 
\end{align}
 The above identity appears in Section 2.1 of \cite{bw} and follows easily due to the symmetry of the weights. We stress that the above relation is an exact equality not just in distribution.

\begin{definition}[$\hslg$ line ensemble]\label{hslg} Fix $N>1$. For each $k\in \ll1,N\rr$ and $p \in \ll1,2N-2k+2\rr$ set
\begin{align}\label{eq:hslg}
    H_N^{(k)}(p) := \log\left(\frac{2\zs^{(k)}(N+\lfloor p/2\rfloor ,N- \lceil p/2\rceil+1)}{\zs^{(k-1)}(N+\lfloor p/2\rfloor ,N- \lceil p/2\rceil+1)}\right)
\end{align}
     We view the $k$-th curve $H_N^{(k)}$ as a random continuous function $H_N^{(k)}: [1,2(N-k+1)]\to \R$ by linearly interpolating its values on integer points. We call the collection of curves $H_N:=(H_N^{(1)},H_N^{(2)},\ldots,H_{N}^{(N)})$ the $\hslg$ line ensemble.  
\end{definition}

\begin{figure}[h!]
		\centering
		\begin{tikzpicture}[line cap=round,line join=round,>=triangle 45,x=0.6cm,y=0.6cm]
  \draw [dotted, color=gray,xstep=0.6cm,ystep=0.6cm] (9,0) grid (23,7);
  \draw [dotted, color=gray] (9,0)--(16,7);
   \foreach \x in {0,1,2,...,6}{
			\draw [line width=1.5pt] (16+\x,7-\x)-- (17+\x,7-\x)--(17+\x,6-\x);
   }
       \node at (16,7.5) {$(8,8)$};
       \node at (9,-0.5) {$(1,1)$};
   \node[] (A) at (15.5,4) {$H_8^{(k)}(3)$};
\node[] (B) at (20.5,8) {$H_8^{(k)}(4)$};
\node[] (C) at (17.5,2.5) {$H_8^{(k)}(5)$};
   \draw[-stealth] (17,6) -- (A);
    \draw[-stealth] (18,6) -- (B);
     \draw[-stealth] (18,5) -- (C);
     \draw [fill=black] (17,6) circle (2pt);
     \draw [fill=black] (18,5) circle (2pt);
     \draw [fill=black] (18,6) circle (2pt);
		\end{tikzpicture}
		\caption{The lattice points used in computing in $H_N^{(k)}(p)$ is shown above for $N=8$ and $p=3,4,5$.}
		\label{fig1b}
	\end{figure}
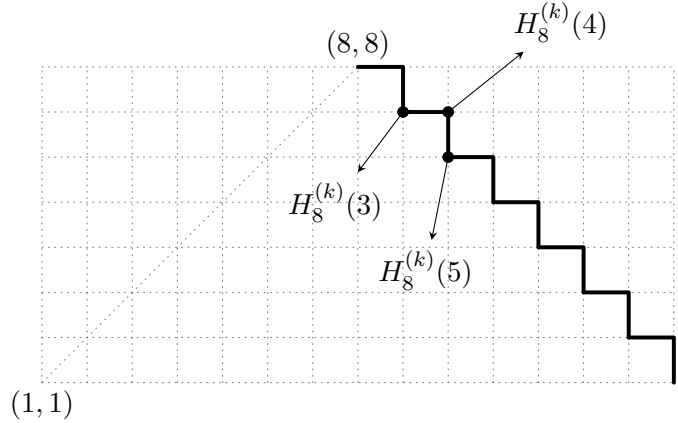
\bl{Some remarks and observations related to the above definition are in order. For each $k$, $H_N^{(k)}(\cdot)$ is computed by taking the logarithm of the ratio of appropriate symmetrized partition function along a staircase path (see Figure \ref{fig1b}). $H_N^{(k)}(p)$ is not defined for $p> 2N-2k+2$ as $\Pi_{N+\lfloor p/2\rfloor ,N- \lceil p/2\rceil+1}^{(k)}=\varnothing$ for $p>2N-2k+2$. The prefactor of $2$ in \eqref{eq:hslg} is kept to ensure that for all $p\le 2N$ we have
\begin{align}\label{eq:rem}
    H_N^{(1)}(p) =\log \zh(N+\lfloor p/2\rfloor ,N- \lceil p/2\rceil +1),
\end{align}
thanks to the relation \eqref{l:iden}. } We remark that in Definition 2.7 in \cite{bcd23}, the authors defined the $\hslg$ line ensemble by defining $\mathcal{L}_i^N(j)=H_N^{(i)}(j)+\mbox{Const}\cdot N$ where the $\mbox{`Const'}$ is explicit and encodes the law of large numbers for the $\hslg$ free energy process (as well as the entire line ensemble) in the unbound phase. Since the laws of large numbers for the first curve and the second curve in the bound phase are possibly different (recall our discussion of the proof idea in the introduction), we choose to not add this constant in our definition of line ensemble. All the results from \cite{bcd23} can be easily translated to results in our setting by adding this appropriate constant.

\medskip

The $\hslg$ line ensemble enjoys a property known as the $\hslg$ Gibbs property. To state the $\hslg$ Gibbs property, we introduce the $\hslg$ Gibbs measures via graphical representation. 

Consider a diamond lattice on the lower-right quadrant with vertices $\{(m,-n),(m+\frac12,-n+\frac12)\mid m,n\in \Z_{>0}^2\}$ and nearest neighbor edges as shown in Figure \ref{fig:5}. We label the vertices by setting $\phi((m,n))=(-\lfloor n\rfloor,2m-1)$. In the rest of the paper, we identify a vertex in this lattice by this labeling and will not mention its actual coordinates.

On the diamond lattice domain, we add potential \textit{directed colored edges}. A directed colored edge $\vec{e}=\{v_1\to v_2\}$ on this lattice is a directed edge from $v_1$ to $v_2$ that has three choices of colors: $\textcolor{blue}{\textbf{blue}}, \textcolor{red}{\textbf{red}},\mbox{ and }\textcolor{black}{\textbf{black}}$. \bl{Given a directed colored edge $\vec{e}$, we associate a weight function $W_{\vec{e}} : \R \to \R_{>0}$ based on the color of the edge defined as follows}:
		\begin{align}\label{eq:color}
			W_{\vec{e}}(x)=\begin{cases}
				\exp((\theta-\alpha)x-e^x) & \mbox{ if $\vec{e}$ is \textcolor{blue}{\textbf{blue}}} \\
				\exp((\theta+\alpha)x-e^x) & \mbox{ if $\vec{e}$ is \textcolor{red}{\textbf{red}}}\\
    \exp(-e^x) & \mbox{ if $\vec{e}$ is \textcolor{black}{\textbf{black}.}}
			\end{cases}
		\end{align}
We consider a graph $G_N$ on the diamond lattice with vertex set
 \begin{align*}
      K_N & :=\{(i,j)\mid i\in\ll1,N\rr, j\in \ll1,2N-2i+2\rr\}.
 \end{align*}
with directed colored edges described below. For each $(p,q)\in K_N$, 
		\begin{itemize}[leftmargin=15pt]
			\item  If $q$ is odd and $p$ is odd (even resp.), we put a \textcolor{blue}{\textbf{blue}} (\textcolor{red}{\textbf{red}} resp.) edge: $(p,q)\to(p,q+1)$. 
            \item If $q\ge 3$ is odd and $p$ is odd (even resp.), we put a \textcolor{red}{\textbf{red}} (\textcolor{blue}{\textbf{blue}} resp.) edge: $(p,q)\to(p,q-1)$.
			\item If $q$ is even, we put two \textbf{black} edges: $(p,q)\to (p-1,q)\mbox{ and }(p,q)\to (p+1,q).$
		\end{itemize}
	The corresponding graph is shown in Figure \ref{fig:5}.	We write $E(F)$ for the set of edges of any graph $F\subset G_N$. 

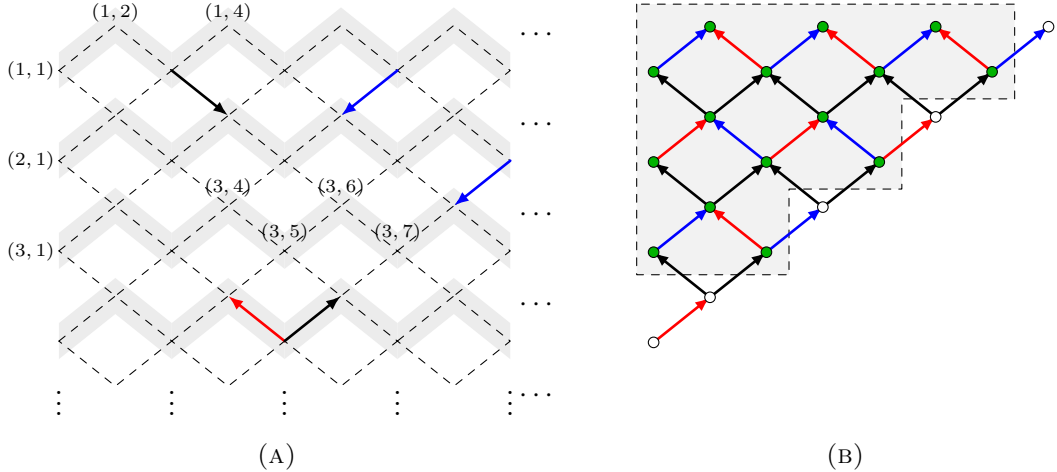
\begin{figure}[h!]
    \centering
\begin{subfigure}[b]{0.45\textwidth}
        \centering
    \begin{tikzpicture}[line cap=round,line join=round,>=triangle 45,x=1.5cm,y=1.2cm]
\foreach \x in {1,2,3,4}{
\foreach \y in {1,2,3,4}{
\path[fill=gray!15] (\x,-\y+0.2)--(\x+0.5,-\y+0.7)--(\x+1,-\y+0.2)--(\x+1,-\y-0.2)--(\x+0.5,-\y+0.3)--(\x,-\y-0.2)--(\x,-\y+0.2);
}
}
    \foreach \x in {1,2,3,4}{
    \foreach \y in {1,2,3,4}{
    \draw[dashed] (\x,-\y)--(\x+0.5,-\y+0.5)--(\x+1,-\y)--(\x+0.5,-\y-0.5)--(\x,-\y);
    }
    }
    \foreach \x in {0,1,2,3,4}{
    \node at (\x+1,-4.6) {$\vdots$};
    \node at (5.25,-\x-0.6) {$\cdots$};
    }
    \node at (0.75,-1) {{\tiny${(1,1)}$}};
    \node at (0.75,-2) {{\tiny${(2,1)}$}};
    \node at (0.75,-3) {{\tiny${(3,1)}$}};
    \node at (1.5,-0.35) {{\tiny${(1,2)}$}};
    \node at (2.5,-0.35) {{\tiny${(1,4)}$}};
    \node at (2.5,-2.3) {{\tiny${(3,4)}$}};
    \node at (3,-2.8) {{\tiny${(3,5)}$}};
    \node at (3.5,-2.3) {{\tiny${(3,6)}$}};
    \node at (4,-2.8) {{\tiny${(3,7)}$}};
    \draw[line width=1pt,black,{Latex[length=2mm]}-]  (2.5,-1.5)--(2,-1);
    \draw[line width=1pt,red,{Latex[length=2mm]}-]  (2.5,-3.5)--(3,-4);
    \draw[line width=1pt,black,{Latex[length=2mm]}-]  (3.5,-3.5)--(3,-4);
    \draw[line width=1pt,blue,{Latex[length=2mm]}-]  (4.5,-2.5)--(5,-2);
    \draw[line width=1pt,blue,{Latex[length=2mm]}-]  (3.5,-1.5)--(4,-1);
\end{tikzpicture}
\caption{}
\end{subfigure}
    \begin{subfigure}[b]{0.45\textwidth}
        \centering
        \begin{tikzpicture}[line cap=round,line join=round,>=triangle 45,x=1.5cm,y=1.2cm]
				\draw[fill=gray!10,dashed] (-0.65,0.25)--(-0.65,-2.75)--(0.7,-2.75)--(0.7,-1.8)--(1.7,-1.8)--(1.7,-0.8)--(2.7,-0.8)--(2.7,0.25)--(-0.65,0.25);
				
				\foreach \x in {0,1}
				{    
					\draw[line width=1pt,blue,{Latex[length=2mm]}-]  (\x,0) -- (\x-0.5,-0.5); 
					\draw[line width=1pt,red,{Latex[length=2mm]}-] (\x,0) -- (\x+0.5,-0.5);
					\draw[line width=1pt,black,{Latex[length=2mm]}-] (\x-0.5,-0.5) -- (\x,-1);
					\draw[line width=1pt,black,{Latex[length=2mm]}-] (\x+0.5,-0.5) -- (\x,-1);
					\draw[line width=1pt,red,{Latex[length=2mm]}-]  (\x,-1) -- (\x-0.5,-1.5); 
					\draw[line width=1pt,blue,{Latex[length=2mm]}-] (\x,-1) -- (\x+0.5,-1.5);
					\draw[line width=1pt,black,{Latex[length=2mm]}-] (\x-0.5,-1.5) -- (\x,-2);
					\draw[line width=1pt,black,{Latex[length=2mm]}-] (\x+0.5,-1.5) -- (\x,-2);
				}
			\draw[line width=1pt,blue,{Latex[length=2mm]}-]  (3,0) -- (3-0.5,-0.5); 
			\draw[line width=1pt,blue,{Latex[length=2mm]}-]  (2,0) -- (2-0.5,-0.5); 
			\draw[line width=1pt,red,{Latex[length=2mm]}-]  (2,0) -- (3-0.5,-0.5);
			\draw[line width=1pt,red,{Latex[length=2mm]}-]  (2,-1) -- (2-0.5,-1.5); 
			\draw[line width=1pt,blue,{Latex[length=2mm]}-]  (1,-2) -- (1-0.5,-2.5); 
			\draw[line width=1pt,blue,{Latex[length=2mm]}-]  (0,-2) -- (0-0.5,-2.5); 
			\draw[line width=1pt,red,{Latex[length=2mm]}-]  (0,-2) -- (1-0.5,-2.5); 
			\draw[line width=1pt,red,{Latex[length=2mm]}-]  (0,-3) -- (-0.5,-3.5);  
			\draw[line width=1pt,black,{Latex[length=2mm]}-]  (-0.5,-2.5) -- (0,-3); 
			\draw[line width=1pt,black,{Latex[length=2mm]}-]  (0.5,-2.5) -- (0,-3); 
			\draw[line width=1pt,black,{Latex[length=2mm]}-]  (1.5,-0.5) -- (2,-1);
			\draw[line width=1pt,black,{Latex[length=2mm]}-]  (2.5,-0.5) -- (2,-1); 
			\foreach \x in {0,1,2}
			{	
			\draw [fill=green!70!black] (-0.5,-\x-0.5) circle (2pt);
			\draw [fill=green!70!black] (0,-\x) circle (2pt);
			\draw [fill=green!70!black] (0.5,-\x-0.5) circle (2pt);
		}
	\foreach \x in {0,1}
	{	
		\draw [fill=green!70!black] (1,-\x) circle (2pt);
		\draw [fill=green!70!black] (1.5,-\x-0.5) circle (2pt);
	}
	\draw [fill=green!70!black] (2,0) circle (2pt);
	\draw [fill=green!70!black] (2.5,-0.5) circle (2pt);
	\foreach \x in {0,1,2,3}
	{\draw[fill=white] (\x,-3+\x)  circle (2pt);}
	\draw[fill=white] (-0.5,-3.5)  circle (2pt);
 \draw[white] (-0.5,-4.35) circle (2pt);
			\end{tikzpicture}
    \caption{}
    \end{subfigure}
    \caption{(A) Diamond lattice with a few of the labeling of the vertices shown in the figure.  The $m$-th gray-shaded region have vertices with labels of the form $\{(m,n)\mid n\in \Z_{>0}^2\}$. Thus each such region consists of vertices with the same first coordinate labeling. Potential directed colored edges on the lattice are also drawn above. (B) $K_N$ with $N=4$. $\Lambda_N^*$ consists of all vertices in the shaded region.}
    \label{fig:5}
\end{figure}

The following result from \cite{bcd23} shows how the conditional distribution of the $\hslg$ line ensemble is given by certain measures called $\hslg$ Gibbs measures.
  
\begin{theorem}[Gibbs property]\label{t:gibbs}
Consider the directed colored graph $G_N$ described above.   Set $\Lambda_N^*$ to be a connected subset of the graph $G_N$ on the diamond lattice $K_N$ $$\Lambda_N^*=\{(i,j)\mid i\in \ll1,N-1\rr,j\in \ll1,2N-2i+1\rr\}.$$ Let $\Lambda$ be a connected subset of $\Lambda_N^*$. Recall the $\hslg$ line ensemble $H_N$ from Definition \ref{hslg}. The law of $\{H_N^{(i)}(j) \mid (i,j)\in \Lambda\}$ conditioned on $\{H_N^{(i)}(j)\mid (i,j)\in \Lambda^c\}$ is a measure on $\R^{|\Lambda|}$ with density at $(u_{i,j})_{(i,j)\in \Lambda}$ proportional to
			\begin{align}\label{eq:gibbs}
				\prod_{\vec{e}=\{v_1\to v_2\}\in E(\Lambda\cup \partial\Lambda)} W_{\vec{e}}(u_{v_1}-u_{v_2}),
			\end{align}
   where $u_{i,j}=H_N^{(i)}(j)$ for $(i,j)\in \partial\Lambda$.
\end{theorem}

We refer to the above conditional law as the $\hslg$ Gibbs measure with boundary condition $\vec{u}=(u_{i,j})_{(i,j)\in \partial\Lambda}$ and denote this measure as $\pg^{\vec{u}}(\cdot)$. \bl{The above theorem follows directly from Theorem 1.3 and Lemma 4.2 in \cite{bcd23}. We remark that \cite{bcd23} specifies the Gibbs property for the centered line ensemble $\mathcal{L}_i^N(j):=H_N^{(i)}(j)+2\Psi(\theta)N$. However from the precise expression of the density in \eqref{eq:gibbs}, it is not hard to see that the Gibbs property remains unchanged upon a constant shift of the entire line ensemble (shown in Lemma 2.1(b) in \cite{bcd23}). Thus the same Gibbs property continues to hold for $H_N$.} 

\medskip

The $\hslg$ Gibbs measures satisfy stochastic monotonicity w.r.t.~the boundary data.
\begin{proposition}[Stochastic monotonicity, Proposition 2.6 in \cite{bcd23}]\label{p:gmc} Fix $k_1\le k_2$, $a_i\le b_i$ for $k_1\le i\le k_2$ and $\alpha>-\theta$. Let 
		\begin{align*}
			\Lambda:=\{(i,j)\mid k_1\le i\le k_2, a_i\le j\le b_i\}.
		\end{align*} 
		There exists a probability space consisting of a collection of random variables
		\begin{align*}
			\{L(v;(u_w)_{w\in \partial \Lambda}) \mid v\in \Lambda, (u_w)_{w\in\partial\Lambda} \in \R^{|\partial\Lambda|}\}
		\end{align*}
		such that
		\begin{enumerate}
			\item for each $(u_w)_{w\in\partial\Lambda}\in \R^{|\partial\Lambda|}$,  the law of $\{L(v;(u_w)_{w\in \partial \Lambda})\mid v\in \Lambda\}$ is given by the $\hslg$ Gibbs measure for the domain $\Lambda$ with boundary condition $(u_w)_{w\in\partial\Lambda}\in \R^{|\partial\Lambda|}$. 
			\item with probability $1$, for all $v\in \Lambda$ we have 
			$$L(v;(u_w)_{w\in \partial \Lambda}) \le L(v;(u_w')_{w\in \partial \Lambda}) \mbox{ whenever }u_w\le u_w' \mbox{ for all }w\in \partial\Lambda.$$
		\end{enumerate}
	\end{proposition}

As mentioned in the introduction, the $\hslg$ line ensemble enjoys a certain soft non-intersection property. This property is captured in our next theorem.

\begin{theorem}[Ordering of points, Theorem 3.1 in \cite{bcd23}]\label{t:order} Fix any $k\in \Z_{>0}$ and $\rho\in (0,1)$. Recall the $\hslg$ line ensemble $H_N$ from Definition \ref{hslg}. There exists $N_0=N_0(\rho,k)>k$ such that for all $N\ge N_0$, $i\in \ll 1,k\rr$ and $p\in \ll 1,N-i\rr$ the following inequalities holds:
		\begin{equation*}
			\begin{aligned}
			\Pr(H_N^{(i)}(2p+1)\le H_N^{(i)}(2p)+\log^2 N) & \ge 1-\rho^{N} , \\ \Pr(H_N^{(i)}(2p-1)\le H_N^{(i)}(2p)+\log^2 N) & \ge 1-\rho^{N}, \\
			\Pr(H_N^{(i+1)}(2p)\le H_N^{(i)}(2p+1)+\log^2 N) & \ge 1-\rho^{N} , \\ \Pr(H_N^{(i+1)}(2p)\le H_N^{(i)}(2p-1)+\log^2 N) & \ge 1-\rho^{N}.
		\end{aligned}
		\end{equation*}
\end{theorem}
We remark that the above theorem is true in the unbound phase as well (i.e., for $\alpha\ge 0$).
\medskip

We now introduce the \textit{interacting random walks} which are a specialized version of $\hslg$ Gibbs measures (see Figure \ref{fig:my_label}). 

\begin{definition}[Interacting random walk] \label{irw} 
We say $\big(L_1(\ll1,2T-2\rr),L_2(\ll1,2T-1\rr)\big)$ is an interacting random walk ($\irw$) of length $T$ with boundary condition $(a,b)\in \R^2$ if its law is a measure on $\R^{4T-3}$ with density at $(u_{1,j})_{j=1}^{2T-2}, (u_{2,j})_{j=1}^{2T-1}$ proportional to
\begin{align}\label{eq:denirw}
    \prod_{j=1}^{T-1} \exp(-e^{u_{2,2j}-u_{1,2j-1}}-e^{u_{2,2j}-u_{1,2j+1}}) \prod_{i=1}^2\prod_{j=1}^{2T-1} G_{\theta+(-1)^{i+j}\alpha}\big((-1)^{j+1}(u_{i,j}-u_{i,j+1})\big)
\end{align}
where $u_{1,2T-1}=a$, $u_{2,2T}=b$, and $u_{1,2T}=0$ and for $\beta>0$ $G_{\beta} :\R \to \R$ is defined by $$G_{\beta}(x)=[\Gamma(\beta)]^{-1}\exp(\beta x-e^x).$$  Note that $G_{\theta+\alpha}(u_{1,2T-1}-u_{1,2T})$ is a constant and can be absorbed into proportionality constant. 
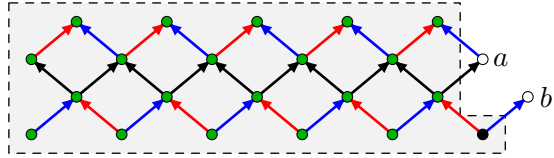
\begin{figure}[h!]
    \centering
    
    \begin{tikzpicture}[line cap=round,line join=round,>=triangle 45,x=1.2cm,y=1cm]
    \draw[fill=gray!10,line width=0.5pt,dashed] (-0.75,1.25)--(-0.75,-0.75)--(4.75,-0.75)--(4.75,-0.25)--(4.25,-0.25)--(4.25,1.25)--(-0.75,1.25);
        \foreach \x in {0,1,2,3,4}
        {\draw[line width=1pt,blue,{Latex[length=2mm]}-]  (\x,0) -- (\x-0.5,-0.5);
        \draw[line width=1pt,red,{Latex[length=2mm]}-]  (\x,1) -- (\x-0.5,0.5);
        \draw[line width=1pt,red,{Latex[length=2mm]}-]  (\x,0) -- (\x+0.5,-0.5);
        \draw[line width=1pt,blue,{Latex[length=2mm]}-]  (\x,1) -- (\x+0.5,0.5);
        \draw[line width=1pt,black,{Latex[length=2mm]}-]  (\x+0.5,0.5)--(\x,0);
        \draw[line width=1pt,black,{Latex[length=2mm]}-]  (\x-0.5,0.5)--(\x,0);}
        \draw[line width=1pt,blue,{Latex[length=2mm]}-]  (5,0) -- (4.5,-0.5);
        \draw[fill=white] (5,0) circle (2pt);
        \draw[fill=white] (4.5,0.5) circle (2pt);
        \node at (5.2,0) {$b$};
        \node at (4.7,0.5) {$a$};
        \draw[fill=black] (4.5,-0.5) circle (2pt);
        \foreach \x in {0,1,2,3,4}
        {
        \draw[fill=green!70!black] (\x-0.5,-0.5) circle (2pt);
        \draw[fill=green!70!black] (\x,0) circle (2pt);
        \draw[fill=green!70!black] (\x,1) circle (2pt);
        \draw[fill=green!70!black] (\x-0.5,0.5) circle (2pt);
        }
    \end{tikzpicture}
    \caption{$\irw$ of length $6$ with boundary condition $a$ and $b$.}
    \label{fig:my_label}
\end{figure}
\end{definition}
\bl{We now explain how the density in the definition in $\irw$ can be written in a form similar to \eqref{eq:gibbs}.
More specifically, for each $i\ge 1$, if we consider the vertex set
$$V_{i,T}:=\{(2i,j),(2i+1,j) \mid j\in \ll1,2T-1\rr\}\cup\{(2i+1,2T)\},$$
and the subgraph induced by $V_{i,T}, E(V_{i,T})$ (see Figure \ref{fig:my_label}), then the density in \eqref{eq:denirw} is same as
\begin{align*}
    \prod_{\vec{e}=\{v_1\to v_2\} \in E(V_{i,T})} W_{\vec{e}}(u_{v_1}-u_{v_2}),
\end{align*}
with the boundary condition $u_{2i,2T-1}=a$ and $u_{2i+1,2T}=b$ (the dummy variables in the above density are $(u_{2i,j})_{j=1}^{2T-2}, (u_{2i+1,j})_{j=1}^{2T-1}$ instead). The same density  can also be extracted in a similar manner from the subgraph induced by $\hat{V}_T,E(\hat{V}_T)$ where
$$\hat{V}_{T}:=\{(1,j),(2,j) \mid j\in \ll1,2T-1\rr\}\cup\{(2,2T)\},$$
provided we switch $\alpha$ to $-\alpha$ in \eqref{eq:color} (i.e., switching red and blue edges).} Since we have restricted $\alpha \in (-\theta,0)$ (bound phase), under this switching $\irw$ can be viewed as certain $\hslg$ Gibbs measures in the unbound phase. Indeed, after switching $\alpha$ to $-\alpha$, in the language of \cite{bcd23}, $\irw$ precisely corresponds to bottom-free measure on the domain $\mathcal{K}_{2, T}$ with boundary condition $(a,b)$ (see Definition 2.3 in \cite{bcd23}). This allows us to use the unbound phase estimates developed in \cite{bcd23}. We end this section by recording one such estimate.

\begin{proposition}[Lemma 5.3 in \cite{bcd23}]\label{p:irw} Fix any $T\ge 2$. Let $(L_1,L_2)$ be an $\irw$ of length $T$ with boundary condition $(0,-\sqrt{T})$.  Fix $\e\in (0,1)$. There exists $M_0=M_0(\e)>0$ such that
\begin{align*}
    \Pr\left(\sup_{p\in \ll1,2T-1\rr}|L_1(p)|+\sup_{q\in \ll1,2T\rr}|L_2(q)|\ge M_0\sqrt{T}\right) \le \e.
\end{align*}
\end{proposition}

\subsection{One-point fluctuations of point-to-(partial)line free energy}

In this section, we gather the point-to-(partial)line free energy fluctuation results from \cite{bw}. To state the theorem, we introduce a few necessary objects first.

\medskip

Recall the point-to-point half-space partition function $Z(m,n)$ from
\eqref{eq:hpp}. For $N\ge 1$, $m\in \ll0,N-1\rr$, we define the point-to-(partial)line half-space partition function as
\begin{align}\label{defzt}
    \zt(m)=\sum_{p=m}^{N-1} Z(N+p,N-p) = \sum_{p=m}^{N-1} e^{H_N^{(1)}(2p+1)}.
\end{align}
For the second equality, note that by \eqref{eq:rem} we have $H_N^{(1)}(2p+1) = \log Z(N+p, N-p)$ and thus we can translate the point-to-(partial)line partition function in Definition 1.8 (or equivalently in Definition 1.3) of \cite{bw} into sums of $e^{H_N^{(1)}(2p+1)}$ by way of the full-space point-to-point partition function $Z(n+p, n-p)$.

Recall the digamma function $\Psi(\cdot)$ from \eqref{psidef}. For any $k \in \Z_{>0}$, we set 
\begin{equation}
    \label{def:dk}
    \begin{aligned}
    R(\theta, \alpha) & := -\Psi(\theta+\alpha) -\Psi(\theta - \alpha),\\ \tau(\theta,\alpha) & : = \Psi(\theta - \alpha) - \Psi(\theta + \alpha), \\ \sigma^2(\theta,\alpha) & := \Psi'(\theta+\alpha)-\Psi'(\theta-\alpha), \\  \Delta_k (\theta, \alpha) & :=\Psi(\theta)-\tfrac12[\Psi(\theta+\alpha)+\Psi(\theta-\alpha)]-\tfrac1{2k}\log2.
\end{aligned}
\end{equation}
For the remainder of the paper, we will make use of the above notation repeatedly. As $\Psi$ is a strictly concave function, for all large enough $k$ (depending on $\alpha,\theta$) we have $\Delta_k>0$. For the results and proofs in the remainder of the paper, we always choose $k$ large enough such that $\Delta_k > 0.$

\medskip

We now state the necessary results from \cite{bw} about the point-to-(partial)line partition function $\zt(m)$ that we need in our subsequent analysis.   
\begin{theorem} \label{t:bw} \bl{Fix $\theta>0$ and $\alpha \in (-\theta,0)$}. Suppose $g: \mathbb{Z}_{>0} \to \mathbb{Z}_{>0}$ is a function that satisfies $\frac{g(N)}{N} \to 0$ as $N\to \infty$. We have  
\begin{align}\label{eq:bw}
    \frac1{N^{1/2}\sigma}\left[\log \zt(g(N))-RN+g(N)\tau\right]\stackrel{d}{\to} \mathcal{N}(0,1).
\end{align} 
where $R,\tau,\sigma$ are defined in \eqref{def:dk}. We have the following law of large numbers
 \begin{align}\label{lln}
         \frac{1}{N}\log \left[\sum_{p=1}^{N-1} \zh(N+p,N-p)\right] \stackrel{p}{\to} R \quad \frac{1}{N}\log \left[\sum_{p=1}^{N} \zh(N+p,N-p+1)\right] \stackrel{p}{\to} R. 
    \end{align}
    Furthermore, the above law of large numbers continues to hold when $\alpha=0$, i.e., the diagonal weights are assumed to be distributed as $\operatorname{Gamma}^{-1}(\theta)$. In that case $R(\theta,\alpha)$ is interpreted as $R(\theta,0)=-2\Psi(\theta)$.
\end{theorem}
\begin{proof} Theorem 1.10 in \cite{bw} discusses several fluctuation results for point-to-(partial)line partition function for the $\hslg$ polymer, including Theorem 1.10(3) which applies to the bound phase in this paper.  Letting $n = N-g(N)$, $m = N+g(N)$, \bl{and $\alpha=1$ in (1.12)  
of \cite{bw} and noting that $(\theta,\theta_0)$ in \cite{bw} corresponds to $(2\theta,\theta+\alpha)$ in our notation, we have
\begin{align*}
    m^{-1}(n\Psi'(\theta+\alpha)-m\Psi'(\theta-\alpha)) \to \Psi'(\theta+\alpha)-\Psi'(\theta-\alpha)
\end{align*}
Since $\Psi^{(2)}$ is positive, $\Psi'$ is strictly increasing and hence the above limit is positive for $\alpha\in (-\theta,0)$. Thus the hypothesis of Theorem 1.10(3) in \cite{bw} holds and we have}
\begin{align*}
    \frac1{(N-g(N))^{1/2}\sigma_p}\left[\log \zt(g(N))+(N-g(N))\mu_p\right]\stackrel{d}{\to} \mathcal{N}(0,1).
\end{align*}   
where $\mu_p:=\Psi(\theta+\alpha)+p\Psi(\theta-\alpha)$ and $\sigma_p^2:=\Psi'(\theta+\alpha)-p\Psi'(\theta-\alpha)$ with $p=\frac{N+g(N)}{N-g(N)}$. Observe that $(N-g(N))\mu_p = -RN+g(N)\tau$. As $g(N)/N\to 0$, we have that $$\frac{(N-g(N))^{1/2}\sigma_p}{N^{1/2}\sigma} \to 1.$$ Therefore the above fluctuation result implies \eqref{eq:bw}. For the law of large numbers, the first convergence in \eqref{lln} follows by taking $g(N)\equiv 1$ and appealing to \eqref{eq:bw}. The second law of large numbers follows in a similar manner by taking $n=N$ and $m=N+2$ in Theorem 1.10(3) in \cite{bw}.

\bl{When $\alpha=0$, we appeal to Theorem 1.10(2) in \cite{bw}. We take $p=\frac{N+1}{N-1}$. Let $\theta_c$ be the unique solution of $\Psi'(\theta_c)=p\Psi'(2\theta-\theta_c)$. By Taylor expansion, it follows that $\theta_c=\theta+O(N^{-1})$. Thus $$(N-1)^{1/3}(2\theta-\theta_c)\big(-\tfrac12\Psi^{(2)}(\theta_c)-\tfrac{p}2\Psi^{(2)}(2\theta-\theta_c)\big)^{1/3}\to 0.$$
Hence by Theorem 1.10(2) in \cite{bw} $(\log  \zt(1)+2\Psi'(\theta)N)/N^{1/3}$ converges weakly (to some distribution which is not relevant here). Thus, in particular, we have the first law of large numbers in \eqref{lln} for $\alpha=0$. The second law of large numbers in \eqref{lln} for $\alpha=0$ follows in a similar manner by taking $n=N,m=N+2$ in Theorem 1.10(2) in \cite{bw}.}
\end{proof}

\section{Controlling the average law of large numbers of the top curves}\label{sec:2k}


In this section, we control the average law of large numbers of the top $2k$ curves for large enough $k$ (Proposition \ref{p:alaw}). \sd{As explained in the introduction, the key idea behind this proposition is to show that the contribution of diagonal weights in the $2k$ many non-intersecting paths of $Z_{\operatorname{sym}}^{(2k)}(m, n)$ (defined in \eqref{e:nz}) essentially comes from $k$ many of them. The starting point of this idea is Lemma \ref{l:umap}.} Given a pair of non-intersecting paths $(\pi_1, \pi_2)$ starting and ending at adjacent locations with the same $x$-coordinate, Lemma \ref{l:umap} constructs two new non-intersecting paths $(\pi_1', \pi_2')$  from $(\pi_1, \pi_2)$ such that the new paths collectively carry the same weight variables but the diagonal weights only rest on the lower path. This combinatorial result proceeds to help us decompose the symmetrized multilayer partition function $Z_{\operatorname{sym}}^{(2k)}(m, n)$ into pairs of single-layer ones in Lemmas \ref{l:sbd} and \ref{l:log} before culminating into the final result in Proposition \ref{p:alaw}.

\medskip

Let $\Pi(\mathbf{u}_1\to \mathbf{u}_2,\mathbf{v}_1\to \mathbf{v}_2)$ denote the set of  pairs of non-intersecting upright paths in $\Z_{>0}^2$ starting from $\mathbf{u}_1, \mathbf{v}_1 \in \Z_{>0}^2$ and ending at $\mathbf{u}_2,\mathbf{v}_2 \in \Z_{>0}^2$ respectively. Recall that $\mathcal{I}^{-} = \left\{(i, j)\in \Z_{>0}^2 \mid j\le i\right\}$. Define $\mathcal{I}^{+}:=\left\{(i, j)\in \Z_{>0}^{2} \mid j\ge i \right\}$ which represents the half-space index set that includes points on and above the diagonal. The first lemma constructs the $U$ map.
\begin{lemma}[Construction of $U$ map]\label{l:umap}
Fix $x\in \Z_{>0}$ and any $(m,n)\in \mathcal{I}^-$ with $n\ge 2$. Then there exists a map $U: \Pi_1\to \Pi_2$ where
\begin{align*}
    \Pi_1 & :=\Pi((1,x+1)\to (m,n), (1,x)\to (m,n-1)) \\  \Pi_2 & :=\Pi((1,x+1)\to (n-1,m), (1,x)\to (n,m)),
\end{align*}
such that the following properties hold (let $(\pi_1', \pi_2'): = U(\pi_1,\pi_2)$):

\medskip

\begin{enumerate}[label=(\alph*),leftmargin=18pt]
\setlength\itemsep{0.5em}
    \item \label{parta} $\pi_1'$ has no diagonal points, i.e., $\{(i,i)\in \Z_{>0}^2\} \cap \pi_1'$ is empty and  $$\{(i,i)\in \Z_{>0}^2\} \cap \pi_2'= \{(i,i)\in \Z_{>0}^2\} \cap \{\pi_1\cup \pi_2\}.$$
    \item \label{partb} Recall the symmetrized weights $(\til{W}_{i,j})_{(i,j)\in \Z_{>0}^2}$ from \eqref{eq:symwt}. We have
$$\prod_{(i,j)\in \pi_1\cup \pi_2} \til{W}_{i,j} \overset{a.s.}{=}\prod_{(i,j)\in \pi_1'\cup \pi_2'} \til{W}_{i,j}.$$
\item \label{partc} For each $(\pi_1',\pi_2')\in \Pi_2$ we have
$$\left|U^{-1}(\{(\pi_1',\pi_2')\})\right| \le 2^{|\{(i,i) \in \pi_1\cup \pi_2\}|}.$$
\end{enumerate}
\end{lemma}

\begin{figure}[h!]
    \centering
    \begin{subfigure}[b]{0.45\textwidth}
        \centering
          \begin{tikzpicture}[line cap=round,line join=round,>=triangle 45,x=0.2cm,y=0.2cm]
\draw [color=gray!10!white,, xstep=0.2cm,ystep=0.2cm] (0,3) grid (31,31);
\draw [line width=1pt] (0,3)-- (1,3);
\draw [line width=1pt] (1,3)-- (1,4);
\draw [line width=1pt] (1,4)-- (3,4);
\draw [line width=1pt] (3,4)-- (3,5);
\draw [line width=1pt] (3,5)-- (6,5);
\draw [line width=1pt] (6,5)-- (6,7);
\draw [line width=1pt] (6,7)-- (10,7);
\draw [line width=1pt,color=blue] (0,4)-- (0,6);
\draw [line width=1pt,color=blue] (0,6)-- (4,6);
\draw [line width=1pt,color=blue] (4,6)-- (4,9);
\draw [line width=1pt,color=blue] (4,9)-- (9,9);
\draw [line width=1pt,color=blue] (9,9)-- (9,10);
\draw [line width=1pt,color=blue] (9,10)-- (12,10);
\draw [line width=1pt,color=blue] (12,10)-- (12,13);
\draw [line width=1pt,color=blue] (12,13)-- (14,13);
\draw [line width=1pt,color=blue] (14,16)-- (14,13);
\draw [line width=1pt] (10,7)-- (16,7);
\draw [line width=1pt] (16,7)-- (16,13);
\draw [line width=1pt] (16,13)-- (18,13);
\draw [line width=1pt] (18,13)-- (18,15);
\draw [line width=1pt,color=blue] (14,16)-- (14,21);
\draw [line width=1pt,color=blue] (14,21)-- (19,21);
\draw [line width=1pt] (18,15)-- (18,20);
\draw [line width=1pt] (18,20)-- (21,20);
\draw [line width=1pt,color=blue] (19,21)-- (23,21);
\draw [line width=1pt,color=blue] (23,21)-- (23,25);
\draw [line width=1pt,color=blue] (23,25)-- (31,25);
\draw [line width=1pt] (21,20)-- (24,20);
\draw [line width=1pt] (24,20)-- (24,21);
\draw [line width=1pt] (24,21)-- (25,21);
\draw [line width=1pt] (25,21)-- (25,22);
\draw [line width=1pt] (25,22)-- (28,22);
\draw [line width=1pt] (28,22)-- (28,24);
\draw [line width=1pt] (28,24)-- (31,24);
\draw [line width=1pt] (3,3)-- (31,31);
\begin{scriptsize}
\end{scriptsize}
\end{tikzpicture}
\caption{}
    \end{subfigure}
\begin{subfigure}[b]{0.45\textwidth}
     \centering
\begin{tikzpicture}[line cap=round,line join=round,>=triangle 45,x=0.2cm,y=0.2cm]
\draw [color=gray!10!white,, xstep=0.2cm,ystep=0.2cm] (0,3) grid (31,31);
\draw [line width=1pt] (0,3)-- (1,3);
\draw [line width=1pt] (1,3)-- (1,4);
\draw [line width=1pt] (1,4)-- (3,4);
\draw [line width=1pt] (3,4)-- (3,5);
\draw [line width=1pt] (3,5)-- (6,5);
\draw [line width=1pt] (6,5)-- (6,7);
\draw [line width=1pt] (6,7)-- (9,7);
\draw [line width=1pt,color=blue] (0,4)-- (0,6);
\draw [line width=1pt,color=blue] (0,6)-- (4,6);
\draw [line width=1pt,color=blue] (4,6)-- (4,9);
\draw [line width=1pt,color=blue] (4,9)-- (7,9);
\draw [line width=1pt,color=blue] (14,18)-- (14,21);
\draw [line width=1pt,color=blue] (14,21)-- (20,21);
\draw [line width=1pt] (18,14)-- (18,20);
\draw [line width=1pt] (18,20)-- (21,20);
\draw [line width=1pt,dashed] (7,9)-- (7,16);
\draw [line width=1pt,dashed] (14,18)-- (13,18);
\draw [line width=1pt,dashed] (13,18)-- (13,16);
\draw [line width=1pt,dashed] (13,16)-- (7,16);
\draw [line width=1pt,dashed] (20,21)-- (20,24);
\draw [line width=1pt,dashed] (20,24)-- (21,24);
\draw [line width=1pt,dashed] (21,24)-- (21,25);
\draw [line width=1pt,dashed] (21,25)-- (22,25);
\draw [line width=1pt,dashed] (22,25)-- (22,28);
\draw [line width=1pt,dashed] (22,28)-- (24,28);
\draw [line width=1pt,dashed] (24,28)-- (24,31);
\draw [line width=1pt,dashed,color=blue] (9,9)-- (9,7);
\draw [line width=1pt,dashed,color=blue] (9,9)-- (10,9);
\draw [line width=1pt,dashed,color=blue] (10,9)-- (10,12);
\draw [line width=1pt,dashed,color=blue] (10,12)-- (13,12);
\draw [line width=1pt,dashed,color=blue] (13,12)-- (13,14);
\draw [line width=1pt,dashed,color=blue] (13,14)-- (16,14);
\draw [line width=1pt,dashed,color=blue] (16,14)-- (18,14);
\draw [line width=1pt,dashed,color=blue] (21,20)-- (21,23);
\draw [line width=1pt,dashed,color=blue] (21,23)-- (25,23);
\draw [line width=1pt,dashed,color=blue] (25,23)-- (25,25);
\draw [line width=1pt,dashed,color=blue] (25,25)-- (25,28);
\draw [line width=1pt,dashed,color=blue] (25,28)-- (25,31);
\draw [line width=1pt] (3,3)-- (31,31);
\begin{scriptsize}
\end{scriptsize}
\end{tikzpicture}
\caption{}
     \end{subfigure}
\caption{The $U$ map takes (A) to (B).}
    \label{fig:c}
\end{figure}
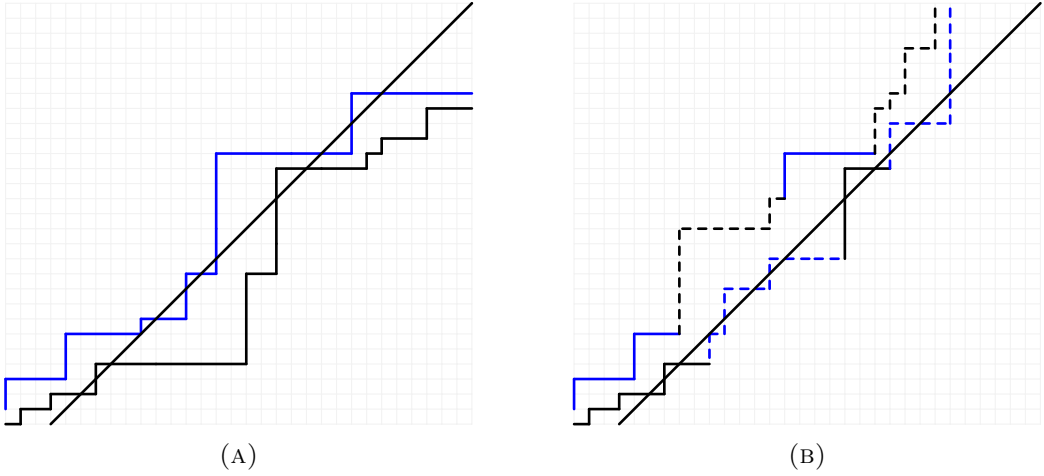

\sd{We remark that Lemma \ref{l:umap} is entirely combinatorial and does not use any results about the integrability of the model. Lemma \ref{l:umap} continues to hold for any collection of symmetrized weights that are not necessarily distributed as inverse-Gamma random variables.}

\medskip

\begin{proof} We begin with several necessary definitions. We define a partial order $<$ on the points $\Z_{>0}^2$ by requiring $P_1=(a_1,b_1)<P_2=(a_2,b_2)$ whenever $a_1+b_1<a_2+b_2$. \bl{Suppose $\pi$ is an upright path from  $\mathbf{u}\in \Z_{>0}^2$ to $\mathbf{v}\in \Z_{>0}^2$ passing through two points $\mathbf{w}_1 \in \Z_{>0}^2$ and $\mathbf{w}_2 \in \Z_{>0}^2$. Suppose $\mathbf{u}<\mathbf{w}_1< \mathbf{w}_2<\mathbf{v}$. The \textit{portion} of the path $\pi$ from $\mathbf{w}_1$ to $\mathbf{w}_2$ is defined to be the unique upright path $\pi_*$ from $\mathbf{w}_1$ to $\mathbf{w}_2$ that overlaps completely with $\pi$, i.e., $\pi_*\cap \pi=\pi_*$.}

\medskip

Fix $x\in \Z_{>0}$ and any $(m,n)\in \mathcal{I}^-$ with $n\ge 2$. Let $\pi_1$ denote the path from $(1, x+1)$ to $(m, n)$ and $\pi_2$ the path from $(1, x)$ to $(m, n-1)$. We denote $\diag(\pi_i)$ as the set of points on $\pi_i$ that lie on the diagonal set $D:=\{(i,i)\in \Z_{>0}^2\}$. Recall that $\mathcal{I}^+ =\{ (i,j)\in \Z_{>0}^2 \mid i\le j\}$ and $\mathcal{I}^- =\{ (i,j)\in \Z_{>0}^2 \mid j\le i\}$.  We define a special collection of points, $\m{SPDiag}$ from $\diag(\pi_2)$. Let $D_1<D_2<D_3<\cdots<D_s$ be all the points in $\diag(\pi_1\cup \pi_2)$ arranged in the  increasing order. We put the point $D_j \in \diag(\pi_2)$ in the set $\m{SPDiag}$ if $D_{j-1} \in \diag(\pi_1)$ or $D_{j+1}\in \diag(\pi_1)$. In other words, $\m{SPDiag}$ consists of the diagonal points in $\pi_2$ that bookend contiguous clusters of $\diag(\pi_1)$ in $\diag(\pi_1 \cup \pi_2).$ We enumerate the points in $\m{SPDiag}$ as $A_1<A_2<\cdots<A_r$. Let $B_j$ be the first point on $\pi_1$ that has the same $x$-coordinate as $A_j$. Note that by construction, either only $\pi_1$ intersects the diagonal or only $\pi_2$ intersects the diagonal between $A_j$ and $A_{j+1}$, $j = 1, \ldots, r$. Let us denote $A_{r+1}:=(m,n-1)$ and $B_{r+1}:=(m,n)$.

\medskip

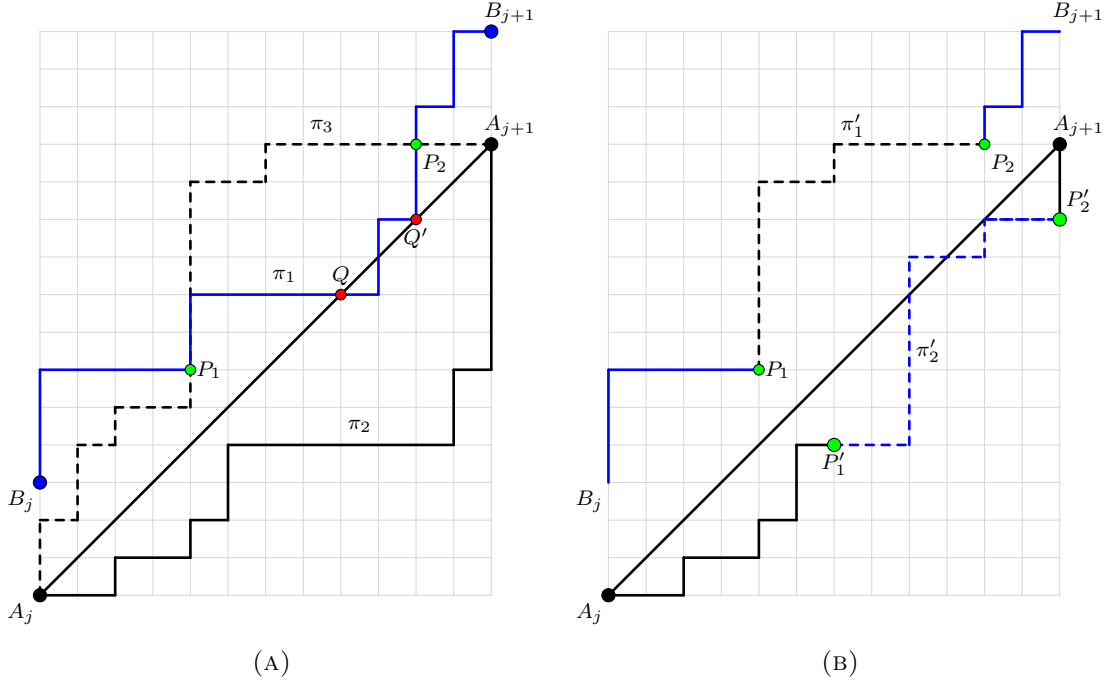
\begin{figure}[h!]
    \centering
    \begin{subfigure}[b]{0.45\textwidth}
        \centering
          \begin{tikzpicture}[line cap=round,line join=round,>=triangle 45,x=0.5cm,y=0.5cm]
\draw [color=gray!30!white,, xstep=0.5cm,ystep=0.5cm] (1,1) grid (13,16);
\draw [line width=1pt] (1,1)-- (3,1);
\draw [line width=1pt] (3,2)-- (3,1);
\draw [line width=1pt] (3,2)-- (5,2);
\draw [line width=1pt] (5,2)-- (5,3);
\draw [line width=1pt] (5,3)-- (6,3);
\draw [line width=1pt] (6,3)-- (6,5);
\draw [line width=1pt] (6,5)-- (12,5);
\draw [line width=1pt] (12,5)-- (12,7);
\draw [line width=1pt] (12,7)-- (13,7);
\draw [line width=1pt] (13,7)-- (13,13);
\draw [line width=1pt] (1,1)-- (13,13);
\draw [line width=1pt,dashed] (1,1)-- (1,3);
\draw [line width=1pt,dashed] (1,3)-- (2,3);
\draw [line width=1pt,dashed] (2,3)-- (2,5);
\draw [line width=1pt,dashed] (2,5)-- (3,5);
\draw [line width=1pt,dashed] (3,5)-- (3,6);
\draw [line width=1pt,dashed] (3,6)-- (5,6);
\draw [line width=1pt,dashed] (5,6)-- (5,12);
\draw [line width=1pt,dashed] (13,13)-- (7,13);
\draw [line width=1pt,dashed] (7,13)-- (7,12);
\draw [line width=1pt,dashed] (5,12)-- (7,12);
\draw [line width=1pt,color=blue] (1,4)-- (1,7);
\draw [line width=1pt,color=blue] (1,7)-- (5,7);
\draw [line width=1pt,color=blue] (5,7)-- (5,9);
\draw [line width=1pt,color=blue] (5,9)-- (10,9);
\draw [line width=1pt,color=blue] (10,9)-- (10,11);
\draw [line width=1pt,color=blue] (10,11)-- (11,11);
\draw [line width=1pt,color=blue] (11,11)-- (11,13);
\draw [line width=1pt,color=blue] (11,13)-- (11,14);
\draw [line width=1pt,color=blue] (11,14)-- (12,14);
\draw [line width=1pt,color=blue] (12,14)-- (12,16);
\draw [line width=1pt,color=blue] (12,16)-- (13,16);
\begin{scriptsize}
\draw [fill=black] (1,1) circle (2.5pt);
\draw [fill=blue] (1,4) circle (2.5pt);
\draw [fill=blue] (13,16) circle (2.5pt);
\draw [fill=black] (13,13) circle (2.5pt);
\draw [fill=green] (5,7) circle (2pt);
\draw [fill=green] (11,13) circle (2pt);
\draw [fill=red] (9,9) circle (2pt);
\draw [fill=red] (11,11) circle(2pt);
\node at (0.5,0.5) {$A_j$};
\node at (0.5,3.5) {$B_j$};
\node at (13.5,13.5) {$A_{j+1}$};
\node at (13.5,16.5) {$B_{j+1}$};
\node at (9,9.5) {$Q$};
\node at (11,10.5) {$Q'$};
\node at (5.5,7) {$P_1$};
\node at (11.5,12.5) {$P_2$};
\node at (8.5,13.5) {$\pi_3$};
\node at (7.5,9.5) {$\pi_1$};
\node at (9.5,5.5) {$\pi_2$};
\end{scriptsize}
\end{tikzpicture}
\caption{}
    \end{subfigure}
\begin{subfigure}[b]{0.45\textwidth}
     \centering
     \begin{tikzpicture}[line cap=round,line join=round,>=triangle 45,x=0.5cm,y=0.5cm]
\draw [color=gray!30!white,, xstep=0.5cm,ystep=0.5cm] (1,1) grid (13,16);
\draw [line width=1pt] (1,1)-- (3,1);
\draw [line width=1pt] (3,2)-- (3,1);
\draw [line width=1pt] (3,2)-- (5,2);
\draw [line width=1pt] (5,2)-- (5,3);
\draw [line width=1pt] (5,3)-- (6,3);
\draw [line width=1pt] (6,3)-- (6,5);
\draw [line width=1pt] (6,5)-- (7,5);
\draw [line width=1pt] (13,11)-- (13,13);
\draw [line width=1pt] (1,1)-- (13,13);
\draw [line width=1pt,dashed] (5,7)-- (5,12);
\draw [line width=1pt,dashed] (11,13)-- (7,13);
\draw [line width=1pt,dashed] (7,13)-- (7,12);
\draw [line width=1pt,dashed] (5,12)-- (7,12);
\draw [line width=1pt,color=blue] (1,4)-- (1,7);
\draw [line width=1pt,color=blue] (1,7)-- (5,7);
\draw [line width=1pt,color=blue] (11,13)-- (11,14);
\draw [line width=1pt,color=blue] (11,14)-- (12,14);
\draw [line width=1pt,color=blue] (12,14)-- (12,16);
\draw [line width=1pt,color=blue] (12,16)-- (13,16);
\draw [line width=1pt,dashed,color=blue] (7,5)-- (9,5);
\draw [line width=1pt,dashed,color=blue] (9,5)-- (9,10);
\draw [line width=1pt,dashed,color=blue] (9,10)-- (11,10);
\draw [line width=1pt,dashed,color=blue] (11,10)-- (11,11);
\draw [line width=1pt,dashed,color=blue] (11,11)-- (13,11);
\draw [line width=1pt,dashed,color=blue] (13,11)-- (11,11);
\begin{scriptsize}
\draw [fill=black] (1,1) circle (2.5pt);
\draw [fill=black] (13,13) circle (2.5pt);
\draw [fill=green] (7,5) circle (2.5pt);
\draw [fill=green] (5,7) circle (2pt);
\draw [fill=green] (11,13) circle (2pt);
\draw [fill=green] (13,11) circle (2.5pt);
\node at (0.5,0.5) {$A_j$};
\node at (0.5,3.5) {$B_j$};
\node at (13.5,13.5) {$A_{j+1}$};
\node at (13.5,16.5) {$B_{j+1}$};
\node at (5.5,7) {$P_1$};
\node at (11.5,12.5) {$P_2$};
\node at (7,4.5) {$P_1'$};
\node at (13.5,11.5) {$P_2'$};
\node at (9.5,7.5) {$\pi_2'$};
\node at (7.5,13.5) {$\pi_1'$};
\end{scriptsize}
\end{tikzpicture}
\caption{}
     \end{subfigure}
\caption{\sd{The second case when $j \le r-1$ and only $\pi_1$ intersects with the diagonal. $\pi_1$ and $\pi_2$ are the black and blue paths in Figure (A) respectively. $\pi_3$ is the black dashed path in Figure (A). $\pi_1'$ is the path in Figure (B) which is formed by the concatenation of solid blue paths and the black dashed path. $\pi_2'$ is the path in Figure (B) which is formed by the concatenation of solid black paths and the blue dashed path. The $U$ map takes $\pi_1,\pi_2$ and returns $\pi_1',\pi_2'$.}}
    \label{fig:c2}
\end{figure}
We now construct new paths $\pi_2'$ and $\pi_1'$ from $\pi_2$ and $\pi_1$ by reconstructing each segment between $A_j$ and $A_{j+1}$ for $\pi_2$ (and $B_j$ and $B_{j+1}$ for $\pi_1$ respectively), $j = 1,\ldots, r$. We separate the reconstruction procedures for each segment into the following cases: if only $\pi_2$ intersects the diagonal and $j\le r-1$, if only $\pi_1$ intersects the diagonal and $j \le r-1$, or if $j=r.$

\medskip

\begin{enumerate}[leftmargin=18pt]
\setlength\itemsep{0.6em}
    \item \textbf{When $1\le j\le r-1$ and only $\pi_2$ intersects the diagonal between $A_j$ and $A_{j+1}$}, we keep the original paths. We set $\pi_1'$ and $\pi_2'$ on these segments to be the same as those on $\pi_1$ and $\pi_2$ respectively.
    \item \textbf{When $1\le j\le r-1$ and only $\pi_1$ intersects the diagonal between $A_j$ and $A_{j+1}$ (see Figure \ref{fig:c2})}, the portion of the path $\pi_2$ from $A_j$ to $A_{j+1}$ (excluding $A_j$ and $A_{j+1}$) lies in $\mathcal{I}^{-}\setminus D$. Reflecting the portion of the path $\pi_2$ from $A_j$ to $A_{j+1}$ (black path in Figure \ref{fig:c2}) across the diagonal yields a path $\pi_3$ (black dashed path in Figure \ref{fig:c2}). Let $Q$ be the first point on $\diag(\pi_1)$ that lies between $A_j$ and $A_{j+1}$ and $Q'$ be the last, which exist by construction of the $\m{SPdiag}$ set ($Q$ and $Q'$ may overlap). As the $y$-coordinate of $A_i$ is strictly smaller than that of $B_i$ and $Q, Q'$ are on the $\diag(\pi_1)$, $\pi_1$ and $\pi_3$ must intersect on the segments between $A_{j}$ and $Q$ and $Q'$ and $A_{j+1}$. Let $P_1$ be the first point of intersection and $P_2$ the last point of intersection. Clearly $P_1\neq P_2$ as the former is between $A_j$ and $Q$ and the latter lies between $Q'$ and $A_{j+1}$. Replacing the portion of $\pi_1$ between $P_1$ and $P_2$ with that of $\pi_3$ yields a path $\pi_1'$ from $B_j$ to $B_{j+1}$. As the part of $\pi_3$ between $P_1$ and $P_2$ lies in $\mathcal{I}^+ \setminus D$, $\pi_1'$ lies entirely in $\mathcal{I}^+ \setminus D$. We denote the reflections of $P_1$ and $P_2$ across the diagonal as $P_1'$ and $P_2'$, which must lie on the original $\pi_2$ by construction. Similarly replacing the portion of $\pi_2$ between $P_1'$ and $P_2'$ with the reflection of $\pi_1$ between $P_1$ and $P_2$ across the diagonal yields a path $\pi_2'$ from $A_j$ to $A_{j+1}$. As $\pi_1$ and $\pi_2$ are non-intersecting, the reflected paths are also non-intersecting. Thus the new paths $\pi_1'$ and $\pi_2'$ are non-intersecting.
    \item  \textbf{When $j=r$,} consider the portion of the path $\pi_2$ from $A_r$ to $A_{r+1}$ (see Figure \ref{fig:c3}). Note that in this segment, all the diagonal points belong to $\pi_1$. Reflecting this portion of $\pi_2$ across the diagonal gives us $\pi_3$ (black dashed path in Figure \ref{fig:c3}).  Let $Q$ be the first point on $\diag(\pi_1)$ that lies between $A_j$ and $A_{j+1}$ and $Q$ exists as $\pi_1$ ends at $B_{r+1}:=(m,n)\in \mathcal{I}^-$. Note that  $\pi_3$ lies entirely in $\mathcal{I}^+\setminus D$, excluding $A_r$. Thus $\pi_1$ and $\pi_3$ necessarily intersect in $\mathcal{I}^{+}\setminus D$. Again, we locate the first point intersection $P$ and replace the portion of $\pi_1$ from $P$ to $B_{r+1}$ with the portion of $\pi_3$ from $P$ to $A_{r+1}':=(n-1,m)$. Similarly, reflecting the portion of $\pi_1$ from $P$ and $B_{r+1}$ across the diagonal and replacing the portions of $\pi_2$ between $P'$ and $A_{r+1}$ with the portion of reflection between $P$ and $B_{r+1}':=(n,m)$ yields a path $\pi_2'$ from $A_r$ to $B'_{r+1}$. Clearly, the new path $\pi_1'$ lies in $\mathcal{I}^{+}\setminus D$ and the paths $\pi_1'$ and $\pi_2'$ are non-intersecting as the reflected portions are non-intersecting.
\end{enumerate}

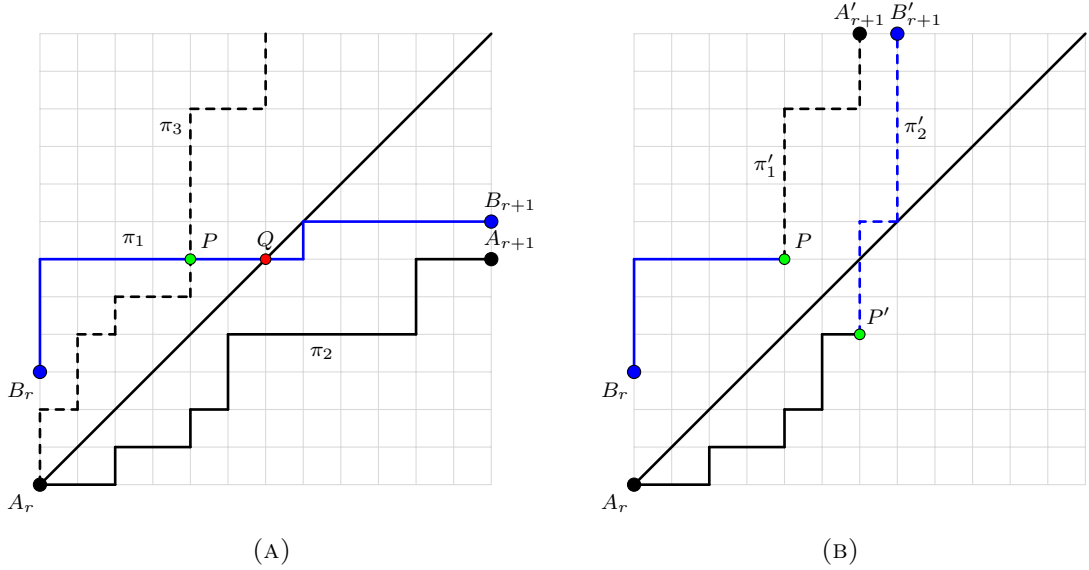
\begin{figure}[h!]
    \centering
    \begin{subfigure}[b]{0.45\textwidth}
        \centering
          \begin{tikzpicture}[line cap=round,line join=round,>=triangle 45,x=0.5cm,y=0.5cm]
\draw [color=gray!30!white,, xstep=0.5cm,ystep=0.5cm] (1,1) grid (13,13);
\draw [line width=1pt] (1,1)-- (3,1);
\draw [line width=1pt] (3,2)-- (3,1);
\draw [line width=1pt] (3,2)-- (5,2);
\draw [line width=1pt] (5,2)-- (5,3);
\draw [line width=1pt] (5,3)-- (6,3);
\draw [line width=1pt] (6,3)-- (6,5);
\draw [line width=1pt] (6,5)-- (11,5);
\draw [line width=1pt] (11,5)-- (11,7);
\draw [line width=1pt] (11,7)-- (13,7);
\draw [line width=1pt] (1,1)-- (13,13);
\draw [line width=1pt,dashed] (1,1)-- (1,3);
\draw [line width=1pt,dashed] (1,3)-- (2,3);
\draw [line width=1pt,dashed] (2,3)-- (2,5);
\draw [line width=1pt,dashed] (2,5)-- (3,5);
\draw [line width=1pt,dashed] (3,5)-- (3,6);
\draw [line width=1pt,dashed] (3,6)-- (5,6);
\draw [line width=1pt,dashed] (5,6)-- (5,11);
\draw [line width=1pt,dashed] (7,13)-- (7,11);
\draw [line width=1pt,dashed] (5,11)-- (7,11);
\draw [line width=1pt,color=blue] (1,4)-- (1,7);
\draw [line width=1pt,color=blue] (1,7)-- (8,7);
\draw [line width=1pt,color=blue] (8,7)-- (8,8);
\draw [line width=1pt,color=blue] (8,8)-- (13,8);
\begin{scriptsize}
\draw [fill=black] (1,1) circle (2.5pt);
\draw [fill=blue] (1,4) circle (2.5pt);
\draw [fill=blue] (13,8) circle (2.5pt);
\draw [fill=black] (13,7) circle (2.5pt);
\draw [fill=green] (5,7) circle (2pt);
\draw [fill=red] (7,7) circle (2pt);
\node at (0.5,0.5) {$A_{r}$};
\node at (0.5,3.5) {$B_{r}$};
\node at (13.5,7.5) {$A_{r+1}$};
\node at (13.5,8.5) {$B_{r+1}$};
\node at (7,7.5) {$Q$};
\node at (5.5,7.5) {$P$};
\node at (3.5,7.5) {$\pi_1$};
\node at (8.5,4.5) {$\pi_2$};
\node at (4.5,10.5) {$\pi_3$};
\end{scriptsize}
\end{tikzpicture}
\caption{}
    \end{subfigure}
\begin{subfigure}[b]{0.45\textwidth}
     \centering
     \begin{tikzpicture}[line cap=round,line join=round,>=triangle 45,x=0.5cm,y=0.5cm]
\draw [color=gray!30!white,, xstep=0.5cm,ystep=0.5cm] (1,1) grid (13,13);
\draw [line width=1pt] (1,1)-- (3,1);
\draw [line width=1pt] (3,2)-- (3,1);
\draw [line width=1pt] (3,2)-- (5,2);
\draw [line width=1pt] (5,2)-- (5,3);
\draw [line width=1pt] (5,3)-- (6,3);
\draw [line width=1pt] (6,3)-- (6,5);
\draw [line width=1pt] (6,5)-- (7,5);
\draw [line width=1pt] (1,1)-- (13,13);
\draw [line width=1pt,dashed] (5,7)-- (5,11);
\draw [line width=1pt,dashed] (7,13)-- (7,11);
\draw [line width=1pt,dashed] (5,11)-- (7,11);
\draw [line width=1pt,color=blue] (1,4)-- (1,7);
\draw [line width=1pt,color=blue] (1,7)-- (5,7);
\draw [line width=1pt,dashed,color=blue] (7,5)-- (7,8);
\draw [line width=1pt,dashed,color=blue] (8,8)-- (7,8);
\draw [line width=1pt,dashed,color=blue] (8,8)-- (8,13);
\begin{scriptsize}
\draw [fill=black] (1,1) circle (2.5pt);
\draw [fill=blue] (1,4) circle (2.5pt);
\draw [fill=blue] (8,13) circle (2.5pt);
\draw [fill=black] (7,13) circle (2.5pt);
\draw [fill=green] (5,7) circle (2pt);
\node at (0.5,0.5) {$A_{r}$};
\node at (0.5,3.5) {$B_{r}$};
\node at (7,13.5) {$A_{r+1}'$};
\node at (8.5,13.5) {$B_{r+1}'$};
\node at (5.5,7.5) {$P$};
\draw [fill=green] (7,5) circle (2pt);
\node at (7.5,5.5) {$P'$};
\node at (8.5,10.5) {$\pi_2'$};
\node at (4.5,9.5) {$\pi_1'$};
\end{scriptsize}
\end{tikzpicture}
\caption{}
     \end{subfigure}
\caption{\sd{The $j = r$ case. $\pi_1$ and $\pi_2$ are the black and blue paths in Figure (A) respectively. $\pi_3$ is the black dashed path in Figure (A). $\pi_1'$ is the path in Figure (B) which is formed by the concatenation of the solid blue path and the black dashed path. $\pi_2'$ is the path in Figure (B) which is formed by the concatenation of the solid black path and the blue dashed path. The $U$ map takes $\pi_1,\pi_2$ and spits out $\pi_1',\pi_2'$.}}
    \label{fig:c3}
\end{figure}

 As $A_j$ and $B_j$ remain unchanged for $j\le r$, connecting all the segments between $A_{j}$'s (and $B_j$'s respectively) for $j\le r$ and $A_r$ and $B'_{r+1}$ 
 (and $B_r$ and $A'_{r+1}$) yields the new path $\pi_2'$ from $(1, x)$ to $(n, m)$ and the new path $\pi_1'$ from $(1, x+1)$ to $(n-1, m)$ (see Figure \ref{fig:c}). At each step of the above construction, the paths remain non-intersecting. Thus $(\pi_1', \pi_2')$ form a non-intersecting pair. We call this explicitly constructed map $U$. By construction, $\pi_1'$ lies entirely in $\mathcal{I}^+\setminus D$ and has no diagonal points. This proves \ref{parta}. Since the construction involves only exchanges of reflected portions, due to the symmetry of the weights $\til{W}_{i,j}$ across the diagonal, we have \ref{partb}. \bl{Finally to verify \ref{partc}, suppose $(\pi_1',\pi_2')$ lies in the image of $U$. Note that $\diag(\pi_1'\cup \pi_2')=\diag(\pi_1\cup \pi_2)$ for every $(\pi_1,\pi_2)$ in the pre-image of $(\pi_1',\pi_2')$. Given $\diag(\pi_1'\cup \pi_2')$, there are $2^{|\diag(\pi_1'\cup \pi_2')|}$ ways to partition $\diag(\pi_1'\cup \pi_2')$ set into two sets -- one would be a candidate for $\diag(\pi_1)$ and the other for $\diag(\pi_2)$.} As $\diag(\pi_1)$ and $\diag(\pi_2)$ uniquely determine $\m{SPDiag}$ where reflections are performed, reverting the same operations on $\pi_1'$ and $\pi_2'$ between consecutive points in $\m{SPDiag}$ leads to original $\pi_1$ and $\pi_2$. Thus the map has at most $2^{|\diag(\pi_1'\cup \pi_2')|}$ pre-images for $(\pi_1',\pi_2')$, which completes the proof.
\end{proof}

\bigskip

\sd{Note that the $U$ map in Lemma \ref{l:umap} gives us a path that does not contain any diagonal vertex. To capture the contribution of this path, we now introduce the \textit{diagonal-avoiding symmetrized} partition function.} Let $\til{\Pi}_{m,n}^{(1)}$ be the collection of all upright paths from $(1,1)$ to $(m,n)\in \Z_{>0}^2 \setminus \{(i,i)\in \Z_{>0}^2\}$ that do not touch the diagonal after $(1,1)$. Set
\begin{align}\label{def:sbd}
    \til{Z}_{\operatorname{sym}}(m,n):=\sum_{\pi \in \til{\Pi}_{m,n}^{(1)}} \prod_{(i,j)\in \pi}  \til{W}_{i,j}, \qquad \til{V}_{q} := \sum_{(i,j) \mid i+j=q} \til{Z}_{\operatorname{sym}}(i,j) \quad\bl{\mbox{ for }q\in \Z_{>0}}
\end{align} where $\til{W}_{i, j}$ is defined in \eqref{eq:symwt}. \bl{Note that, if we ignore the starting point, the paths in $\til{\Pi}_{m,n}^{(1)}$ lies entirely on one side of the diagonal (either in $\mathcal{I}^+\setminus D$ or $\mathcal{I}^{-}\setminus D$)}. 
We call $\til{Z}_{\operatorname{sym}}(m,n)$ the \textit{diagonal-avoiding symmetrized} partition function. Let us recall $Z_{\operatorname{sym}}(m,n)$ from \eqref{e:nz} and we similarly define \begin{align}\label{defvq}
	   V_{q} := \sum_{(i,j) \mid i+j=q} Z_{\operatorname{sym}}(i,j) \quad \bl{\mbox{ for }q\in \Z_{>0}}.
	\end{align} 
 The next lemma establishes a relation between $  Z_{\operatorname{sym}}^{(2k)}(m,n)$, $V_{m+n}$ and  $\til{V}_{m+n}$.
\begin{lemma}\label{l:sbd} For all $(m,n)\in \mathcal{I}^-$ and $k\in \ll 1,n/2\rr$, almost surely we have
\begin{align}\label{eq:sbd}
    Z_{\operatorname{sym}}^{(2k)}(m,n) \le 2^n \cdot  \prod_{i=2}^{2k} \prod_{j=1}^{i-1} (\til{W}_{1,j})^{-1}\cdot \prod_{i=1}^k \left[V_{m+n+2-2i}\til{V}_{m+n+1-2i}\right]
\end{align}
where $ Z_{\operatorname{sym}}^{(i)}(m,n) $, $V_{m+n+2-2i}$ and $\til{V}_{m+n+1-2i}$ are defined in \eqref{e:nz}, \eqref{defvq} and \eqref{def:sbd} respectively. 
\end{lemma}

\begin{proof}  We extend  our definition of $U$ map from Lemma \ref{l:umap} to the domain $\Pi_{m,n}^{(2k)}$ (recall its definition from the discussion around \eqref{e:nz}) by defining
$$U(\pi_1,\ldots,\pi_{2k}):=(U(\pi_1,\pi_2),\ldots,U(\pi_{2k-1},\pi_{2k}))$$ 
for $(\pi_1,\ldots,\pi_{2k}) \in \Pi_{m,n}^{(2k)}$. \bl{The $2k$ paths in $U(\pi_1,\ldots,\pi_{2k})$ are pairwise non-intersecting but may not be non-interesecting as a whole as shown in Figure \ref{fig:c3a}}. Let $\til{R}_{m,n}^{i,k}$ be the collection of all upright paths from $(1,2k-2i+2)$ to $(n-2i+1,m)$ that avoid the diagonal.
Let $R_{m,n}^{i,k}$ be the collection of all upright paths from $(1,2k-2i+1)$ to $(n-2i+2,m)$. \bl{Given any $(\pi_1',\ldots,\pi_{2k}')\in U(\Pi_{m,n}^{(2k)})$, we have $\pi_{2i-1}'\in \til{R}_{m,n}^{i,k}$ and $\pi_{2i}'\in R_{m,n}^{i,k}$}. By \ref{partc}, there are at most
\begin{align*}
    \prod_{i=1}^k 2^{|(j,j)\in \pi_{2i-1}'\cup \pi_{2i}'|}
\end{align*}
many $U$-pre-images of $(\pi_1',\ldots,\pi_{2k}')$.

\begin{figure}[h!]
    \centering
    \begin{subfigure}[b]{0.45\textwidth}
        \centering
          \begin{tikzpicture}[line cap=round,line join=round,>=triangle 45,x=0.5cm,y=0.5cm]
\draw [color=gray!30!white,, xstep=0.5cm,ystep=0.5cm] (1,1) grid (13,13);
\draw [line width=1pt] (1,4)-- (1,8);
\draw [line width=1pt] (8,10)-- (13,10);
\draw [line width=1pt] (8,10)-- (8,8);
\draw [line width=1pt] (1,8)-- (8,8);
\draw [line width=1pt] (1,3)-- (2,3);
\draw [line width=1pt] (2,3)-- (2,6);
\draw [line width=1pt] (3,6)-- (2,6);
\draw [line width=1pt] (3,6)-- (3,7);
\draw [line width=1pt] (3,7)-- (9,7);
\draw [line width=1pt] (9,9)-- (9,7);
\draw [line width=1pt] (13,9)-- (9,9);
\draw [line width=1pt] (3,4)-- (3,2);
\draw [line width=1pt] (5,6)-- (5,4);
\draw [line width=1pt] (5,6)-- (10,6);
\draw [line width=1pt] (10,6)-- (10,8);
\draw [line width=1pt] (13,8)-- (10,8);
\draw [line width=1pt] (3,4)-- (5,4);
\draw [line width=1pt] (1,2)-- (3,2);
\draw [line width=1pt] (1,1)-- (4,1);
\draw [line width=1pt] (4,2)-- (4,1);
\draw [line width=1pt] (4,2)-- (5,2);
\draw [line width=1pt] (5,2)-- (5,3);
\draw [line width=1pt] (5,3)-- (6,3);
\draw [line width=1pt] (6,3)-- (6,5);
\draw [line width=1pt] (6,5)-- (11,5);
\draw [line width=1pt] (11,5)-- (11,7);
\draw [line width=1pt] (11,7)-- (13,7);
\draw [line width=1pt] (1,1)-- (13,13);
\begin{scriptsize}
\draw [fill=black] (1,1) circle (2.5pt);
\node at (3.5,8.5) {$\pi_1$};
\node at (2.5,6.5) {$\pi_2$};
\node at (3.5,4.5) {$\pi_3$};
\node at (8.5,4.5) {$\pi_4$};
\end{scriptsize}
\end{tikzpicture}
\caption{}
    \end{subfigure}
\begin{subfigure}[b]{0.45\textwidth}
     \centering
     \begin{tikzpicture}[line cap=round,line join=round,>=triangle 45,x=0.5cm,y=0.5cm]
\draw [color=gray!30!white,, xstep=0.5cm,ystep=0.5cm] (1,1) grid (13,13);
\draw [line width=1pt] (1,4)-- (1,8);
\draw [line width=1pt] (9,13)-- (9,10);
\draw [line width=1pt] (8,10)-- (8,9);
\draw [line width=1pt] (1,8)-- (7,8);
\draw [line width=1pt] (7,9)-- (7,8);
\draw [line width=1pt] (7,9)-- (8,9);
\draw [line width=1pt] (8,10)-- (9,10);
\draw [line width=1pt] (1,3)-- (2,3);
\draw [line width=1pt] (2,3)-- (2,6);
\draw [line width=1pt] (3,6)-- (2,6);
\draw [line width=1pt] (3,6)-- (3,7);
\draw [line width=1pt] (3,7)-- (8,7);
\draw [line width=1pt] (9,9)-- (9,8);
\draw [line width=1pt] (8,8)-- (9,8);
\draw [line width=1pt] (8,8)-- (8,7);
\draw [line width=1pt] (10,9)-- (9,9);
\draw [line width=1pt] (10,9)-- (10,13);

\draw [line width=1pt,blue] (6,5)-- (4,5);
\draw [line width=1pt,blue] (6,5)-- (6,10);
\draw [line width=1pt,blue] (6,10)-- (8,10);
\draw [line width=1pt,blue] (8,13)-- (8,10);
\draw [line width=1pt,blue] (4,3)-- (4,5);
\draw [line width=1pt,blue] (1,1)-- (2,1);
\draw [line width=1pt,blue] (2,3)-- (2,1);
\draw [line width=1pt,blue] (2,3)-- (4,3);
\draw [line width=1pt,blue] (1,2)-- (1,4);
\draw [line width=1pt,blue] (2,4)-- (1,4);
\draw [line width=1pt,blue] (2,4)-- (2,5);
\draw [line width=1pt,blue] (2,5)-- (3,5);
\draw [line width=1pt,blue] (3,5)-- (3,6);
\draw [line width=1pt,blue] (3,6)-- (5,6);
\draw [line width=1pt,blue] (5,6)-- (5,11);
\draw [line width=1pt,blue] (5,11)-- (7,11);
\draw [line width=1pt,blue] (7,11)-- (7,13);
\draw [line width=1pt] (1,1)-- (13,13);
\begin{scriptsize}
\draw [fill=blue] (1,1) circle (2.5pt);
\draw [fill=blue] (8,13) circle (2.5pt);
\draw [fill=cyan] (1,2) circle (2.5pt);
\draw [fill=cyan] (7,13) circle (2.5pt);
\draw [fill=black] (1,3) circle (2.5pt);
\draw [fill=black] (10,13) circle (2.5pt);
\draw [fill=gray] (1,4) circle (2.5pt);
\draw [fill=gray] (9,13) circle (2.5pt);
\node at (4.5,3.5) {$\pi_4'$};
\node at (4.5,10.5) {$\pi_3'$};
\node at (9.5,8.5) {$\pi_2'$};
\node at (2.5,8.5) {$\pi_1'$};
\end{scriptsize}
\end{tikzpicture}
\caption{}
     \end{subfigure}
\caption{(A) $(\pi_1,\pi_2,\pi_3,\pi_4) \in \Pi_{13,10}^{(4)}$. (B) $(\pi_1',\pi_2',\pi_3',\pi_4')=U(\pi_1,\pi_2,\pi_3,\pi_4)$.}
    \label{fig:c3a}
\end{figure}
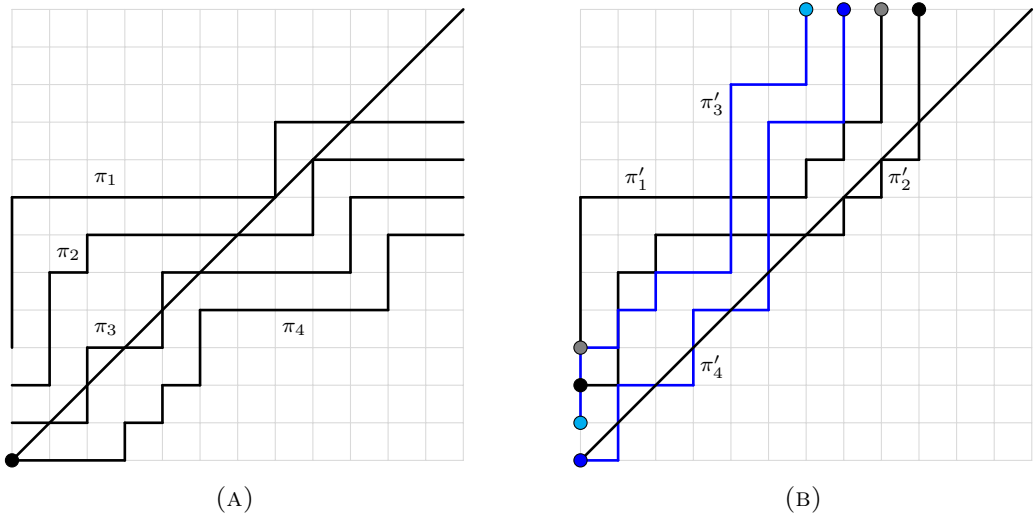

The $U$ map preserves the number of diagonal vertices by \ref{parta}. Furthermore by non-intersection, a $2k$-tuple of paths in $\Pi_{m,n}^{(2k)}$ has at most $n$ many diagonal vertices. Thus there are at most $2^{n}$ many inverses. Hence by \ref{parta} we have
\begin{align}\label{z2k}
     Z_{\operatorname{sym}}^{(2k)}(m,n) \le 2^n \cdot   \prod_{i=1}^k \left[\sum_{\pi_1\in \til{R}_{m,n}^{i,k}} \prod_{(i,j)\in \pi_1} \til{W}_{i,j} \right]\cdot   \prod_{i=1}^k \left[\sum_{\pi_2\in {R}_{m,n}^{i,k}} \prod_{(i,j)\in \pi_2} \til{W}_{i,j}\right].
\end{align}
We may elongate each of the path in $\til{R}_{m,n}^{i,k}$ and ${R}_{m,n}^{i,k}$ by appending an up-path from $(1,1)$ to $(1,2k-2i+2)$ and from $(1,1)$ to $(1,2k-2i+1)$ respectively. This produces elongated paths in $\til{\Pi}_{n-2i+1,m}^{(1)}$ and $\Pi_{n-2i+2,m}$ respectively. In terms of weights, we need to multiply the existing weights in \eqref{z2k} by $\prod_{j=1}^{2k-2i+1} \til{W}_{1,j}$ and $\prod_{j=1}^{2k-2i} \til{W}_{1,j}$ respectively to get the corresponding weights of elongated paths. After doing precisely the above, we have
\begin{align}\label{eq:sbd0}
     Z_{\operatorname{sym}}^{(2k)}(m,n) \le 2^n \cdot  \prod_{i=2}^{2k} \prod_{j=1}^{i-1} (\til{W}_{1,j})^{-1}  \cdot \prod_{i=1}^k \left[\til{Z}_{\operatorname{sym}}(n-2i+1,m)\zs(n-2i+2,m) \right].
\end{align}
We get \eqref{eq:sbd} from the above inequality in \eqref{eq:sbd0} by observing the definition of $V_q$ and $\til{V}_q$ from \eqref{defvq}. This completes the proof.     
\end{proof}

\medskip

The next lemma bounds $\log V_q$ and  $\log\til{V}_q$ from above with high probability. 
\begin{lemma}\label{l:log} Recall $R$ from \eqref{def:dk}. For every $\delta>0$  we have
\begin{align}\label{eq:vq}
\lim_{q\to\infty}\Pr\left( \log V_q \le (R+\delta)\tfrac{q}{2}\right)=1, \quad \lim_{q\to\infty}\Pr\left( \log \til{V}_q \le (-2\Psi(\theta)+\delta)\tfrac{q}{2}\right)=1.
    \end{align}
\end{lemma}
\begin{proof} Fix any $\delta>0$. Following \eqref{l:iden} and \eqref{defvq} we have
\begin{align*}
    V_{2N} = \sum_{p=0}^{N-1} \zh(N+p,N-p), \quad  V_{2N+1} = \sum_{p=1}^{N} \zh(N+p,N-p+1).
\end{align*}
From Theorem \ref{t:bw} (\eqref{lln} in particular) we have that
    \begin{align}\label{lln1}
        \frac{1}{N}\log \left[\sum_{p=1}^{N} \zh(N+p,N-p+1)\right] \stackrel{p}{\to} R, \quad \frac{1}{N}\log \left[\sum_{p=1}^{N-1} \zh(N+p,N-p)\right] \stackrel{p}{\to} R, 
    \end{align}
 Note that in the above equation, we have excluded $\zh(N,N)$ as their result does not contain $\zh(N, N)$ in the sum. However, in our case, we may include $\zh(N,N)$ by appealing to Theorem \ref{t:order}. First, in view of the above law of large numbers in \eqref{lln1}, we have 
 \begin{align}\label{vodd}
 \Pr(\log V_{2N+1}\le (R+\tfrac{1}{2}\delta)N)\to 1. 
 \end{align}
 On the other hand, by \eqref{eq:rem} we have $\sum_{p=1}^{N} e^{H_N^{(1)}(2p)}=\sum_{p=1}^{N} Z(N+p,N-p+1)$. Since $H_N^{(1)}(2)\le \log \sum_{p=1}^{N} e^{H_N^{(1)}(2p)} = \log V_{2N+1}$, \eqref{vodd} implies
    \begin{align*}
        \Pr(H_N^{(1)}(2) \le (R+\tfrac12\delta)N) \to 1,
    \end{align*}
    as $N\to\infty$. In addition, by ordering of points in the line ensemble (Theorem \ref{t:order}) we know that with probability at least $1-2^{-N}$, $H_N^{(1)}(1)\le H_N^{(1)}(2)+\log^2N$ for large enough $N$. Thus we have 
    \begin{align}\label{eq:h1}
        \Pr(H_N^{(1)}(1) \le (R+\delta)N) \to 1,
    \end{align}
    as $N\to\infty$. Given that $H_N^{(1)}(1) = \log Z(N, N)$, combining \eqref{eq:h1} and the second convergence in \eqref{lln1} yields $\Pr(\log V_{2N}\le (R+\delta)N) \to 1$ and together with \eqref{vodd} this concludes the proof of the first convergence in \eqref{eq:vq}. 

\medskip

Next, for the diagonal-avoiding case, let $(W_{i,i}^{\alpha=0})_{i\ge 1}$ be a family of weights distributed as $\frac12\operatorname{Gamma}(\theta)$ independent of $(W_{i,j})$. We set $W_{i,j}^{\alpha=0}:=\til{W}_{i,j}$ for $i\neq j$. This gives us a new collection of symmetrized weights. We denote the corresponding symmetrized partition function and the diagonal-avoiding symmetrized partition function as $\zs^{\alpha=0}$ and $\til{Z}_{\operatorname{sym}}^{\alpha=0}$ respectively. Observe that
    \begin{align} \label{eq:rel}
        \til{Z}_{\operatorname{sym}}(i,j) \le \frac{\til{W}_{1,1}}{W_{1,1}^{\alpha=0}} \cdot \til{Z}_{\operatorname{sym}}^{\alpha=0}(i,j) \le \frac{\til{W}_{1,1}}{W_{1,1}^{\alpha=0}} \cdot Z_{\operatorname{sym}}^{\alpha=0}(i,j).
    \end{align}
    The first equality in \eqref{eq:rel} is due to the fact that the weight corresponding $(1,1)$ is common in all paths and that is the only diagonal weight that appears in the diagonal avoiding symmetrized partition functions. The next inequality is obvious as we have just removed the diagonal avoiding restriction. This leads to
    \begin{align*}
        \log \til{V}_{q} \le \log \til{W}_{1,1}-\log W_{1,1}^{\alpha=0} +\log\left[\sum_{(i,j)\mid i+j=q} Z_{\operatorname{sym}}^{\alpha=0}(i,j)\right].
    \end{align*}
    The first two terms on the right-hand side of the above display are tight. An upper bound on the third term can be computed by the exact same analysis as $V_q$. Indeed, the law of large numbers and Theorem \ref{t:order} continue to  hold for $\alpha=0$ when $R$ becomes $-2\Psi(\theta)$ (see the last point in Theorem \ref{t:bw}). This concludes the proof of \eqref{eq:vq}.
\end{proof}
\medskip
Finally, with Lemmas \ref{l:sbd} and \ref{l:log} in place, we are ready to control the average law of large numbers of the top curves of the $\hslg$ line ensemble. 

\begin{proposition}\label{p:alaw} Recall $\Delta_k, R$ in \eqref{def:dk}. Fix any $\e>0$ and $k \in \Z_{>0}$ large such that $\Delta_k>0$. Then there exists $N_0(k,\e)>2k+1$ such that for all $N\ge N_0$ we have
    \begin{align*}
        \Pr\left( \sup_{p\in \ll1,2N-4k+2\rr} \frac1{2k}\sum_{i=1}^{2k} H_N^{(i)}(p) \le (R-\tfrac12\Delta_k)N\right) \ge 1-\e.
    \end{align*}
\end{proposition}

\sd{In other words, Proposition \ref{p:alaw} claims that when $k$ is large enough so that $\Delta_k > 0$, the average law of large numbers of top $2k$ curves is strictly less than $R$, which is the law of large numbers for point-to-(partial)line free energy process (see Theorem \ref{t:bw}).}

\medskip

\begin{proof}  Fix any $\e>0$. The definition of the $\hslg$ line ensemble in \eqref{eq:hslg} and \eqref{eq:sbd} collectively yield that, for all $p\in \ll1,N-2k+1\rr$,
\begin{align*}
    \sum_{i=1}^{2k} H_N^{(i)}(2p) & = 2k\log 2+\log  Z_{\operatorname{sym}}^{(2k)}(N+p,N-p+1) \\ & \le 2k\log 2+N\log 2- \log \left[\prod_{i=2}^{2k} \prod_{j=1}^{i-1} \til{W}_{1,j}\right] +\log \prod_{i=1}^k \left[V_{2N+3-2i}\til{V}_{2N+2-2i}\right],
\end{align*}
where the r.h.s.~is free of $p$. Hence we may take supremum over $p\in \ll1, N-2k+1\rr$ over both sides of the above display to get
\begin{equation}
    \label{eq:logbd}
    \begin{aligned}
    \sup_{p\in \ll1, N-2k+1\rr} \sum_{i=1}^{2k} H_N^{(i)}(2p)  & \le (2k+N)\log 2- \log \left[\prod_{i=2}^{2k} \prod_{j=1}^{i-1} \til{W}_{1,j}\right] \\ & \hspace{3cm}+\log \prod_{i=1}^k \left[V_{2N+3-2i}\til{V}_{2N+2-2i}\right].
\end{aligned}
\end{equation}
We now provide high probability upper bounds for each of the terms on the r.h.s. of \eqref{eq:logbd}. Let us take $\delta:=\frac{\Delta_k}{4}$. By Lemma \ref{l:log}, we may choose $N_0(k,\e)>2k+1$ large enough such that for all $N\ge N_0$
$$\Pr(\log V_{N} \le (R+\delta)\tfrac{N}2) \ge 1-\tfrac{\e}{8k}, \quad \Pr(\log \til{V}_{N} \le (-2\Psi(\theta)+\delta)\tfrac{N}2) \ge 1-\tfrac{\e}{8k}.$$
Thus applying a union bound we see that for all large enough $N$, with probability $1-\frac{\e}{4}$,
\begin{align}\label{eq:bd1}
    \log \prod_{i=1}^k \left[V_{N+3-2i}\til{V}_{N+2-2i}\right]\le RkN-2\Psi(\theta)kN +2k\delta N.
\end{align}
Note that the random variable $\log \left[\prod_{i=2}^{2k} \prod_{j=1}^{i-1} \til{W}_{1,j}\right]$ is tight and free of $N$. Hence with probability $1-\frac{\e}{4}$ one can ensure that
\begin{align}\label{eq:bd2}
    (2k+N)\log 2 - \log \left[\prod_{i=2}^{2k} \prod_{j=1}^{i-1} \til{W}_{1,j}\right] \le N\log 2+ 2k\delta N.
\end{align}
holds for all large enough $N$. Inserting the above two bounds in \eqref{eq:bd1} and \eqref{eq:bd2} back in \eqref{eq:logbd}, we have that with probability at least $1-\frac\e2,$
\begin{align}\label{even}
    \sup_{p\in \ll1,N-2k+1\rr}\frac{1}{2k}\sum_{i=1}^{2k} H_N^{(i)}(2p) \le \left[\tfrac1{2k}\log 2+2\delta +\tfrac{R}{2}-\Psi(\theta)\right]N,
\end{align}
for all large enough $N$.
As $\delta=\frac{\Delta_k}{4}$, the r.h.s. of \eqref{even} is precisely $(R-\tfrac12\Delta_k)N$. By the exact same argument, one can check that with probability at least $1-\frac\e2$ we have
\begin{align}\label{odd}
    \sup_{p\in \ll0,N-2k\rr}\frac{1}{2k}\sum_{i=1}^{2k} H_N^{(i)}(2p+1) \le (R-\tfrac12\Delta_k)N,
\end{align}
for all large enough $N$. Taking another union bound of \eqref{even} and \eqref{odd}, we get the desired result.
\end{proof}

\section{Controlling the second curve}\label{sec:2nd}

In this section, we establish the separation between the first and the second curve of our $\hslg$ line ensemble. On the basis of Proposition \ref{p:alaw}, Lemma \ref{lowcurve} first establishes that for large enough $k$ with high probability the $(2k+2)$-th curve $H_N^{(2k+2)}(\cdot)$ is uniformly $\mbox{const}\cdot N$ below than $RN$ over an interval of $\ll1, N\rr$, where $R$ defined in \eqref{def:dk} is the law of large numbers for point-to-(partial)line free energy process (Theorem \ref{t:bw}). Using this we will show in Proposition \ref{p:benchmark} that with high probability the second curve $H_N^{(2)}(\cdot)$ over an interval of length $O(\sqrt{N})$ is $M\sqrt{N}$ below $RN$ for any $M>0$. 
\begin{lemma}\label{lowcurve} Recall $R$ in \eqref{def:dk}. Fix any $\e>0$ and $k\in \Z_{>0}$ large enough such that $\Delta_k>0$. Then there exists $N_0(k,\e)$ such that for all $N\ge N_0$ we have
\begin{align}\label{eq:lowcurve}
        \Pr\left( \sup_{p\in \ll1,N\rr}  H_N^{(2k+2)}(p) \le (R-\tfrac14\Delta_k)N\right) \ge 1-\e.
\end{align} 
\end{lemma}
\begin{proof} Let us consider the following events
\begin{align*}
    \m{A} & :=\left\{\sup_{p\in \ll1,N\rr}  H_N^{(2k+2)}(p) \le (R-\tfrac14\Delta_k)N\right\}, \\ \m{B} & :=\left\{H_N^{(i+1)}(p) \le H_N^{(i)}(p)+2\log^2N, \mbox{ for all } i\in \ll1,2k+1\rr, p\in \ll1,N\rr\right\}, \\
    \m{C} & :=\left\{\sup_{p\in \ll1,N\rr}\frac1{2k}\sum_{i=1}^{2k} H_N^{(i)}(p) \le (R-\tfrac12\Delta_k)N\right\}.
\end{align*}
We claim that for all large enough $N$, we have $(\m{B} \cap\neg \m{A})\subset \m{\neg C}$. To see this, note that on $\m{B}\cap \neg \m{A}$, there exists a point $p^*\in \ll1,N\rr$ such that $H_N^{(2k+2)}(p^*) > (R-\tfrac14\Delta_k)N)$ and hence (as $\m{B}$ holds)
$$H_N^{(i)}(p^*) > (R-\tfrac14\Delta_k)N-(4k+4)\log^2N,$$
for all $i\in \ll1,2k+1\rr$. However, the above display also implies that
$$\sup_{p\in \ll1,N\rr}\frac1{2k}\sum_{i=1}^{2k} H_N^{(i)}(p) > (R-\tfrac14\Delta_k)N-(4k+4)\log^2N$$
which is strictly bigger than $(R-\tfrac12\Delta_k)N$ (for large enough $N$) and implies $\neg \m{C}$. Thus by a union bound, we have
\begin{align}\label{eq:negbd}
    \Pr(\neg\m{A})\le \Pr(\neg \m{B})+\Pr(\m{B} \cap \neg \m{A} ) \le \Pr(\neg\m{B})+\Pr(\neg\m{C}).
\end{align}
Note that for fixed $k$, by Theorem \ref{t:order} with $\rho=\frac12$ and a union bound, we have $\Pr(\neg\m{B})\le N\cdot (2k+1)\cdot 2^{-N}\le \frac{\e}{2}$ for all $N\ge N_1(k, \e)$. On the other hand, Proposition \ref{p:alaw} yields that for fixed $k$ and $\e$, $\Pr(\neg \m{C})\le \frac\e2$ for all $N$ greater than some $N_2(k, \frac\e2)$. Letting $N_0(k, \e)= \max\{N_1, N_2\}$ and inserting these two bounds in \eqref{eq:negbd} leads to \eqref{eq:lowcurve}.
\end{proof}

\medskip

Building on Lemma \ref{lowcurve}, the next result demonstrates that on a given interval of length $O(\sqrt{N})$ starting from $1$ and any $M_2>0$, the second curve $H_N^{(2)}(\cdot)$ is uniformly lower than $RN-M_2\sqrt{N}$ with high probability (see Figure \ref{fig:2curve}).
\begin{proposition} \label{p:benchmark} Recall $\Delta_k, R$ in \eqref{def:dk}. Fix $\e\in (0,1)$, $M_1,M_2\ge 1$ and $k \in \Z_{>0}$ such that $\Delta_k > 0$. Then there exists a constant $N_2(\e,M_1,M_2)>0$ such that for all $N\ge N_2$ we have
  \begin{align}\label{eq:benchmark}
      \Pr\left(\sup_{p\in [1,2\lfloor M_1\sqrt{N}\rfloor +1]} H_N^{(2)}(p) \le RN-M_2\sqrt{N}\right) \ge 1-\tfrac12\epsilon.
  \end{align}  
\end{proposition}

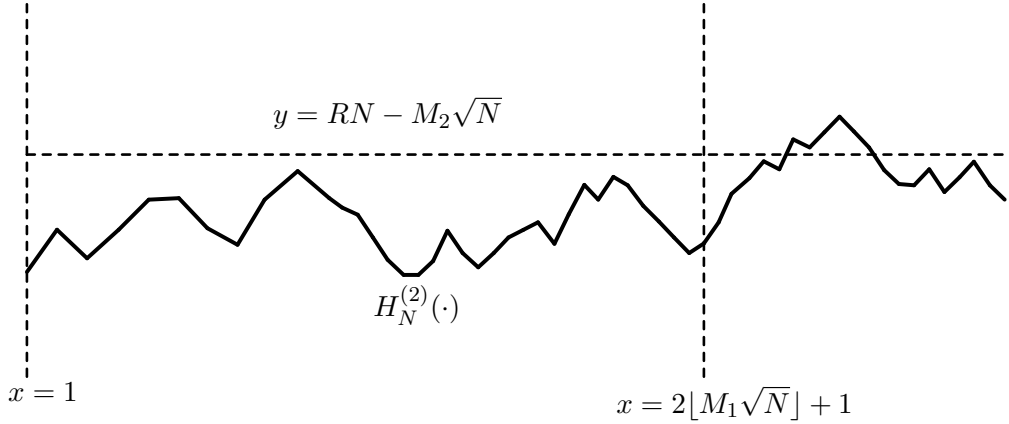
\begin{figure}[h!]
    \centering
    \begin{tikzpicture}[line cap=round,line join=round,>=triangle 45,x=1cm,y=1cm]
\draw [line width=1pt,dashed] (5,5)-- (18,5);
\draw [line width=1pt,dashed] (5,7)-- (5,2);
\draw [line width=1pt,dashed] (14,7)-- (14,2);
\draw [line width=1.5pt] (5,3.4375)-- (5.4,4);
\draw [line width=1.5pt] (5.4,4)-- (5.8,3.62);
\draw [line width=1.5pt] (5.8,3.62)-- (6.22,4);
\draw [line width=1.5pt] (6.22,4)-- (6.62,4.4);
\draw [line width=1.5pt] (6.62,4.4)-- (7.02,4.42);
\draw [line width=1.5pt] (7.02,4.42)-- (7.4,4.02);
\draw [line width=1.5pt] (7.4,4.02)-- (7.8,3.8);
\draw [line width=1.5pt] (7.8,3.8)-- (8.16,4.4);
\draw [line width=1.5pt] (8.16,4.4)-- (8.6,4.78);
\draw [line width=1.5pt] (8.6,4.78)-- (9.02,4.42);
\draw [line width=1.5pt] (9.02,4.42)-- (9.190994133662238,4.294523498358001);
\draw [line width=1.5pt] (9.190994133662238,4.294523498358001)-- (9.398154303610122,4.199341798652217);
\draw [line width=1.5pt] (9.398154303610122,4.199341798652217)-- (9.599715550045898,3.8969999289985515);
\draw [line width=1.5pt] (9.599715550045898,3.8969999289985515)-- (9.795677872969572,3.6002569828569904);
\draw [line width=1.5pt] (9.795677872969572,3.6002569828569904)-- (10.008436966429558,3.3986957364212134);
\draw [line width=1.5pt] (10.008436966429558,3.3986957364212134)-- (10.20439928935323,3.3986957364212134);
\draw [line width=1.5pt] (10.20439928935323,3.3986957364212134)-- (10.400361612276903,3.583460212320676);
\draw [line width=1.5pt] (10.400361612276903,3.583460212320676)-- (10.59072501168847,3.9865827051922302);
\draw [line width=1.5pt] (10.59072501168847,3.9865827051922302)-- (10.792286258124248,3.689839759050669);
\draw [line width=1.5pt] (10.792286258124248,3.689839759050669)-- (11,3.5);
\draw [line width=1.5pt] (11,3.5)-- (11.206606598020011,3.689839759050669);
\draw [line width=1.5pt] (11.206606598020011,3.689839759050669)-- (11.402568920943684,3.8969999289985515);
\draw [line width=1.5pt] (11.402568920943684,3.8969999289985515)-- (11.598531243867356,3.99778055221644);
\draw [line width=1.5pt] (11.598531243867356,3.99778055221644)-- (11.79449356679103,4.0985611754343285);
\draw [line width=1.5pt] (11.79449356679103,4.0985611754343285)-- (12.012851583763121,3.8130160763169774);
\draw [line width=1.5pt] (12.012851583763121,3.8130160763169774)-- (12.20321498317469,4.204940722164322);
\draw [line width=1.5pt] (12.20321498317469,4.204940722164322)-- (12.41037515312257,4.596865368011667);
\draw [line width=1.5pt] (12.41037515312257,4.596865368011667)-- (12.595139629022032,4.400903045087994);
\draw [line width=1.5pt] (12.595139629022032,4.400903045087994)-- (12.79670087545781,4.70324491474166);
\draw [line width=1.5pt] (12.79670087545781,4.70324491474166)-- (12.987064274869377,4.591266444499562);
\draw [line width=1.5pt] (12.987064274869377,4.591266444499562)-- (13.19422444481726,4.316919192406421);
\draw [line width=1.5pt] (13.19422444481726,4.316919192406421)-- (13.401384614765142,4.1097590224585385);
\draw [line width=1.5pt] (13.401384614765142,4.1097590224585385)-- (13.602945861200919,3.8969999289985515);
\draw [line width=1.5pt] (13.602945861200919,3.8969999289985515)-- (13.804507107636697,3.689839759050669);
\draw [line width=1.5pt] (13.804507107636697,3.689839759050669)-- (14,3.8176236100142256);
\draw [line width=1.5pt] (14,3.8176236100142256)-- (14.196431753484026,4.098561175434321);
\draw [line width=1.5pt] (14.196431753484026,4.098561175434321)-- (14.36693898140803,4.478040409376432);
\draw [line width=1.5pt] (14.36693898140803,4.478040409376432)-- (14.605153169867686,4.686448144205338);
\draw [line width=1.5pt] (14.605153169867686,4.686448144205338)-- (14.795516569279252,4.910405084689534);
\draw [line width=1.5pt] (14.795516569279252,4.910405084689534)-- (15.002676739227134,4.804025537959541);
\draw [line width=1.5pt] (15.002676739227134,4.804025537959541)-- (15.187441215126597,5.201549107318989);
\draw [line width=1.5pt] (15.187441215126597,5.201549107318989)-- (15.405799232098675,5.095169560588994);
\draw [line width=1.5pt] (15.405799232098675,5.095169560588994)-- (15.601761555022346,5.296730807024771);
\draw [line width=1.5pt] (15.601761555022346,5.296730807024771)-- (15.803322801458123,5.503890976972652);
\draw [line width=1.5pt] (15.803322801458123,5.503890976972652)-- (15.99368620086969,5.307928654048981);
\draw [line width=1.5pt] (15.99368620086969,5.307928654048981)-- (16.195247447305466,5.095169560588994);
\draw [line width=1.5pt] (16.195247447305466,5.095169560588994)-- (16.396808693741242,4.7928276909353285);
\draw [line width=1.5pt] (16.396808693741242,4.7928276909353285)-- (16.592771016664916,4.608063215035867);
\draw [line width=1.5pt] (16.592771016664916,4.608063215035867)-- (16.794332263100692,4.591266444499552);
\draw [line width=1.5pt] (16.794332263100692,4.591266444499552)-- (16.99589350953647,4.8040255379595385);
\draw [line width=1.5pt] (16.99589350953647,4.8040255379595385)-- (17.197454755972245,4.501683668305874);
\draw [line width=1.5pt] (17.197454755972245,4.501683668305874)-- (17.399016002408022,4.697645991229545);
\draw [line width=1.5pt] (17.399016002408022,4.697645991229545)-- (17.58937940181959,4.904806161177427);
\draw [line width=1.5pt] (17.58937940181959,4.904806161177427)-- (17.802138495279575,4.591266444499552);
\draw [line width=1.5pt] (17.802138495279575,4.591266444499552)-- (17.99810081820325,4.4009030450879845);
\draw (8.13641575535154,5.9063714364755455) node[anchor=north west] {$y=RN-M_2\sqrt{N}$};
\draw (9.473386931292762,3.4) node[anchor=north west] {$H_N^{(2)}(\cdot)$};
\draw (4.6146380236039395,2.1074167536181707) node[anchor=north west] {$x=1$};
\draw (12.70168318539473,2.058503174010565) node[anchor=north west] {$x=2\lfloor M_1\sqrt{N} \rfloor +1$};
\end{tikzpicture}
    \caption{The high probability event in Proposition \ref{p:benchmark}.}
    \label{fig:2curve}
\end{figure}

\begin{proof} \sd{The proof of Proposition \ref{p:benchmark} is conducted in the following stages: }

\sd{\begin{itemize}[leftmargin=18pt]
\setlength\itemsep{0.3em}
   \item Using Theorem \ref{t:bw} and Lemma \ref{lowcurve}, we determine the high probability locations of $H_N^{(1)}(2M\sqrt{N}+1)$ and $H_N^{(2k+2)}(\cdot)$.  Using the ordering of points in Theorem \ref{t:order}, we then bound the endpoints $H_N^{(i)}(2M\sqrt{N}+1)$, $i \in \ll 1, 2k+1\rr$ from above based on the high probability locations of $H_N^{(1)}(2M\sqrt{N}+1)$ and the $(2k+2)$-th curve.
    \item We next consider the conditional law of $(H_N^{(i)}\ll1,2M\sqrt{N}\rr)_{i\in \ll1,2k+1\rr}$ given the above boundary conditions. By Theorem \ref{t:gibbs}, this law is given by an appropriate $\hslg$ Gibbs measure. Applying stochastic monotonicity, we may also assume that the $H_N^{(2i-1)}(2{M}\sqrt{N}+1)$ and $H_N^{(2i)}(2{M}\sqrt{N}+1)$ are sufficiently far apart. This will allow us to approximate the Gibbs measure as a product of interacting random walks defined in Definition \ref{irw}.
    \item Lastly, we use the associated estimates of interacting random walks from Proposition \ref{p:irw} to dissect the Gibbs measure and yield a quantitative bound in our favor.
\end{itemize}}

\sd{We now flesh out the technical details of the above stages. In the following proof, we assume all the multiples of $\sqrt{N}$ appearing below are integers for convenience in notation. The general case follows verbatim by considering the floor function. For clarity, we split our proof into several steps.}

\medskip

\noindent\textbf{Step 1.}   Let us consider the $\hslg$ line ensemble $H_N = (H_N^{(1)}, \ldots, H_N^{(N)})$.  Fix any $\e \in (0, 1)$, $M_1,M_2\ge 1$ and any $k\in \Z_{>0}$ such that $\Delta_k>0$. \bl{In this step we will show how to choose a boundary point $p^*=2M\sqrt{N}+1$ where $H_N^{(1)}(p^*)$ is sufficiently small. Using this boundary location, we will then control the supremum of $H_N^{(2)}(p)$ on the interval $\ll1,p^*\rr$ in subsequent steps.}

Let $\Phi(x)$ be the cumulative distribution function of a standard Gaussian random variable. Set $\tau: = \Psi(\theta-\alpha)- \Psi(\theta+\alpha)$. Let $M\in \Z_{>0}$ whose precise value is to be determined. Taking $g(N) =M\sqrt{N}$ in Theorem \ref{t:bw} yields 
  \begin{align}\label{eq:gauconv2}
    \frac{1}{\sigma\sqrt{N}}\left[\log \zt(M\sqrt{N})-RN+M\tau\sqrt{N}\right] \stackrel{d}{\to} \mathcal{N}\big(0,1\big).
\end{align} Note that \eqref{eq:gauconv2} implies
\begin{align*}
    \Pr\left(\frac{1}{\sigma\sqrt{N}}\left[\log \left[\zt(M\sqrt{N})\right]-RN+M\tau\sqrt{N}\right] \le \Phi^{-1}(1-\tfrac{\e}{2})\right) \to 1-\frac{\e}{2}.
\end{align*}
Thus for $N$ large enough, we have that with probability greater than $1-\e$, 
\begin{align}\label{eq:sup}
\log \left[\zt(M\sqrt{N})\right] \le  RN - \big(M\tau- \Phi^{-1}(1-\tfrac{\e}{2})\sigma\big)\sqrt{N}.
\end{align}
Set $M = \max\{M_1, \frac1\tau(M_2+k+1+\Phi^{-1}(1-\tfrac\e2)\sigma)\}.$ Note that by definition, $H_N^{(1)}(2M\sqrt{N}+1)\le \log \zt(M\sqrt{N})$ and as $M\tau- \Phi^{-1}(1-\frac{\e}{2})\sigma > M_2+k +1$, \eqref{eq:sup}
yields that
\begin{align}\label{bdA}
   \Pr(\m{A}) \ge 1-\e, \mbox{ where } \m{A}:=\left\{H_N^{(1)}(2M\sqrt{N}+1) \le RN-(M_2+k+1)\sqrt{N}\right\}
\end{align}
for all large enough $N$. This fixes our choice for the boundary location and complete our work for this step.

\medskip

\noindent\textbf{Step 2.} Set $T=M\sqrt{N}+1$. We claim that
\begin{align}\label{e:tshw}
    \Pr(\neg\m{E}) \le 3\e+ \frac{k\e}{(1-\e)^{k+1}}, \mbox{ where }\m{E} & :=\left\{\sup_{p\in \ll 1,2T -1\rr} H_N^{(2)}(p) \le RN-M_2\sqrt{N}\right\}.
\end{align}
Since $2T-1 \ge 2M_1\sqrt{N}+1$, assuming \eqref{e:tshw} and adjusting $\e$ yield \eqref{eq:benchmark}.  In this and subsequent steps we will prove \eqref{e:tshw}. To begin with, we consider several events:
\begin{align*}
    \m{B} & :=\bigcap_{i=1}^{2k}\left\{H_N^{(i+1)}(2T) \le H_N^{(i)}(2T-1)+\log^2N, \right. \\ & \hspace{4cm}\left. H_N^{(i+1)}(2T-1) \le H_N^{(i+1)}(2T)+\log^2N\right\}, \\
    \m{C} & := \left\{\sup_{p\in \ll1,N\rr} H_N^{(2k+2)}(p) \le (R-\tfrac14\Delta_k)N\right\}, \\
    \m{D} & :=\bigcap_{i=1}^{k}\bigg\{\max\big\{H_N^{(2i)}(2T-1), H_N^{(2i)}(2T),H_N^{(2i+1)}(2T)\big\} \\ & \hspace{5cm}\le RN-(M_2+k+1)\sqrt{N}+4k\log^2N\bigg\}.
\end{align*}
Let us consider the $\sigma$-field
\begin{align*}
    \mathcal{F} & :=\sigma\left\{H_N^{(2i)}\ll 2T-1,2N-4i+2\rr, H_N^{(2i+1)}\ll 2T,2N-4i\rr, \ i\in \ll1,k\rr,\right. \\ & \hspace{4cm} \left. H_N^{(1)}\ll 1,2N\rr, H_N^{(j)}\ll 1,2N-2j+2\rr, \ j\in \ll2k+2,N\rr\right\}.
\end{align*}
 By Theorem \ref{t:order} with $\rho=\frac12$, we have $\Pr(\neg\m{B}) \le 4k 2^{-N} \le \e$ for all large enough $N$.  Observe that $\m{A}\cap\m{B} \subset \m{D}$ and recall that $\Pr(\neg\m{A})< \e$ in \eqref{bdA}. Thus via the union bound, we have $\Pr(\neg\m{D})\le \Pr(\neg\m{A})+\Pr(\neg\m{B})\le 2\e$. Note that $\m{C}\cap\m{D}$ is measurable w.r.t.~$\mathcal{F}$. Applying the union bound and tower property of conditional expectation we get
\begin{align}\label{fbd}
    \Pr(\neg\m{E}) & \le \Pr(\neg\m{C})+\Pr(\neg\m{D})+\Pr(\m{C}\cap\m{D}\cap \neg\m{E}) \le 3\e+\Ex\left[\ind_{\m{C}\cap\m{D}}\cdot \Ex\left[\ind_{\neg \m{E}} \mid \mathcal{F}\right]\right].
\end{align}
where in the last inequality we have used Lemma \ref{lowcurve} to get that $\Pr(\neg\m{C}) \le \e$ for all large enough $N$. We claim that 
\begin{align}\label{nebd}
    \Ex\left[\ind_{\m{C}\cap\m{D}}\cdot \Ex\left[\ind_{\neg \m{E}} \mid \mathcal{F}\right]\right] \le \frac{k\e}{(1-\e)^{k+1}}.
\end{align}
We will demonstrate \eqref{nebd} in the \textbf{Steps 3-4}. Currently, assuming the validity of \eqref{nebd} and appealing to \eqref{fbd} prove \eqref{e:tshw}.
\medskip

\noindent\textbf{Step 3.} In this step we study $\ind_{\m{C}\cap\m{D}}\Ex\left[\ind_{\neg \m{E}} \mid \mathcal{F}\right]$ by invoking the Gibbs property (Theorem \ref{t:gibbs}). Let us consider the domain
\begin{align*}
    \Theta_{k,T}:=\{(i,j) \mid i\in \ll2,2k+1\rr, j\in \ll1,2T-1-\ind_{i=\operatorname{even}}\rr\}.
\end{align*}
\begin{figure}[h!]
    \centering
    \begin{tikzpicture}[line cap=round,line join=round,>=triangle 45,x=1.2cm,y=1cm]
    \draw[fill=gray!10,line width=0.5pt,dashed] (-0.75,9.25)--(-0.75,3.25)--(2.75,3.25)--(2.75,3.75)--(2.25,3.75)--(2.25,5.25)--(2.75,5.25)--(2.75,5.75)--(2.25,5.75)--(2.25,7.25)--(2.75,7.25)--(2.75,7.75)--(2.25,7.75)--(2.25,9.25)--(-0.75,9.25);
    \foreach \x in {0,1,2,3}
        {\draw[line width=1pt,red,{Latex[length=2mm]}-]  (\x,9) -- (\x-0.5,8.5);
        \draw[line width=1pt,blue,{Latex[length=2mm]}-]  (\x,8) -- (\x-0.5,7.5);
        \draw[line width=1pt,red,{Latex[length=2mm]}-]  (\x,8) -- (\x+0.5,7.5);
        \draw[line width=1pt,blue,{Latex[length=2mm]}-]  (\x,9) -- (\x+0.5,8.5);
        \draw[line width=1pt,red,{Latex[length=2mm]}-]  (\x,7) -- (\x-0.5,6.5);
        \draw[line width=1pt,blue,{Latex[length=2mm]}-]  (\x,6) -- (\x-0.5,5.5);
        \draw[line width=1pt,red,{Latex[length=2mm]}-]  (\x,6) -- (\x+0.5,5.5);
        \draw[line width=1pt,blue,{Latex[length=2mm]}-]  (\x,7) -- (\x+0.5,6.5);
        \draw[line width=1pt,red,{Latex[length=2mm]}-]  (\x,5) -- (\x-0.5,4.5);
        \draw[line width=1pt,blue,{Latex[length=2mm]}-]  (\x,4) -- (\x-0.5,3.5);
        \draw[line width=1pt,red,{Latex[length=2mm]}-]  (\x,4) -- (\x+0.5,3.5);
        \draw[line width=1pt,blue,{Latex[length=2mm]}-]  (\x,5) -- (\x+0.5,4.5);
        \foreach \y in {4,5,6,7,8,9,10}{
     \draw[line width=1pt,black,{Latex[length=2mm]}-]  (\x+0.5,\y-0.5) -- (\x,\y-1);
     \draw[line width=1pt,black,{Latex[length=2mm]}-]  (\x-0.5,\y-0.5) -- (\x,\y-1);
     }
     }
        \foreach \x in {0,1,2,3}{
     \draw[line width=1pt,red,{Latex[length=2mm]}-]  (\x,8) -- (\x+0.5,7.5);
     \draw[line width=1pt,blue,{Latex[length=2mm]}-]  (\x,9) -- (\x+0.5,8.5);
     \draw[line width=1pt,red,{Latex[length=2mm]}-]  (\x,10) -- (\x+0.5,9.5);
     \draw[line width=1pt,blue,{Latex[length=2mm]}-]  (\x,10) -- (\x-0.5,9.5);
     \draw[line width=1pt,red,{Latex[length=2mm]}-]  (\x,9) -- (\x-0.5,8.5);
     }
    \end{tikzpicture}
     
    \caption{$\Theta_{k,T}$ for $k=3$, $T=4$ shown in the shaded region. The $\hslg$ Gibbs measure on $\Theta_{3,4}$ with boundary condition $(u_{i,j})_{(i,j)\in \partial\Theta_{3,4}}.$}
    \label{figsm}
\end{figure}
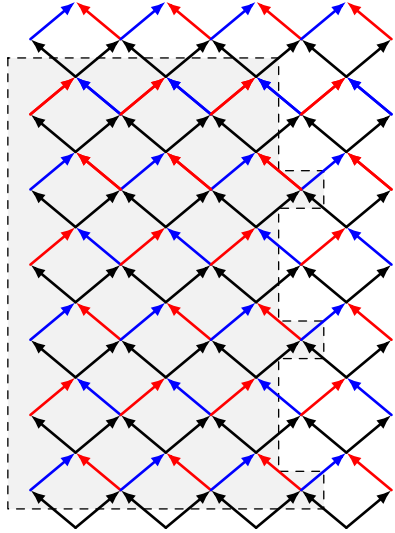
By Theorem \ref{t:gibbs}, the distribution of the line ensemble conditioned on $\mathcal{F}$ is given by $\pg^{\vec{u}}$, i.e. the $\hslg$ Gibbs measure on the domain $\Theta_{k,T}$ with boundary condition $\vec{u}:=(H_N^{(i)}(j))_{(i,j)\in \partial\Theta_{k,T}}$ and the boundary set of $\Theta_{k,T}$ is given by
\begin{align*}
    \partial\Theta_{k,T} & =\big\{(1,2j-1), (2k+2, 2j), (2, 2T-1), (3, 2T), \\ & \hspace{3cm} (2i,2T-1),(2i,2T),(2i+1,2T) \mid i\in \ll2,k\rr, j\in \ll1,T\rr\big\}.
\end{align*}
Note that for large enough $N$, on the event $\m{C}\cap\m{D}$ we have
\begin{align}
   \label{inf} & H_N^{(1)}(2j-1) \le x_{1,2j-1} :=\infty, \quad j\in \ll1,T\rr, \\ & H_N^{(2i)}(2T-1) \le x_{2i,2T-1}:=RN-(M_2+i)\sqrt{N}, \quad i\in \ll1,k\rr, \notag\\ & H_N^{(2i)}(2T) \le x_{2i,2T}:=RN-(M_2+i)\sqrt{N}, \quad i\in \ll 2,k\rr, \notag\\ & H_N^{(2i+1)}(2T) \le x_{2i+1,2T}  :=RN-(M_2+i)\sqrt{N}-\sqrt{T}, \quad i\in \ll1,k\rr,\notag\\ & H_N^{(2k+2)}(2j) \le x_{2k+2,2j} := RN-(M_2+k+1)\sqrt{N}, \quad j\in \ll1,T\rr.\notag
\end{align}
where $\m{C}$ holds only in the last inequality. Since $\neg \m{E}$ event is increasing with respect to the boundary data, by stochastic monotonicity we have
\begin{align}\label{eq:stoc}
    \ind_{\m{C}\cap\m{D}}\cdot \Ex\left[\ind_{\neg \m{E}} \mid \mathcal{F}\right] \le \ind_{\m{C}\cap\m{D}}\cdot  \pg^{\vec{u}}(\neg\m{E}) \le \pg^{\vec{x}}(\neg\m{E}).
\end{align}
\sd{To bound $\pg^{\vec{x}}(\neg\m{E})$ we need a convenient alternative representation for the $\pg^{\vec{x}}$ measure. Towards this end, through the Gibbs property, we dissect the $\pg^{\vec{x}}$ measure into blocks of independent interacting random walks (Definition \ref{irw}) and the Radon-Nikodym derivatives interleaved between adjacent blocks (see Figure \ref{figsm2}). Let us  now describe this decomposition.} 

\smallskip

Recall the interacting random walk ($\irw$) from Definition \ref{irw}. Let $(L_{2i}\ll1,2T-2\rr, L_{2i+1}\ll1,2T-1\rr)_{i=1}^k$ be $k$ independent $\irw$s of length $T$ with boundary condition $(x_{2i,2T-1}, x_{2i+1,2T})$. Let us denote the joint law and expectation of $L$ as $\pc^{\vec{x}}$ and $\ec^{\vec{x}}$ respectively. Set
\begin{align}
   \label{e:wb} \wb:=\exp\left(-\sum_{i=1}^{k}\sum_{j=1}^{T} \left[e^{L_{2i+2}(2j)-L_{2i+1}(2j+1)}+e^{L_{2i+2}(2j)-L_{2i+1}(2j-1)}\right]\right)
\end{align}
with the convention $L_{2i+1}(2T+1)=\infty$ for $i\in \ll1,k\rr$ and $L_{i}(j)=x_{i,j}$ for all $(i,j)\in \partial\Theta_{k,T}$. Note that here only $H_N^{(1)}(2j+1), j \in \ll1, T\rr$ are in the boundary and are set to $\infty$ in \eqref{inf}. Thus, their contributions to the Radon-Nikodym derivative $\wb$ would be $\prod_{j =1 }^{2T-2}\exp(-e^{H_N^{(2)}(j)-\infty}) =1.$ From the description of the $\hslg$ Gibbs measure, we have
\begin{align}\label{e:gibb2}
    \pg^{\vec{x}}(\neg\m{E}) = \frac{\ec^{\vec{x}}[\wb \ind_{\neg \m{E}}]}{\ec^{\vec{x}}[\wb]},
\end{align}

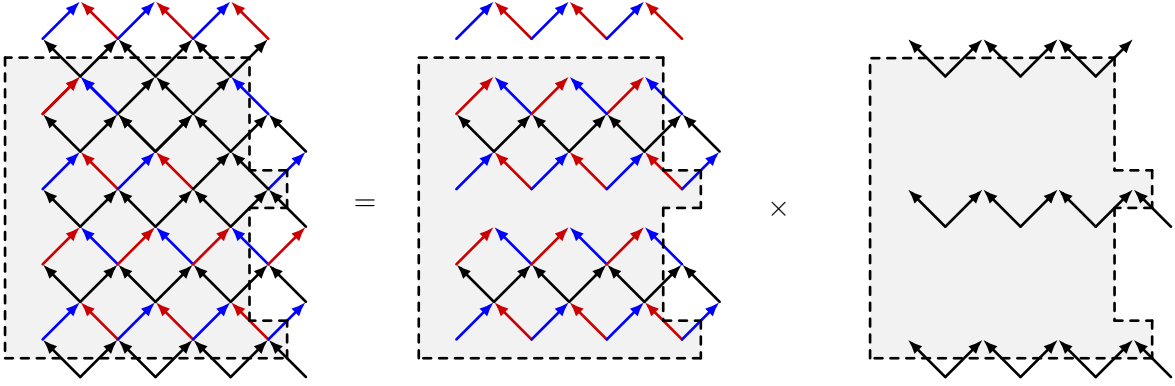
\begin{figure}[h!]
    \centering
\definecolor{ccqqqq}{rgb}{0.8,0,0}
\definecolor{qqqqff}{rgb}{0,0,1}
\begin{tikzpicture}[line cap=round,line join=round,>=triangle 45,x=0.5cm,y=0.5cm]
\fill[line width=1pt,dashed,fill=gray!10] (4,2.5) -- (10.5,2.5) -- (10.5,-0.5) -- (11.5,-0.5) -- (11.5,-1.5) -- (10.5,-1.5) -- (10.5,-4.5) -- (11.5,-4.5) -- (11.5,-5.5) -- (4,-5.5) -- cycle;
\fill[line width=1pt,dashed,fill=gray!10] (16,2.5) -- (22.5,2.5) -- (22.5,-0.5) -- (23.5,-0.5) -- (23.5,-1.5) -- (22.5,-1.5) -- (22.5,-4.5) -- (23.5,-4.5) -- (23.5,-5.5) -- (16,-5.5) -- cycle;
\fill[line width=1pt,dashed,fill=gray!10] (-0.5,2.5) -- (-7,2.5) -- (-7,-5.5) -- (0.5,-5.5) -- (0.5,-4.5) -- (-0.5,-4.5) -- (-0.5,-1.5) -- (0.5,-1.5) -- (0.5,-0.5) -- (-0.5,-0.5) -- cycle;
\draw [-{Latex[length=2mm]},line width=1pt,color=qqqqff] (-6,3) -- (-5,4);
\draw [-{Latex[length=2mm]},line width=1pt,color=ccqqqq] (-4,3) -- (-5,4);
\draw [-{Latex[length=2mm]},line width=1pt,color=qqqqff] (-4,3) -- (-3,4);
\draw [-{Latex[length=2mm]},line width=1pt,color=ccqqqq] (-2,3) -- (-3,4);
\draw [-{Latex[length=2mm]},line width=1pt,color=qqqqff] (-2,3) -- (-1,4);
\draw [-{Latex[length=2mm]},line width=1pt,color=ccqqqq] (0,3) -- (-1,4);
\draw [-{Latex[length=2mm]},line width=1pt] (-5,2) -- (-6,3);
\draw [-{Latex[length=2mm]},line width=1pt] (-5,2) -- (-4,3);
\draw [-{Latex[length=2mm]},line width=1pt] (-3,2) -- (-4,3);
\draw [-{Latex[length=2mm]},line width=1pt] (-3,2) -- (-2,3);
\draw [-{Latex[length=2mm]},line width=1pt] (-1,2) -- (-2,3);
\draw [-{Latex[length=2mm]},line width=1pt] (-1,2) -- (0,3);
\draw [-{Latex[length=2mm]},line width=1pt,color=qqqqff] (0,1) -- (-1,2);
\draw [-{Latex[length=2mm]},line width=1pt,color=ccqqqq] (-6,1) -- (-5,2);
\draw [-{Latex[length=2mm]},line width=1pt] (-5,0) -- (-6,1);
\draw [-{Latex[length=2mm]},line width=1pt] (-5,0) -- (-4,1);
\draw [-{Latex[length=2mm]},line width=1pt,color=qqqqff] (-4,1) -- (-5,2);
\draw [-{Latex[length=2mm]},line width=1pt] (-3,0) -- (-4,1);
\draw [-{Latex[length=2mm]},line width=1pt] (-3,0) -- (-2,1);
\draw [-{Latex[length=2mm]},line width=1pt] (-1,0) -- (-2,1);
\draw [-{Latex[length=2mm]},line width=1pt,color=ccqqqq] (-6,1) -- (-5,2);
\draw [-{Latex[length=2mm]},line width=1pt,color=qqqqff] (-4,1) -- (-5,2);
\draw [-{Latex[length=2mm]},line width=1pt] (-3,0) -- (-4,1);
\draw [-{Latex[length=2mm]},line width=1pt] (-3,0) -- (-2,1);
\draw [-{Latex[length=2mm]},line width=1pt,color=qqqqff] (-6,-1) -- (-5,0);
\draw [-{Latex[length=2mm]},line width=1pt,color=ccqqqq] (-4,-1) -- (-5,0);
\draw [-{Latex[length=2mm]},line width=1pt,color=qqqqff] (-4,-1) -- (-3,0);
\draw [-{Latex[length=2mm]},line width=1pt,color=ccqqqq] (-2,-1) -- (-3,0);
\draw [-{Latex[length=2mm]},line width=1pt,color=qqqqff] (0,-1) -- (1,0);
\draw [-{Latex[length=2mm]},line width=1pt] (1,-2) -- (0,-1);
\draw [-{Latex[length=2mm]},line width=1pt] (-1,-2) -- (0,-1);
\draw [-{Latex[length=2mm]},line width=1pt] (-1,-2) -- (-2,-1);
\draw [-{Latex[length=2mm]},line width=1pt] (-3,-2) -- (-2,-1);
\draw [-{Latex[length=2mm]},line width=1pt] (-3,-2) -- (-4,-1);
\draw [-{Latex[length=2mm]},line width=1pt] (-5,-2) -- (-4,-1);
\draw [-{Latex[length=2mm]},line width=1pt] (-5,-2) -- (-6,-1);
\draw (2,-1) node[anchor=north west] {=};
\draw [-{Latex[length=2mm]},line width=1pt,color=qqqqff] (5,3) -- (6,4);
\draw [-{Latex[length=2mm]},line width=1pt,color=ccqqqq] (7,3) -- (6,4);
\draw [-{Latex[length=2mm]},line width=1pt,color=qqqqff] (7,3) -- (8,4);
\draw [-{Latex[length=2mm]},line width=1pt,color=ccqqqq] (9,3) -- (8,4);
\draw [-{Latex[length=2mm]},line width=1pt,color=qqqqff] (9,3) -- (10,4);
\draw [-{Latex[length=2mm]},line width=1pt,color=ccqqqq] (11,3) -- (10,4);
\draw [-{Latex[length=2mm]},line width=1pt,color=ccqqqq] (9,1) -- (10,2);
\draw [-{Latex[length=2mm]},line width=1pt,color=qqqqff] (11,1) -- (10,2);
\draw [-{Latex[length=2mm]},line width=1pt] (6,0) -- (5,1);
\draw [-{Latex[length=2mm]},line width=1pt] (6,0) -- (7,1);
\draw [-{Latex[length=2mm]},line width=1pt] (8,0) -- (7,1);
\draw [-{Latex[length=2mm]},line width=1pt] (8,0) -- (9,1);
\draw [-{Latex[length=2mm]},line width=1pt] (10,0) -- (9,1);
\draw [-{Latex[length=2mm]},line width=1pt] (10,0) -- (11,1);
\draw [-{Latex[length=2mm]},line width=1pt] (12,0) -- (11,1);
\draw [-{Latex[length=2mm]},line width=1pt,color=qqqqff] (5,-1) -- (6,0);
\draw [-{Latex[length=2mm]},line width=1pt,color=ccqqqq] (7,-1) -- (6,0);
\draw [-{Latex[length=2mm]},line width=1pt,color=qqqqff] (7,-1) -- (8,0);
\draw [-{Latex[length=2mm]},line width=1pt,color=ccqqqq] (9,-1) -- (8,0);
\draw [-{Latex[length=2mm]},line width=1pt,color=qqqqff] (9,-1) -- (10,0);
\draw [-{Latex[length=2mm]},line width=1pt,color=ccqqqq] (11,-1) -- (10,0);
\draw [-{Latex[length=2mm]},line width=1pt,color=qqqqff] (11,-1) -- (12,0);
\draw [-{Latex[length=2mm]},line width=1pt,color=ccqqqq] (5,1) -- (6,2);
\draw [-{Latex[length=2mm]},line width=1pt,color=qqqqff] (7,1) -- (6,2);
\draw [-{Latex[length=2mm]},line width=1pt,color=ccqqqq] (7,1) -- (8,2);
\draw [-{Latex[length=2mm]},line width=1pt,color=qqqqff] (9,1) -- (8,2);
\draw [-{Latex[length=2mm]},line width=1pt] (18,2) -- (17,3);
\draw [-{Latex[length=2mm]},line width=1pt] (18,2) -- (19,3);
\draw [-{Latex[length=2mm]},line width=1pt] (20,2) -- (19,3);
\draw [-{Latex[length=2mm]},line width=1pt] (20,2) -- (21,3);
\draw [-{Latex[length=2mm]},line width=1pt] (22,2) -- (21,3);
\draw [-{Latex[length=2mm]},line width=1pt] (22,2) -- (23,3);
\draw [-{Latex[length=2mm]},line width=1pt,color=ccqqqq] (-6,-3) -- (-5,-2);
\draw [-{Latex[length=2mm]},line width=1pt,color=qqqqff] (-4,-3) -- (-5,-2);
\draw [-{Latex[length=2mm]},line width=1pt,color=ccqqqq] (-4,-3) -- (-3,-2);
\draw [-{Latex[length=2mm]},line width=1pt,color=qqqqff] (-2,-3) -- (-3,-2);
\draw [-{Latex[length=2mm]},line width=1pt,color=ccqqqq] (-2,-3) -- (-1,-2);
\draw [-{Latex[length=2mm]},line width=1pt,color=qqqqff] (0,-3) -- (-1,-2);
\draw [-{Latex[length=2mm]},line width=1pt,color=ccqqqq] (0,-3) -- (1,-2);
\draw [-{Latex[length=2mm]},line width=1pt] (-5,-4) -- (-6,-3);
\draw [-{Latex[length=2mm]},line width=1pt] (-5,-4) -- (-4,-3);
\draw [-{Latex[length=2mm]},line width=1pt] (-3,-4) -- (-4,-3);
\draw [-{Latex[length=2mm]},line width=1pt] (-3,-4) -- (-2,-3);
\draw [-{Latex[length=2mm]},line width=1pt] (-1,-4) -- (-2,-3);
\draw [-{Latex[length=2mm]},line width=1pt] (-1,-4) -- (0,-3);
\draw [-{Latex[length=2mm]},line width=1pt] (1,-4) -- (0,-3);
\draw [-{Latex[length=2mm]},line width=1pt,color=ccqqqq] (5,-3) -- (6,-2);
\draw [-{Latex[length=2mm]},line width=1pt,color=qqqqff] (7,-3) -- (6,-2);
\draw [-{Latex[length=2mm]},line width=1pt,color=ccqqqq] (7,-3) -- (8,-2);
\draw [-{Latex[length=2mm]},line width=1pt,color=qqqqff] (9,-3) -- (8,-2);
\draw [-{Latex[length=2mm]},line width=1pt,color=ccqqqq] (9,-3) -- (10,-2);
\draw [-{Latex[length=2mm]},line width=1pt,color=qqqqff] (11,-3) -- (10,-2);
\draw [-{Latex[length=2mm]},line width=1pt] (6,-4) -- (5,-3);
\draw [-{Latex[length=2mm]},line width=1pt] (6,-4) -- (7,-3);
\draw [-{Latex[length=2mm]},line width=1pt] (8,-4) -- (7,-3);
\draw [-{Latex[length=2mm]},line width=1pt] (8,-4) -- (9,-3);
\draw [-{Latex[length=2mm]},line width=1pt] (10,-4) -- (9,-3);
\draw [-{Latex[length=2mm]},line width=1pt] (10,-4) -- (11,-3);
\draw [-{Latex[length=2mm]},line width=1pt] (12,-4) -- (11,-3);
\draw [-{Latex[length=2mm]},line width=1pt] (18,-2) -- (17,-1);
\draw [-{Latex[length=2mm]},line width=1pt] (18,-2) -- (19,-1);
\draw [-{Latex[length=2mm]},line width=1pt] (20,-2) -- (19,-1);
\draw [-{Latex[length=2mm]},line width=1pt] (20,-2) -- (21,-1);
\draw [-{Latex[length=2mm]},line width=1pt] (22,-2) -- (21,-1);
\draw [-{Latex[length=2mm]},line width=1pt] (22,-2) -- (23,-1);
\draw [-{Latex[length=2mm]},line width=1pt] (24,-2) -- (23,-1);
\draw (13,-1) node[anchor=north west] {$\times$};
\draw [-{Latex[length=2mm]},line width=1pt,color=qqqqff] (-6,-5) -- (-5,-4);
\draw [-{Latex[length=2mm]},line width=1pt,color=ccqqqq] (-4,-5) -- (-5,-4);
\draw [-{Latex[length=2mm]},line width=1pt,color=qqqqff] (-4,-5) -- (-3,-4);
\draw [-{Latex[length=2mm]},line width=1pt,color=ccqqqq] (-2,-5) -- (-3,-4);
\draw [-{Latex[length=2mm]},line width=1pt,color=qqqqff] (-2,-5) -- (-1,-4);
\draw [-{Latex[length=2mm]},line width=1pt,color=ccqqqq] (0,-5) -- (-1,-4);
\draw [-{Latex[length=2mm]},line width=1pt,color=qqqqff] (0,-5) -- (1,-4);
\draw [-{Latex[length=2mm]},line width=1pt] (-5,-6) -- (-6,-5);
\draw [-{Latex[length=2mm]},line width=1pt] (-5,-6) -- (-4,-5);
\draw [-{Latex[length=2mm]},line width=1pt] (-3,-6) -- (-4,-5);
\draw [-{Latex[length=2mm]},line width=1pt] (-3,-6) -- (-2,-5);
\draw [-{Latex[length=2mm]},line width=1pt] (-1,-6) -- (-2,-5);
\draw [-{Latex[length=2mm]},line width=1pt] (-1,-6) -- (0,-5);
\draw [-{Latex[length=2mm]},line width=1pt] (1,-6) -- (0,-5);
\draw [-{Latex[length=2mm]},line width=1pt,color=qqqqff] (5,-5) -- (6,-4);
\draw [-{Latex[length=2mm]},line width=1pt,color=ccqqqq] (7,-5) -- (6,-4);
\draw [-{Latex[length=2mm]},line width=1pt,color=qqqqff] (7,-5) -- (8,-4);
\draw [-{Latex[length=2mm]},line width=1pt,color=ccqqqq] (9,-5) -- (8,-4);
\draw [-{Latex[length=2mm]},line width=1pt,color=qqqqff] (9,-5) -- (10,-4);
\draw [-{Latex[length=2mm]},line width=1pt,color=ccqqqq] (11,-5) -- (10,-4);
\draw [-{Latex[length=2mm]},line width=1pt,color=qqqqff] (11,-5) -- (12,-4);
\draw [-{Latex[length=2mm]},line width=1pt] (18,-6) -- (17,-5);
\draw [-{Latex[length=2mm]},line width=1pt] (18,-6) -- (19,-5);
\draw [-{Latex[length=2mm]},line width=1pt] (20,-6) -- (19,-5);
\draw [-{Latex[length=2mm]},line width=1pt] (20,-6) -- (21,-5);
\draw [-{Latex[length=2mm]},line width=1pt] (22,-6) -- (21,-5);
\draw [-{Latex[length=2mm]},line width=1pt] (22,-6) -- (23,-5);
\draw [-{Latex[length=2mm]},line width=1pt] (24,-6) -- (23,-5);
\draw [line width=1pt,dashed] (4,2.5)-- (10.5,2.5);
\draw [line width=1pt,dashed] (10.5,2.5)-- (10.5,-0.5);
\draw [line width=1pt,dashed] (10.5,-0.5)-- (11.5,-0.5);
\draw [line width=1pt,dashed] (11.5,-0.5)-- (11.5,-1.5);
\draw [line width=1pt,dashed] (11.5,-1.5)-- (10.5,-1.5);
\draw [line width=1pt,dashed] (10.5,-1.5)-- (10.5,-4.5);
\draw [line width=1pt,dashed] (10.5,-4.5)-- (11.5,-4.5);
\draw [line width=1pt,dashed] (11.5,-4.5)-- (11.5,-5.5);
\draw [line width=1pt,dashed] (11.5,-5.5)-- (4,-5.5);
\draw [line width=1pt,dashed] (4,-5.5)-- (4,2.5);
\draw [line width=1pt,dashed] (16.03618450638755,2.484732635153282)-- (22.5,2.5);
\draw [line width=1pt,dashed] (22.5,2.5)-- (22.5,-0.5);
\draw [line width=1pt,dashed] (22.5,-0.5)-- (23.5,-0.5);
\draw [line width=1pt,dashed] (23.5,-0.5)-- (23.5,-1.5);
\draw [line width=1pt,dashed] (23.5,-1.5)-- (22.5,-1.5);
\draw [line width=1pt,dashed] (22.5,-1.5)-- (22.5,-4.5);
\draw [line width=1pt,dashed] (22.5,-4.5)-- (23.5,-4.5);
\draw [line width=1pt,dashed] (23.5,-4.5)-- (23.5,-5.5);
\draw [line width=1pt,dashed] (23.5,-5.5)-- (16,-5.5);
\draw [line width=1pt,dashed] (16,-5.5)-- (16,2.5);
\draw [-{Latex[length=2mm]},line width=1pt] (-1,0) -- (0,1);
\draw [-{Latex[length=2mm]},line width=1pt] (-2,-1) -- (-1,0);
\draw [-{Latex[length=2mm]},line width=1pt] (0,-1) -- (-1,0);
\draw [-{Latex[length=2mm]},line width=1pt] (-4,1) -- (-3,2);
\draw [-{Latex[length=2mm]},line width=1pt] (-2,1) -- (-3,2);
\draw [-{Latex[length=2mm]},line width=1pt] (-2,1) -- (-1,2);
\draw [-{Latex[length=2mm]},line width=1pt] (1,0) -- (0,1);
\draw [line width=1pt,dashed] (-0.5,2.5)-- (-7,2.5);
\draw [line width=1pt,dashed] (-7,2.5)-- (-7,-5.5);
\draw [line width=1pt,dashed] (-7,-5.5)-- (0.5,-5.5);
\draw [line width=1pt,dashed] (0.5,-5.5)-- (0.5,-4.5);
\draw [line width=1pt,dashed] (0.5,-4.5)-- (-0.5,-4.5);
\draw [line width=1pt,dashed] (-0.5,-4.5)-- (-0.5,-1.5);
\draw [line width=1pt,dashed] (-0.5,-1.5)-- (0.5,-1.5);
\draw [line width=1pt,dashed] (0.5,-1.5)-- (0.5,-0.5);
\draw [line width=1pt,dashed] (0.5,-0.5)-- (-0.5,-0.5);
\draw [line width=1pt,dashed] (-0.5,-0.5)-- (-0.5,2.5);
\begin{scriptsize}
\end{scriptsize}
\end{tikzpicture}
    \caption{Proof Scheme: The Gibbs measure on $\Theta_{2,4}$ domain (left figure) can be decomposed into two parts: One is the combination of the top colored row and 2 $\irw$s (middle figure) and two are the remaining black weights (right figure) which will be viewed as a Radon-Nikodym derivative. Here note that in the middle figure, the only contribution from the top row comes from the odd points, $H_N^{(1)}(2j-1)$ for $j \in \ll 1, T\rr$, which are set to $\infty$. Thus, their contribution to \eqref{e:wb} from \eqref{eq:gibbs} would be $\exp(-e^{-\infty}) = 1.$ 
    }
    \label{figsm2}
\end{figure}

\medskip

\noindent\textbf{Step 4.} Finally in this step, we provide an upper bound for the right-hand side of \eqref{e:gibb2} by bounding its numerator and denominator separately. Let us consider the event:
\begin{align*}
    \m{G} & := \bigcap_{i=1}^k\left\{\sup_{p\in \ll1,2T-1\rr}|L_{2i}(p)-x_{2i,2T-1}|+\sup_{q\in \ll1,2T\rr}|L_{2i+1}(q)-x_{2i,2T-1}|\le M_0\sqrt{T}\right\}.
\end{align*}
where $M_0$ comes from Proposition \ref{p:irw}. From the description of the Gibbs measure, it is clear that if $(L_{2i}(\cdot),L_{2i+1}(\cdot))$ is an $\irw$ with boundary condition $(x_{2i,2T-1},x_{2i,2T-1}-\sqrt{T})$, then $(L_{2i}(\cdot)-x_{2i,2T-1}, L_{2i+1}(\cdot)-x_{2i,2T-1})$ is an $\irw$ with boundary condition $(0,-\sqrt{T})$. Thus, appealing to Proposition \ref{p:irw}, we see that $$\pc^{\vec{x}}(\m{G}) \ge (1-\e)^{k} \ge 1-k\e.$$ Let us assume $N$ is large enough so that $\sqrt{N}-2M_0\sqrt{T} \ge \frac12\sqrt{N}$ (recall $T=O(\sqrt{N})$). Observe that under the event $\m{G}$, we have for all $p\le 2T-1$ $$L_{2}(p) \le x_{2,2T-1}+M_0\sqrt{T}=RN-(M_2+1)\sqrt{N}+M_0\sqrt{T} \le RN-M_2\sqrt{N}.$$ Thus, $\m{E}$ defined in \eqref{e:tshw} holds. This implies $\neg \m{E} \subset \neg\m{G}$. Hence
\begin{align}\label{e:tu}
    \ec^{\vec{x}}[\wb \ind_{\neg \m{E}}] \le \pc^{\vec{x}}(\neg \m{E}) \le \pc^{\vec{x}}(\neg \m{G}) \le k\e.
\end{align}
On the event $\m{G}$, for all $p\in \ll1,2T-1\rr$
 and $q\in \ll1,2T \rr$ we have
$$L_{2i+2}(p)-L_{2i+1}(q) \le 2M_0\sqrt{T}+x_{2i+2,2T-1}-x_{2i,2T-1} =2M_0\sqrt{T}-\sqrt{N} \le -\tfrac12\sqrt{N}.$$
This implies $\wb \ge \exp(-k(2T-1)e^{-\frac12\sqrt{N}}) \ge (1-\e)$ for large enough $N$ (recall $T=O(\sqrt{N})$) on the event $\m{G}$. 
Thus, $\ec^{\vec{x}}[\wb] \ge (1-\e)\pc^{\vec{x}}(\m{G}) \ge (1-\e)^{k+1}$. Inserting this bound and the bound in \eqref{e:tu} back in \eqref{e:gibb2} we get that $\pg^{\vec{x}}(\neg\m{E})\le \frac{k\e}{(1-\e)^{k+1}}.$ Combining this bound with \eqref{eq:stoc} yields \eqref{nebd}. This completes the proof.
\end{proof}

\section{Proof of main theorems}\label{sec:mainpf}

In this section, we prove our main theorems, Theorems \ref{t:bdpt}, \ref{t:walk}, \ref{t:qdistn}, and \ref{t:fluc}. \sd{In Section \ref{sec.2tech}, we first present a few supporting technical results. Next in Section \ref{sec.maintech} we complete the proof of our main theorems by assuming a technical proposition (Proposition \ref{p:tech}) which is proved in Section \ref{sec:tech}.} 

\subsection{Preparatory lemmas} \label{sec.2tech}
The first lemma settles a weaker version of Theorem \ref{t:bdpt}. Recall the polymer measure $\Pr^{W}$ from \eqref{eq:Gibbs}, the partition function $\zh(m,n)$ from \eqref{eq:hpp}, and the $\hslg$ line ensemble $H_N$ from Definition \ref{hslg}. Note that the quenched distribution of the endpoint of the polymer is related via
\begin{align}
    \label{eq:bdpt}
    \Pr^{W}(\pi(2N-2) = N-r)= \frac{\zh(N+r,N-r)}{\sum_{p=0}^{N-1} \zh(N+p,N-p)}=\frac{e^{H_N^{(1)}(2r+1)}}{\sum_{p=0}^{N-1} e^{H_N^{(1)}(2p+1)}}.
\end{align}
where the second equality follows from the relation \eqref{eq:rem}. Recalling  $\zt(k) = \sum_{p=k}^{N-1} e^{H_N^{(1)}(2p+1)}$ from \eqref{defzt}, we obtain $$\Pr^{W}(\pi(2N-2) \le N-k)=\frac{\zt(k)}{\zt(0)}.$$ Theorem \ref{t:bdpt} claims that this quenched probability decays as $N\to \infty$ followed by $k\to \infty$. Lemma \ref{l:deep} takes $k=\lfloor M\sqrt{N}\rfloor$. For notational convenience, we assume all the multiples of $\sqrt{N}$ appearing in the proofs in this section are integers. The general case follows verbatim by considering the floor function. 

\begin{lemma}\label{l:deep}
 Fix $\e > 0$ and recall $\zt(\cdot)$ from Theorem \ref{t:bw}. There exist constants $M(\e)>0, N_1(\e)>0$ such that for all $N\ge N_1$,
    \begin{align}\label{eq:ldeep}
\Pr\left(\frac{ \zt (M\sqrt{N})}{\zt (0)}\le e^{-\sqrt{N}}\right) \ge 1-\tfrac12\e.
    \end{align}
\end{lemma}

\begin{proof} Fix $\e \in (0, 1)$.  Recall $\sigma$ from \eqref{def:dk} Taking $g=1$ and $g=M\sqrt{N}$ in Theorem \ref{t:bw} yields
 \begin{align*}
    \frac{1}{\sigma\sqrt{N}}\left[\log \zt(1)-RN\right] \stackrel{d}{\to} \mathcal{N}\big(0,1\big),
\end{align*} 
  \begin{align}\label{eq:gauconv}
    \frac{1}{\sigma\sqrt{N}}\left[\log \zt(M\sqrt{N})-RN+M\tau\sqrt{N}\right] \stackrel{d}{\to} \mathcal{N}\big(0,1\big)
\end{align}  
respectively, where $R, \sigma, \tau$ are defined in \eqref{def:dk}. Let us set $P:= P(\e) = \Phi^{-1}(1- \tfrac{\e}{8}) + 1$, where $\Phi(\cdot)$ is the cumulative distribution function of $\mathcal{N}(0,1)$. For all large enough $N$ we have
\begin{align*}
    \Pr\left(\log \zt(0)\ge RN -P\sigma\sqrt{N}\right) \ge \Pr\left(\log \zt(1)\ge RN -P\sigma\sqrt{N}\right) \ge 1-\tfrac{\e}4,
\end{align*}
\begin{align*}
    \Pr\left(\log \zt(M\sqrt{N})\le RN-M\tau\sqrt{N}+ P\sigma\sqrt{N}\right) \ge 1-\tfrac{\e}4.
\end{align*}
Applying a union bound gives us
\begin{align*}
    \Pr\left(\log \zt(M\sqrt{N})+(M\tau-2P\sigma)\sqrt{N}\le \log\zt(0)\right) \ge 1-\tfrac\e2,
\end{align*}
for all large enough $N$. Taking $M:=\frac1\tau(2P\sigma +1)$ in above equation leads to \eqref{eq:ldeep}. This completes the proof. 
\end{proof}

\sd{Recall our discussion in Section \ref{sec1.2.4} and Figure \ref{f.tri1}. We refer to the region $\ll N-M\sqrt{N}, N-k\rr$ and the region $\ll 1, N-M\sqrt{N}\rr$ as the shallow tail and deep tail regions respectively (see Figure \ref{f.tri1}). Lemma \ref{l:deep} implies that with high probability the quenched probability of $\pi(2N-2)$ living in the deep tail region is exponentially small. Thus the mass accumulates in the window of size $M\sqrt{N}$ below the point $(N, N)$. To establish Theorem \ref{t:bdpt}, we thus have to show the mass in the shallow tail also goes to zero. For convenience, in our proofs below we shall often refer to the point $(N+M\sqrt{N}, N-M\sqrt{N})$ as the {deep tail starting point}. Given the connection in \eqref{eq:rem}, the {deep tail starting point} corresponds to $(2M\sqrt{N}+1)$-th point for the top curve $H_N^{(1)}(\cdot)$ of the $\hslg$ line ensemble. So, in the coordinates of the $\hslg$ line ensemble, we shall refer to $2M\sqrt{N}+1$ as the deep tail starting point.}

\medskip

Below, we record another important preparatory lemma which claims the existence of a ``high point" in $H_N^{(1)}(\cdot)$ not far after the deep tail starting point (see Figure \ref{fig:1stcurve}). 

\begin{lemma}\label{l:high}
     Fix any $\e > 0$ and recall $R, \tau$ from \eqref{def:dk}. There exists a constant $M_0(\e)>0$ such that for all $M \ge M_0$, there exists $N_0(\e,M)$ such that for all $N\ge N_0,$
     \begin{align}\label{eq:highp}
     \Pr\left(\sup_{p\in \ll M\sqrt{N}, 2M\sqrt{N}\rr} H_N^{(1)}(2p+1) \ge RN- \tfrac52M\tau\sqrt{N}\right) \ge 1- \tfrac12\e,
    \end{align}
where $\tau:=\Psi(\theta-\alpha)-\Psi(\theta+\alpha).$
\end{lemma}
\begin{proof}
Let us set $P:= P(\e) = \Phi^{-1}(1- \tfrac{\e}{6}) + 1$, where $\Phi(\cdot)$ is the cumulative distribution function of $\mathcal{N}(0,1)$. In view of \eqref{eq:gauconv}, for all large enough $N$ we have
\begin{align}\label{eq:gauprob}
    \Pr\left(\log \zt(M\sqrt{N})\ge RN-M\tau\sqrt{N} -P\sigma\sqrt{N}\right) \ge 1-\tfrac{\e}6,
\end{align}
\begin{align*}
    \Pr\left(\log \zt(2M\sqrt{N})\le RN-2M\tau\sqrt{N}+ P\sigma\sqrt{N}\right) \ge 1-\tfrac{\e}6.
\end{align*}
Applying a union bound gives us
\begin{align*}
    \Pr\left(\log \zt(2M\sqrt{N})+(M\tau-2P\sigma)\sqrt{N}\le \log\zt(M\sqrt{N})\right) \ge 1-\tfrac{\e}3.
\end{align*}

 Thus for any $M \ge \frac{2P\sigma+1}{\tau}$, we have that with probability at least $1-\tfrac{\e}3$, $\log \zt(2M\sqrt{N}) \le \log\zt(M\sqrt{N})- \sqrt{N}$, which implies $$2\zt(2M\sqrt{N})\le \zt(M\sqrt{N}).$$ However, by definition of $\zt(\cdot)$, the above display implies that with probability at least $1-\tfrac{\e}3$,
 \begin{align*}
\sup_{p\in \ll M\sqrt{N},2M\sqrt{N}\rr} H_N^{(1)}(2p+1) &\ge   \log \left[\tfrac1{2M\sqrt{N}}(\zt(M\sqrt{N})- \zt(2M\sqrt{N}))\right]\\&\ge \log \zt(2M\sqrt{N})-\log(2M\sqrt{N}).
\end{align*}
Note that by the first entry in \eqref{eq:gauprob} with $M$ substituted by $2M$, with probability at least $1-\frac\e6$, we have $\log \zt(2M\sqrt{N})\ge RN-2M\tau\sqrt{N}- P\sigma\sqrt{N}$. For all large enough $N$, we have $RN-(2M\tau+P\sigma)\sqrt{N}-\log(2M\sqrt{N}) \ge RN-\tfrac52M\tau \sqrt{N}$. Thus applying another union bound helps us arrive at \eqref{eq:highp} and complete the proof.   
\end{proof}

\subsection{Proof of Theorems \ref{t:bdpt}, \ref{t:walk}, \ref{t:qdistn}, and \ref{t:fluc}} \label{sec.maintech}
\sd{In this section, we prove our main theorems assuming a technical proposition, described as follows.} Fix any $M,N\ge 1$ and assume  $M\sqrt{N} \in \Z_{>0}$. For any Borel set  $A$ of $\R^{M\sqrt{N}}$ we consider the event
\begin{align}\label{defeva}
    \m{A} = \left\{(H_N^{(1)}(1)-H_N^{(1)}(2r+1))_{r=1}^{M\sqrt{N}}\in A\right\}.
\end{align}
for $N>M^2+1$. Let $(S_r)_{r=0}^{M\sqrt{N}}$ be the log-gamma random walk defined in Definition \ref{def:lgrw}. We write
\begin{align}\label{def:rwev}
    \Pr_{RW}(\m{A}):= \Pr\left( (S_r)_{r=1}^{M\sqrt{N}} \in A\right)
\end{align}

The technical proposition below is the main crux of our proofs and it claims that $\Pr$ and $\Pr_{RW}$ are close to each other when $N$ is large. We postpone its proof to Section \ref{sec:tech}.

\begin{proposition}\label{p:tech} Fix any $\e \in (0,\frac12)$. Set $M(\e)>0, N_1(\e)>0$ such that Lemma \ref{l:deep} and Lemma \ref{l:high} hold simultaneously for all $N \ge N_1$ for this fixed choice of $M$. Then there exists $N_0(\e)>0$ such that for all $N\ge N_0$,
\begin{align}\label{eq:tech}
    |\Pr(\m{A})-\Pr_{RW}(\m{A})| \le 9\e,
\end{align}
where $\m{A}$ and $\Pr_{RW}(\m{A})$ are defined in \eqref{defeva} and \eqref{def:rwev}.
\end{proposition}

Given these results, we are ready to prove our main theorems. Theorems \ref{t:walk} and \ref{t:bdpt} are direct applications of the supporting lemmas. For convenience, we shall assume in the proofs below $M\sqrt{N}$ is an integer. The general case follows verbatim by considering floor functions.

{\begin{proof}[Proof of Theorem \ref{t:fluc}] Given \eqref{e:gaufl}, it suffices to check that 
\begin{align*}
    \tfrac{1}{\sqrt{N}}\left(\log Z(N,N)-\log Z(N+a_N,N-a_N)\right) \stackrel{p}{\to} 0,
\end{align*}
where $\{a_N\}_{N\ge 1}$ is a sequence of nonnegative integers less than $N$, with $a_N/\sqrt{N}\to 0$. In light of \eqref{eq:rem}, it boils down to checking
\begin{align*}
    \tfrac{1}{\sqrt{N}}\left(H_N^{(1)}(1)-H_N^{(1)}(2a_N+1)\right) \stackrel{p}{\to} 0.
\end{align*}
But thanks to Proposition \ref{p:tech}, it is equivalent to argue that  $S_{a_N}/{\sqrt{N}} \stackrel{p}{\to} 0$ where $(S_r)_{r\ge 0}$ is the log-gamma random walk defined in Definition \ref{def:lgrw}. Since the increment of the walk has the finite first moment and $a_N/\sqrt{N} \to 0$, by Markov inequality we deduce that $S_{a_N}/{\sqrt{N}} \stackrel{p}{\to} 0$. This establishes Theorem \ref{t:fluc}.
\end{proof}}

\begin{proof}[Proof of Theorem \ref{t:walk}] Take the set $A$ as $(-\infty,x_1]\times (-\infty,x_2]\times \cdots \times (-\infty,x_k]\times \R^{M\sqrt{N}-k}$ in \eqref{defeva}. By Proposition \ref{p:tech}, 
\begin{align*}
    \limsup_{N\to \infty} \left|\Pr\left(\bigcap_{r=1}^k \{H_N^{(1)}(1)-H_N^{(1)}(2r+1)\in (-\infty,x_r]\}\right)-\Pr_{RW}\left(\bigcap_{r=1}^k \{S_r\in (-\infty,x_r]\}\right)\right| \le 9\e,
\end{align*}
where $(S_r)_{r=0}^k$ is defined in Definition \eqref{def:lgrw}.
As $\e$ is arbitrary, this implies
\begin{align*}
 \left( H_N^{(1)}(1)-H_N^{(1)}(2r+1)\right)_{r=0}^k \stackrel{d}{\to} (S_r)_{r=0}^{k}.
\end{align*}
 In conjunction with the relation \eqref{eq:rem}, we get the desired convergence in Theorem \ref{t:walk}.
\end{proof}

\begin{proof}[Proof of Theorem \ref{t:bdpt}] Fix any $\e>0$. Get $M(\e), N_1(\e)>0$ such that Lemma \ref{l:deep} and Lemma \ref{l:high} hold simultaneously for all $N\ge N_1$ for this fixed choice of $M$. Using this $M$ we split the probability as follows 
\begin{equation*}
    \begin{aligned}
 &\Pr^{W}(\pi(2N-2) \le N- k) \\ & \hspace{1cm} = \Pr^{W}\big(\pi(2N-2) \in (N-M\sqrt{N}, N-k]\big)+\Pr^{W}\big(\pi(2N-2) \le N- M\sqrt{N}\big).
\end{aligned}
\end{equation*}
For the first term observe that by \eqref{eq:bdpt}
\begin{align*}
 \Pr^{W}\big(\pi(2N-2) \in (N-M\sqrt{N}, N-k]\big)  & =\frac{\sum_{p =k}^{M\sqrt{N}-1} e^{H_N^{(1)}(2p+1)}}{\sum_{p = 0}^{N-1} e^{H_N^{(1)}(2p+1)}} \\ &\le  \frac{\sum_{p =k}^{ M\sqrt{N}} e^{H_N^{(1)}(2p+1)}}{\sum_{p = 0}^{M\sqrt{N}} e^{H_N^{(1)}(2p+1)}}  =  \frac{\sum_{p = k}^{ M\sqrt{N}} e^{H_N^{(1)}(2p+1)-H_N^{(1)}(1)}}{\sum_{p = 0}^{M\sqrt{N}} e^{H_N^{(1)}(2p+1)-H_N^{(1)}(1)}}.
\end{align*}
Fix any $\delta>0$ and consider the set
\begin{align*}
    \m{A}_{\delta}:=\left\{\frac{\sum_{p = K}^{ M\sqrt{N}} e^{H_N^{(1)}(2p+1)-H_N^{(1)}(1)}}{\sum_{p = 1}^{M\sqrt{N}} e^{H_N^{(1)}(2p+1)-H_N^{(1)}(1)}}\ge \delta\right\}.
\end{align*}
By Proposition \ref{p:tech}, $\Pr(\m{A}_{\delta})\le \Pr_{RW}(\m{A}_{\delta})+9\e$ for all large enough $N$. On the other hand, by Corollary \ref{l:cor1} we see that
$\lim_{k\to\infty}\lim_{N\to\infty}\Pr_{RW}(\m{A}_{\delta})=0$. Thus, as $\e$ is arbitrary, 
\begin{align}\label{eq:bdpt1}
\lim_{k\to\infty}\lim_{N\to\infty}\Pr^{W}\big(\pi(2N-2) \in (N-M\sqrt{N}, N-k]\big)=0, \mbox{ in probability}.
\end{align}
For the second term by Lemma \ref{l:deep}, we see that with probability $1-\frac\e2$
\begin{align*}
\Pr^{W}\big(\pi(2N-2) \le N- M\sqrt{N}\big) =\frac{\zt(M\sqrt{N})}{\zt(0)}\le  e^{-\sqrt{N}}.\end{align*} 
Again, as $\e$ is arbitrary, we have that as $N\to \infty$, $\Pr^{W}(\pi(2N-2)\le N-M\sqrt{N}) \to 0$ in probability. This completes the proof together with \eqref{eq:bdpt1}. 
\end{proof}
Lastly, with Theorems \ref{t:bdpt} and \ref{t:walk} established, we present the proof of the limiting quenched distribution of the endpoint viewed from around the diagonal.
\begin{proof}[Proof of Theorem \ref{t:qdistn}]
Fixed $\theta > 0$, $\alpha \in (-\theta, 0)$, and $M,k\in \Z_{>0}.$ with $M \ge k$. For each $r\in \ll0,k\rr$ \bl{We write
\begin{align}\label{eq:wtail}
    \Pr^{W}(\pi(2N-2) = N-r)= \Pr^W(\pi(2N-2)\ge N-M)\cdot \frac{\Pr^{W}(\pi(2N-2) = N-r)}{\Pr^W(\pi(2N-2)\ge N-M)}
\end{align}
and study weak convergence (under $N \rightarrow \infty$ followed by $M \rightarrow \infty$) of two quantities on the right side of the above equation separately. Firstly, Theorem \ref{t:bdpt} ensures that 
\begin{align*}
    \Pr^W(\pi(2N-2)\ge N-M)= 1 - \Pr^W(\pi(2N-2)< N-M)\stackrel{p}{\to}1
\end{align*}
as $N \rightarrow \infty$ followed by $M \rightarrow \infty.$ On the other hand, for the ratio in \eqref{eq:wtail}, recall from \eqref{eq:bdpt} that
\begin{align}\label{eq:qd1}
    \frac{\Pr^{W}(\pi(2N-2) = N-r)}{\Pr^W(\pi(2N-2)\ge N-M)} &= \frac{e^{H_N^{(1)}(2r+1)}}{\sum_{p=0}^{M} e^{H_N^{(1)}(2p+1)}} =  \frac{e^{H_N^{(1)}(2r+1)-H_N^{(1)}(1)}}{\sum_{p=0}^{M} e^{H_N^{(1)}(2p+1)-H_N^{(1)}(1)}}.
\end{align}
 Note that by Theorem \ref{t:walk}, a continuous mapping theorem immediately implies that 
\begin{align}\label{eq:qd2}
    \left(\frac{e^{H_N^{(1)}(2r+1)-H_N^{(1)}(1)}}{\sum_{p=0}^{M} e^{H_N^{(1)}(2p+1)-H_N^{(1)}(1)}}\right)_{r\in \ll0,k\rr} \stackrel{d}{\to}\left(\frac{e^{-S_r}}{\sum_{p=0}^{M}e^{-S_p}}\right)_{r\in \ll0,k\rr}.
\end{align}
 Here $(S_i)_{i\ge 0}$ denotes a log-gamma random walk. Upon taking $M\to \infty$, in view of lemma \ref{a1}, the right hand side of the above equation converges to  $\left(\frac{e^{-S_r}}{\sum_{p=0}^{\infty}e^{-S_p}}\right)_{r\in \ll0,k\rr}$. This proves the theorem.}
\end{proof}

\subsection{Proof of Proposition \ref{p:tech}}\label{sec:tech}

For clarity, we divide the proof into several steps. 

\medskip

\noindent\textbf{Step 1.} In this step we sketch out the main ideas behind the proof. At this point, we encourage the readers to consult with Figure \ref{fig:1stcurve}. Recall the event $\m{A}$ defined in \eqref{defeva}.

\begin{figure}[h!]
    \centering
    \definecolor{uuuuuu}{rgb}{0.26666666666666666,0.26666666666666666,0.26666666666666666}
\definecolor{blue}{rgb}{0,0,1}
\definecolor{ffqqqq}{rgb}{1,0,0}
\definecolor{ududff}{rgb}{0.30196078431372547,0.30196078431372547,1}
\begin{tikzpicture}[line cap=round,line join=round,>=triangle 45,x=1cm,y=1cm]
\draw [line width=1pt,dashed] (-4,2)-- (8,2);
\draw [line width=1pt,dashed] (2.819383230896712,5.636570400030574)-- (2.8406270484351963,-2.032447731362127);
\draw [line width=1pt,dashed] (-3.9998821989566804,5.61532658249209)-- (-4,-2);
\draw [line width=1pt,dashed] (7.195609643824435,5.594082764953606)-- (7.195609643824435,-1.9687162787466754);
\draw [line width=2pt] (-3.9998821989566804,5.61532658249209)-- (-3.3625676728021543,4.765573880952732);
\draw [line width=2pt] (-3.3625676728021543,4.765573880952732)-- (-3.0104647504494744,3.804711197247093);
\draw [line width=2pt] (-3.0104647504494744,3.804711197247093)-- (-2.1108296658213477,4.364484138793478);
\draw [line width=2pt] (-2.1108296658213477,4.364484138793478)-- (-1.471089161196902,3.7447355249385517);
\draw [line width=2pt] (-1.471089161196902,3.7447355249385517)-- (-0.7713729842639144,4.664362500336185);
\draw [line width=2pt] (-0.7713729842639144,4.664362500336185)-- (0,4.304508466484937);
\draw [line width=2pt] (0,4.304508466484937)-- (0.6080674788325467,4.144573340328827);
\draw [line width=2pt] (0.6080674788325467,4.144573340328827)-- (1.1878323111484508,3.3448977095482766);
\draw [line width=2pt] (1.1878323111484508,3.3448977095482766)-- (1.6276539080777572,3.824703088016607);
\draw [line width=2pt] (1.6276539080777572,3.824703088016607)-- (2.247402521932689,3.5248247264739003);
\draw [line width=2pt] (2.247402521932689,3.5248247264739003)-- (2.861870865973689,3.873333544336407);
\draw [line width=2pt] (2.861870865973689,3.873333544336407)-- (3.5268835311815807,3.564808508012928);
\draw [line width=2pt] (3.5268835311815807,3.564808508012928)-- (4.006688909649915,3.4648490541653594);
\draw [line width=2pt] (4.006688909649915,3.4648490541653594)-- (4.566461851196305,3.424865272626332);
\draw [line width=2pt] (4.566461851196305,3.424865272626332)-- (5,3);
\draw [line width=2pt] (5,3)-- (5.526072608132973,2.6451815326152954);
\draw [line width=2pt] (5.526072608132973,2.6451815326152954)-- (6.08585,2.66517);
\draw [line width=2pt] (6.08585,2.66517)-- (6.4520760299774915,2.0676090535652722);
\draw [line width=2pt] (6.4520760299774915,2.0676090535652722)-- (6.834464745670205,1.7701956080264973);
\draw [line width=2pt] (6.834464745670205,1.7701956080264973)-- (7.216853461362919,1.557757432641658);
\draw [line width=2pt] (7.216853461362919,1.557757432641658)-- (7.6417298121326,1.3878068923337865);
\draw [line width=1pt,dashed, color=black] (-3.9999627106987155,0.41059161542833245)-- (7.960387075209875,0.38934746802504144);
\draw [line width=2pt,dashed] (-3.999994586778187,-1.6500586853115191)-- (-3.468786760494575,-1.3101579350536734);
\draw [line width=2pt] (-3.999994586778187,-1.6500586853115191)-- (-3.468786760494575,-1.3101579350536734);
\draw [line width=2pt] (-3.468786760494575,-1.3101579350536734)-- (-3.0651542272633767,-1.3526455701306412);
\draw [line width=2pt] (-3.0651542272633767,-1.3526455701306412)-- (-2.5128149712627903,-0.8640377667455108);
\draw [line width=2pt] (-2.5128149712627903,-0.8640377667455108)-- (-2.087938620493108,-1.140207394745802);
\draw [line width=2pt] (-2.087938620493108,-1.140207394745802)-- (-1.5780869995694895,-0.7153310439761232);
\draw [line width=2pt] (-1.5780869995694895,-0.7153310439761232)-- (-1.1956982838767756,-1.0764759421303502);
\draw [line width=2pt] (-1.1956982838767756,-1.0764759421303502)-- (-0.6008713927992205,-0.6940872264376393);
\draw [line width=2pt] (-0.6008713927992205,-0.6940872264376393)-- (0,-0.460405233514316);
\draw [line width=2pt] (0,-0.460405233514316)-- (0.35510039643256425,-0.4391614159758321);
\draw [line width=2pt] (0.35510039643256425,-0.4391614159758321)-- (0.8224643822792146,-0.5241366861297678);
\draw [line width=2pt] (0.8224643822792146,-0.5241366861297678)-- (1.6084856312031266,-0.33294232828341236);
\draw [line width=2pt] (1.6084856312031266,-0.33294232828341236)-- (2.2033125222806813,-0.4391614159758321);
\draw [line width=2pt] (2.2033125222806813,-0.4391614159758321)-- (2.8354480508214768,-0.16282959280933085);
\draw [line width=2pt] (2.8354480508214768,-0.16282959280933085)-- (3.2230157641279185,-0.4179175984373481);
\draw [line width=2pt] (3.2230157641279185,-0.4179175984373481)-- (4,0);
\draw [line width=2pt] (4,0)-- (4.73132680936029,-0.31169851074492844);
\draw [line width=2pt] (4.73132680936029,-0.31169851074492844)-- (5.283666065360877,-0.14174797043705695);
\draw [line width=2pt] (5.283666065360877,-0.14174797043705695)-- (5.9422244090538845,0.26188456279413785);
\draw [line width=2pt] (5.9422244090538845,0.26188456279413785)-- (6.515807482592955,-0.12050415289857301);
\draw [line width=2pt] (6.515807482592955,-0.12050415289857301)-- (7.0681467385935415,0.2406407452556539);
\draw [line width=2pt] (7.0681467385935415,0.2406407452556539)-- (7.7904365349020015,0.4955665557174611);
\draw [color=black](-1,1.1) node[anchor=north west] {$y=RN-(3M\tau+1)\sqrt{N}$};
\draw (1.0696122714816316,-0.3525876286455738) node[anchor=north west] {$H_{N}^{(2)}(\cdot)$};
\draw (0.705577321333596,5.150764264768837) node[anchor=north west] {$H_{N}^{(1)}(\cdot)$};
\draw (-1,2.7) node[anchor=north west] {$y= RN- 3M\tau\sqrt{N}$};
\draw (2.3330276867012856,-1.7444859674468838) node[anchor=north west] {$2M\sqrt{N}+1$};
\draw (6.315998317732736,-1.765899788043827) node[anchor=north west] {$4M\sqrt{N}+1$};
\draw [color=black](2.076061839537966,4.6) node[anchor=north west] {Deep tail starting point};
\draw [color=black](5.4594454938550045,3.4) node[anchor=north west] {High point};
\begin{scriptsize}
\draw [fill=blue] (6.08585,2.66517) circle (4pt);
\draw [fill=red] (2.861870865973689,3.8733335443364076) circle (4pt);
\end{scriptsize}
\end{tikzpicture}
    \caption{Illustration of the proof of Proposition \ref{p:tech}. As claimed by Lemma \ref{l:high}, there exists a high point $2p^*+1$ in $\ll 2M\sqrt{N}+1, 4M\sqrt{N}+1\rr$ such that $H_N^{(1)}(2p^*+1)$ lies above $RN - \frac52M\tau\sqrt{N}$ with high probability. \sd{This high point is illustrated in blue in the figure and helps us show that $H_N^{(1)}(\cdot)\ge RN - 3M\tau\sqrt{N}$ between $\ll1, 2p^*+1\rr.$  Meanwhile, invoking Proposition \ref{p:benchmark}, we can ensure the second curve stays below the benchmark of $RN-(3M\tau+1)\sqrt{N}$ on the interval $\ll1,4M\sqrt{N}+1\rr$ with high probability. Thus there is a $\sqrt{N}$ separation (with high probability) between the two curves. By the Gibbs property, this separation ensures that the top curve is close to a log-gamma random walk.} 
    }
    \label{fig:1stcurve}
\end{figure}
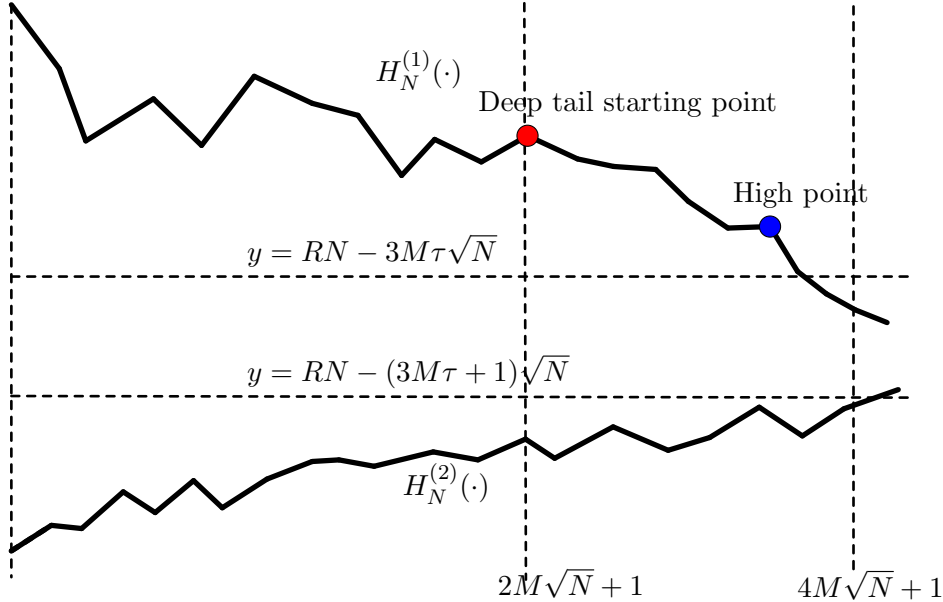



\begin{itemize}[leftmargin=18pt]
\setlength\itemsep{0.5em}
\item We take $M$ and $N_1$ as described in the statement of the Proposition \ref{p:tech}. \sd{In the language introduced in Figure \ref{f.tri1} and the text before Lemma \ref{l:high}, $2M\sqrt{N}+1$ is the \textit{deep tail starting point}. Assuming Lemma \ref{l:high}, we have a point $2p^*+1 \in \ll 2M\sqrt{N}+1,4M\sqrt{N}+1\rr$ where $H_N^{(1)}(2p^*+1)$ is `high' enough (see Figure \ref{fig:1stcurve}). This high point event is denoted as event $\m{B}$ in \textbf{Step 2} and has a probability of at least $1-\tfrac12\e$.} 

\item Invoking Proposition \ref{p:benchmark} with $M_1 = 2M$ and $M_2 = 3M\tau + 1$, with high probability the second curve of the line ensemble is lower than a certain benchmark, i.e. $$\sup_{p\in \ll1,4M\sqrt{N}+1\rr} H_N^{(2)}(p)\le RN-(3M\tau+1)\sqrt{N}$$ 
with probability at least $1 - \frac{\e}{2}$.  We denote this phenomenon as the $\m{Fluc}$ event. As $\m{B}$ and $\m{Fluc}$ are high probability events, it suffices to show that $|\Pr(\m{A}\cap \m{B}\cap\m{Fluc})-\Pr_{RW}(\m{A})|$ is small to prove \eqref{eq:tech}. \sd{This result is achieved by considering the measure conditioned on the entire second curve and the first curve beyond $2p^*+1$. We remark that this is only a formal description of the proof and refer to the last bullet point for details.}

    \item \sd{To elaborate on the above idea, by the Gibbs property in Theorem \ref{t:gibbs}, we introduce an explicity Radon-Nikodym derivative $W_{p^*}$ for the conditional measure in \textbf{Step 3}. Informally,  the conditional measure is absolutely continuous w.r.t.~a log-gamma random walk $(S_k)_{k\ge 0}$ from Definition \ref{def:lgrw} starting at $H_N^{(1)}(2p^*+1)$. As the free law is precisely the limiting law we are interested in, it suffices to prove $W_{p^*}$ on $[1, 2p^*+1]$ is approximately $1$.} 
    
    \item \sd{$W_{p^*}$ is close to $1$ whenever there is a wide enough separation between the two curves. The diffusive nature of the random walk (with positive drift) prevents the walk from dipping too low under the free law. Thus under $\m{B}\cap \m{Fluc}$ we have a uniform separation of $\sqrt{N}$ between the top two curves between $\ll 1, 2p^*+1\ll$.} and deduce that $W_{p^*} \approx 1$ for large $N$. This is the idea of \textbf{Step 5} and concludes that the law of the $H_N^{(1)}(\cdot)$ is close to the free law of a log-gamma random walk starting at $H_N^{(1)}(2p^*+1)$.
    
    \item \sd{One issue in carrying out the arguments in the last two bullet points is that $p^*$ is \textit{random}. Thus the Gibbs property formulated for \textit{fixed} boundary points cannot be directly applied at $p^*$. This issue is circumvented by a graining argument where we denote $\m{B}$ as $\m{B}=\bigsqcup\m{B_i}$ for a disjoint collection of events $\m{B}_i \subset \{H_N^{(1)}(2i+1)\ge RN - \tfrac{5}{2}M\tau\sqrt{N}\}$ defined in \textbf{Step 2} and then apply the usual Gibbs property for each $i$.}  
\end{itemize}
\medskip

\textbf{Step 2.} Take $M_1=2M$ and $M_2=3M \tau+1$ in Proposition \ref{p:benchmark}. Taking $N_2(\e,M_1,M_2)>0$ (which depends only on $\e$ as $M_1,M_2$ depends only on $\e$) from Proposition \ref{p:benchmark}, we see that 
\begin{align}\label{def:fluc}
\Pr(\m{Fluc}) \ge 1-\tfrac{\e}{2}, \mbox{ where }\m{Fluc}: = \left\{\sup_{p\in \ll1,4M\sqrt{N}+1\rr}  H_N^{(2)}(p)\le RN-(3M\tau+1)\sqrt{N}\right\}
\end{align}
for all $N\ge N_2$. Next we consider the events 
 \begin{align*}
 	\m{G}_i := \big\{H_N^{(1)}(2i+1) \ge RN- \frac{5}{2}M\tau\sqrt{N}\big\}\mbox{ and } \m{B}_i: =  \bigcap_{j = i+1}^{2M\sqrt{N}} \m{G}_j^c\cap \m{G}_i.
 \end{align*}
 Note that $(\m{B}_i)_{i\in \ll M\sqrt{N},2M\sqrt{N}\rr}$ forms a disjoint collection of events. Define
 \begin{align*}
     \m{B} & :=\bigsqcup_{i\in \ll M\sqrt{N},2M\sqrt{N}\rr} \m{B}_i \\ & =\bigcup_{i\in \ll M\sqrt{N},2M\sqrt{N}\rr} \m{G}_i=\left\{\sup_{p\in \ll M\sqrt{N},2M\sqrt{N}\rr} H_N^{(1)}(2p+1) \ge RN- \tfrac{5}{2}M\tau\sqrt{N}\right\},
 \end{align*} 
where we write $\sqcup$ to stress that the events are disjoint in the union. In particular, as Lemma \ref{l:high} holds, there exists $N_1(\e,M)>0$ such that $\Pr(\m{B}) \ge 1- \tfrac{1}{2}\e$ for all $N\ge N_1$. Thus for all $N\ge N_1+N_2$, by a union bound we have
\begin{align*}
    |\Pr(\m{A})-\Pr(\m{A}\cap \m{B}\cap \m{Fluc})| \le \Pr(\neg \m{B})+\Pr(\neg \m{Fluc}) \le \e.
\end{align*}
Hence to prove \eqref{eq:tech} it suffices to show
\begin{align}\label{eq:mconv}
    |\Pr(\m{A}\cap \m{B}\cap \m{Fluc})-\Pr_{RW}(\m{A})| \le 8\e.
\end{align}
Define $\mathcal{F}_i$ as the $\sigma$-field 
 $\sigma\big({H_N^{(1)}(x)}_{x \ge 2i+1}, {H_N^{(j)}}(x)_{j \ge 2,x\ge 1}\big).$ Note that $\m{B}_i, \m{Fluc}$ are both measurable w.r.t. $\mathcal{F}_i.$ Exploiting the fact that $\m{B}_i$ events are disjoint yields
 \begin{align}\label{eq:decomp}
     \Pr(\m{A}\cap \m{B}\cap \m{Fluc})= \sum_{i=M\sqrt{N}}^{2M\sqrt{N}} \Ex\left[\ind_{\m{B}_i \cap \m{Fluc}}\Ex\left[\ind_{\m{A}}\mid \mathcal{F}_i\right]\right]
 \end{align}
where the last equality is due to the tower property of the conditional expectation. Thus we are left to estimate $\Ex\left[\ind_{\m{A}}\mid \mathcal{F}_i\right]$ for each $i$.

\medskip

\noindent\textbf{Step 3. Gibbs law.} To analyze $\Ex\left[\ind_{\m{A}}\mid \mathcal{F}_i\right]$, we invoke the Gibbs property (Theorem \ref{t:gibbs}) for the $\hslg$ line ensemble.  By Theorem \ref{t:gibbs}, the distribution of $(H_N^{(1)}(j))_{j=1}^{2i}$ conditioned on $\mathcal{F}_i$ has a density at $(u_j)_{j=1}^{2i}$
\begin{align}\label{cold}
     & \exp\left(-\sum_{j=1}^{i} \left[e^{H_N^{(2)}(2j)-u_{2j+1}}+e^{H_N^{(2)}(2j)-u_{2j-1}}\right]\right) \\ & \label{colr} \hspace{2cm} \cdot \prod_{j=1}^{i} \exp\left((\theta+\alpha)(u_{2j+1}-u_{2j})-e^{u_{2j+1}-u_{2j}}\right) \\ & \label{colb} \hspace{4cm} \cdot \prod_{j=1}^{i}\exp\left((\theta-\alpha)(u_{2j-1}-u_{2j})-e^{u_{2j-1}-u_{2j}}\right)
\end{align}
with $u_{2i+1}=H_N^{(1)}(2i+1)$. The above explicit expression is obtained from \eqref{eq:gibbs} and \eqref{eq:color}. Note that the terms in \eqref{cold}, \eqref{colr}, and \eqref{colb} correspond to weights of black, red, and blue edges in the graphical representation (see left figure of Figure \ref{figsm3}) respectively.

Based on the above decomposition, we define a free law $\Pr_{\operatorname{free},i}$ that depends only on $H_N^{(1)}{(2i+1)}$. Let  $\Pr_{\operatorname{free},i}$  be the law under which the distribution of  $(H_N^{(1)}(j))_{j=1}^{2i}$ has a density at $(u_{j})_{j=1}^{2i}$ proportional to 
\begin{align*}
      \prod_{j=1}^{i} \exp\left((\theta+\alpha)(u_{2j+1}-u_{2j})-e^{u_{2j+1}-u_{2j}}\right)\cdot \prod_{j=1}^{i}\exp\left((\theta-\alpha)(u_{2j-1}-u_{2j})-e^{u_{2j-1}-u_{2j}}\right)
\end{align*}
with $u_{2i+1}=H_N^{(1)}(2i+1)$. Note that free law collects all the blue and red edge weights only. A quick comparison of the above formula with \eqref{lgrw:den} shows that under the free law, $(H_N^{(1)}(1)-H_N^{(1)}(2r+1))_{r=0}^{i}$ is precisely distributed as log-gamma random walk defined in Definition \ref{def:lgrw}.

In order to obtain the original conditional distribution from the free law, we may introduce the black weights as a Radon-Nikodym derivative (see the decomposition in Figure \ref{figsm3}). Indeed, we have
\begin{align}\label{eq:giden}
    \Ex\left[\ind_{\m{A}}\mid \mathcal{F}_i\right]=\frac{\Ex_{\operatorname{free},i}[W_i\ind_{\m{A}}]}{\Ex_{\operatorname{free},i}[W_i]}
\end{align}
where
\begin{align}\label{def:w}
    W_i:=\exp\left(-\sum_{j=1}^{i} \left[e^{H_N^{(2)}(2j)-H_N^{(1)}(2j+1)}+e^{H_N^{(2)}(2j)-H_N^{(1)}(2j-1)}\right]\right)
\end{align}

\begin{figure}[h!]
    \centering
    \begin{tikzpicture}[line cap=round,line join=round,>=triangle 45,x=1.3cm,y=1cm]
    \draw[fill=gray!10,line width=0.5pt,dashed] (-6.75,10.25)--(-6.75,9.25)--(-3.75,9.25)--(-3.75,10.25)--(-6.75,10.25);
    \draw[fill=gray!10,line width=0.5pt,dashed] (-2.45,10.25)--(-2.45,9.25)--(0.55,9.25)--(0.55,10.25)--(-2.45,10.25);
    \draw[fill=gray!10,line width=0.5pt,dashed] (1.25,10.25)--(1.25,9.25)--(4.25,9.25)--(4.25,10.25)--(1.25,10.25);
     \node at (-3,9.7) {\Large{$=$}};
     \node at (1,9.7) {\Large{$\times$}};
     \foreach \x in {-6,-5,-4}{
     \draw[line width=1pt,red,{Latex[length=2mm]}-]  (\x,10) -- (\x+0.5,9.5);
     \draw[line width=1pt,blue,{Latex[length=2mm]}-]  (\x,10) -- (\x-0.5,9.5);
     \foreach \y in {10}{
     \draw[line width=1pt,black,{Latex[length=2mm]}-]  (\x+0.5,\y-0.5) -- (\x,\y-1);
     \draw[line width=1pt,black,{Latex[length=2mm]}-]  (\x-0.5,\y-0.5) -- (\x,\y-1);
     }
     }
        \foreach \x in {-1.7,-0.7,0.3}{
     \draw[line width=1pt,red,{Latex[length=2mm]}-]  (\x,10) -- (\x+0.5,9.5);
     \draw[line width=1pt,blue,{Latex[length=2mm]}-]  (\x,10) -- (\x-0.5,9.5);
     }
    \foreach \x in {2,3,4}{
     \foreach \y in {10}{
     \draw[line width=1pt,black,{Latex[length=2mm]}-]  (\x+0.5,\y-0.5) -- (\x,\y-1);
     \draw[line width=1pt,black,{Latex[length=2mm]}-]  (\x-0.5,\y-0.5) -- (\x,\y-1);
     }
    }
     \foreach \x in {-6,-5,-4,-1.7,-0.7,0.3,2,3,4}{
     \draw[fill=green!70!black] (\x,10) circle (1.5pt); 
     \draw[fill=green!70!black] (\x-0.5,9.5) circle (1.5pt); 
     \draw[fill=white] (\x,9) circle (1.5pt); 
     }
     \draw[fill=white] (4.5,9.5) circle (1.5pt);
     \draw[fill=white] (-3.5,9.5) circle (1.5pt);
     \draw[fill=white] (0.8,9.5) circle (1.5pt);
     \node at (4.7, 9.5) {$a$};
     \node at (-3.5,9.2) {$a$};
     \node at (0.8,9.3) {$a$};
     \node at (-6,8.7) {$z_1$};
     \node at (-5,8.7) {$z_2$};
     \node at (-4,8.7) {$z_3$};
     \node at (-1.7,8.7) {$z_1$};
     \node at (-0.7,8.7) {$z_2$};
     \node at (0.3,8.7) {$z_3$};
     \node at (2,8.7) {$z_1$};
     \node at (3,8.7) {$z_2$};
     \node at (4,8.7) {$z_3$};
     \end{tikzpicture}
    \caption{Gibbs decomposition. The left figure shows the Gibbs measure corresponding to conditioned on $\mathcal{F}_i$ with $i=3$. Here $a=H_N^{(1)}(2i+1)$, and $z_j:=H_N^{(2)}(2j)$ for $j\in \ll 1,i\rr$. The measure has been decomposed into two parts. The free law (middle) and a Radon-Nikodym derivative (right).}
    \label{figsm3}
\end{figure}
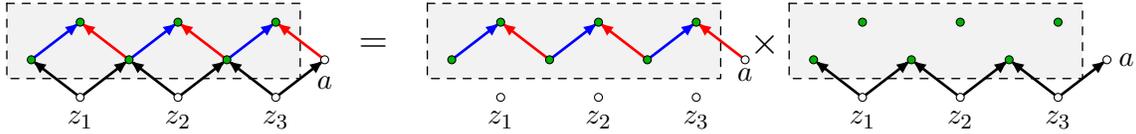

We notice that $W_i$ has a trivial upper bound: $W_i \le 1.$ For the lower bound, we claim that there exists $N_0(\e)>0$ such that for all $N\ge N_0$ we have
 \begin{align}\label{eq:favw}
     \ind_{\m{Fluc}\cap \m{B}_i}\Pr_{\operatorname{free},i}(W_i\ge 1-\e) \ge \ind_{\m{Fluc}\cap \m{B}_i} \cdot (1-\e).
 \end{align}
 Thus, \eqref{eq:favw} implies that $W_i$ is close to 1 with high probability under $\m{Fluc}\cap\m{B}_i$. Thus, going back to \eqref{eq:giden}, we expect $\Ex\left[\ind_{\m{A}}\mid \mathcal{F}_i\right]$ to be close to $\Pr_{\operatorname{free},i}(\m{A})$. As under the free law $\Pr_{\operatorname{free},i}(\m{A})= \Pr_{RW}(\m{A})$, for all $i\in \ll M\sqrt{N},2M\sqrt{N}\rr$, \eqref{eq:decomp} eventually leads to \eqref{eq:mconv}, which we make precise in the next step.

\medskip

 \noindent\textbf{Step 4.} Assuming \eqref{eq:favw}, we complete the proof of \eqref{eq:mconv} in this step. As $W_i \le 1,$ we have
 \begin{align*}
     \ind_{\m{Fluc}\cap \m{B}_i}\frac{\Ex_{\operatorname{free},i}[W_i\ind_{\m{A}}]}{\Ex_{\operatorname{free},i}[W_i]}& \ge \ind_{\m{Fluc}\cap \m{B}_i}\Ex_{\operatorname{free},i}[W_i\ind_{\m{A}}] \ge (1-\e)\cdot \ind_{\m{Fluc}\cap \m{B}_i}\Pr_{\operatorname{free},i}(\m{A} \cap \{W_i \ge 1-\e\})\\& \ge (1-\e)\cdot \ind_{\m{Fluc}\cap \m{B}_i}\left[ \Pr_{\operatorname{free},i}(\m{A}) -  \Pr_{\operatorname{free},i}(W_i < 1- \e)\right]\\& \ge (1-\e)\cdot\ind_{\m{Fluc}\cap \m{B}_i}\left[ \Pr_{\operatorname{free},i}(\m{A}) - \e\right]
 \end{align*}
 where we use \eqref{eq:favw} in the last inequality. Recall $\Pr_{\operatorname{free},i}(\m{A})=\Pr_{RW}(\m{A})$. Inserting this bound in \eqref{eq:giden} and then going back to \eqref{eq:decomp} yields 
 \begin{align*}
\Pr(\m{A}\cap\m{B}\cap \m{Fluc})  & \ge(1-\e)\cdot \left[ \Pr_{RW}(\m{A}) - \e\right]\sum_{i=M\sqrt{N}}^{2M\sqrt{N}} \Pr(\m{B}_i \cap \m{Fluc})\\& = (1-\e)\cdot\left[ \Pr_{RW}(\m{A}) - \e\right] \Pr(\m{B} \cap \m{Fluc}) \ge (1-\e)^2\left[\Pr_{RW}(\m{A})-\e\right].
 \end{align*}
 for all large enough $N$. The equality in the above equation follows by recalling that  $\m{B}_i$'s form a disjoint collection of events and the result implies that $\Pr(\m{A}\cap\m{B}\cap \m{Fluc})-\Pr_{RW}(\m{A}) \ge -3\e$. This proves the lower bound inequality in \eqref{eq:mconv}. Similarly for the upper bound, as $W_i\le 1$, we have
  \begin{align*}
     \ind_{\m{Fluc}\cap \m{B}_i}\cdot\frac{\Ex_{\operatorname{free},i}[W_i\ind_{\m{A}}]}{\Ex_{\operatorname{free},i}[W_i]} & \le \ind_{\m{Fluc}\cap \m{B}_i}\cdot\frac{\Pr_{\operatorname{free},i}(\m{A})}{(1-\e)\Pr_{\operatorname{free},i}(W_i \ge 1-\e)} \le \ind_{\m{Fluc}\cap \m{B}_i}\cdot \frac{\Pr_{\operatorname{free},i}(\m{A})}{(1-\e)^2}
 \end{align*}
 where the last inequality stems from \eqref{eq:favw}. Again, inserting this bound in \eqref{eq:giden} and then going back to \eqref{eq:decomp} gives us
 \begin{align*}
     \Pr(\m{A}\cap\m{B}\cap \m{Fluc})   & \le  \frac{\Pr_{RW}(\m{A})}{(1-\e)^2}\sum_{i=M\sqrt{N}}^{2M\sqrt{N}} \Pr(\m{B}_i \cap \m{Fluc}) = \frac{\Pr_{RW}(\m{A})}{(1-\e)^2} \Pr(\m{B} \cap \m{Fluc}) \le \frac{\Pr_{RW}(\m{A})}{(1-\e)^2}
 \end{align*}
where again the equality comes from the disjointness of $\m{B}_i$'s. As $\e\le \frac12$, this implies 
\begin{align*}
    \Pr(\m{A}\cap\m{B}\cap \m{Fluc}) -\Pr_{RW}(\m{A}) \le \frac{1-(1-\e)^2}{(1-\e)^2} \le 8\e
\end{align*}
 which proves the upper bound in \eqref{eq:mconv}. The proof of Theorem \ref{t:walk} modulo \eqref{eq:favw} is thus complete.

\medskip

\noindent\textbf{Step 5.} Finally in this step we prove \eqref{eq:favw}.  We define the event
 \begin{align*}
 \m{Sink}(i):= \left\{\inf_{p \in \ll 0,i\rr} H_N^{(1)}(2p+1)\ge RN -3M\tau\sqrt{N}\right\} .
 \end{align*} 
 We claim that there exists $N_0(\e)>0$ such that for all $N\ge N_0$, we have
\begin{align}\label{eq:sinkbd}
    \ind_{\m{B}_i}\Pr_{\operatorname{free},i}(\m{Sink}(i)) \ge \ind_{\m{B}_i}(1-\e),
\end{align}
for all $i\in \ll M\sqrt{N},2M\sqrt{N}\rr$. 

Recall that the event $\m{Fluc}$ in \eqref{def:fluc} requires the second curve $H_N^{(2)}(p)$ to lie below certain threshold within the range $p\in \ll1,4M\sqrt{N}+1\rr$. Recall the definition of $W_j$ from \eqref{def:w}. Note that on $\m{Sink}(j)\cap \m{Fluc}$ we have $$W_j \ge \exp(-2j e^{-\sqrt{N}}) \ge \exp(-4M\sqrt{N}e^{-\sqrt{N}})$$ 
as $j\le 2M\sqrt{N}$. Note that $ \exp(-4M\sqrt{N}e^{-\sqrt{N}})\ge 1-\e$ for all large enough $N$. Therefore, assuming \eqref{eq:sinkbd} we have 
\begin{align*}
  \ind_{\m{Fluc}\cap \m{B}_i}\Pr_{\operatorname{free},i}(W_i\ge 1-\e) \ge \ind_{\m{Fluc}\cap \m{B}_i}\Pr_{\operatorname{free},i}(\m{Sink}(i)) \ge \ind_{\m{Fluc}\cap \m{B}_i}\cdot (1-\e)  
\end{align*}
 for all large enough $N$. This verifies \eqref{eq:favw}. We are left to show \eqref{eq:sinkbd}. Towards this end, note that on the event $\m{B}_i$, we have $H_N^{(1)}(2i+1) \ge RN-\frac{5}{2}M\tau \sqrt{N}$. Thus,
\begin{equation}
     \label{eq:biuse}
\begin{aligned}
    \ind_{\m{B}_i}\Pr_{\operatorname{free},i}(\m{Sink}(i)) \ge \ind_{\m{B}_i}\Pr_{\operatorname{free},i}\left(\inf_{x \in \ll 0,i\rr} H_N^{(1)}(2x+1)-H_N^{(1)}(2i+1) \ge -\tfrac{1}{2}M\tau\sqrt{N}\right).
\end{aligned}
\end{equation}
Recall from our discussion in \textbf{Step 2} that under the law $\Pr_{\operatorname{free},i}$, $(H_N^{(1)}(1)-H_N^{(1)}(2r+1))_{r=0}^i$ is distributed as a log-gamma random walk. Let us use $(S_k)_{k=0}^i$ to denote a log-gamma random walk. We have 
\begin{equation}
    \label{eq:laweq}
\begin{aligned}
    & \Pr_{\operatorname{free},i}\left(\inf_{p \in \ll 0,i\rr} H_N^{(1)}(2p+1)-H_N^{(1)}(2i+1) \ge  -\tfrac{1}{2}M\tau\sqrt{N}\right) \\ & \hspace{5cm}=\Pr\left(\inf_{p \in \ll 0,i\rr} (S_i-S_p) \ge  -\tfrac{1}{2}M\tau\sqrt{N}\right).
\end{aligned}
\end{equation}
Note that $(S_i-S_p)_{p= 0}^i$ is again a time-reversed log-gamma random walk. As $i\le 2M\sqrt{N}$, appealing to Lemma \ref{l:kol} yields that
 \begin{align*}
 	\ind_{\m{B}_i}\Pr_{\operatorname{free},i}(\m{Sink}(i)) \ge  \ind_{\m{B}_i}\Pr\left(\inf_{p \in \ll 0,i\rr} (S_i-S_p) \ge -\tfrac{1}{2}M\tau \sqrt{N}\right)\ge \ind_{\m{B}_i}\left[1- \frac{8\operatorname{Var}(S_1)}{M\tau^2\sqrt{N}}\right] \ge \ind_{\m{B}_i}(1-\e)
 \end{align*} 
 for all large enough $N$ (uniformly over $i\in \ll M\sqrt{N},2M\sqrt{N}\rr$). This verifies  \eqref{eq:sinkbd}, completing the proof of Proposition \ref{p:tech}.

\appendix

\section{Properties of random walks with positive drift}
\label{sec:app}
In this section, we collect some useful properties of random walks with positive drift whose proofs follow by classical analysis. Note that the log-gamma random walk introduced in Definition \ref{def:lgrw} is a random walk with positive drift. This is because the density $p(x)$ introduced in \eqref{lgrw:den} has mean:
\begin{align*}
    \int_{\R} xp(x)dx = \Psi(\theta-\alpha)-\Psi(\theta+\alpha),
\end{align*}
which is positive as the digamma function $\Psi$ is strictly increasing (recall $\alpha<0$).
\begin{lemma}\label{l:kol}
Let $(X_i)_{i\ge 0}$ be a sequence of iid random variables with $\Ex[X_1]=\beta>0$ and $\operatorname{Var}[X_1]=\gamma<\infty$. Set $S_0=0$ and $S_k=\sum_{i=1}^k X_i$. For all $M,N, \lambda>0$ we have
\begin{align*}
    \Pr\left(\inf_{k \in \llbracket1, M\sqrt{N}\rrbracket} S_k \le -\lambda\right) \le \frac{M\sqrt{N}\gamma}{\lambda^2}.
\end{align*}
\end{lemma}
\begin{proof}
As $\beta>0$, by Kolmogorov's maximal inequality, we have 
\begin{align*}
   \Pr\left(\inf_{k \in \llbracket1, M\sqrt{N}\rrbracket} S_k \le - \lambda\right) & \le \Pr\left(\sup_{k \in \llbracket1, M\sqrt{N}\rrbracket} |S_k-k\beta| \ge \lambda\right)  \le \frac{1}{\lambda^2}\sum_{i = 1}^{M\sqrt{N}}\operatorname{Var}(X_i) =  \frac{M\sqrt{N}\gamma}{\lambda^2},
\end{align*} 
which is precisely what we want to show.
\end{proof}

\begin{lemma}\label{a1} Let $(X_i)_{i\ge 0}$ be a sequence of iid random variables with $\Ex[X_1]=\beta>0$ and $\operatorname{Var}[X_1]=\gamma<\infty$. Set $S_0=0$ and $S_n=\sum_{i=1}^n X_i$. We have
\begin{align*}
    \Pr\left(\sum_{r=0}^{\infty} e^{-S_r} <\infty\right)=1
\end{align*}
\end{lemma}
\begin{proof} By Kolmogorov's maximal inequality
    \begin{align*}
        \Pr\big(\sup_{1\le i\le n^2} |S_i-i\beta| \ge \tfrac{n^2}{2}\beta\big) \le \frac{4}{n^4\beta^2}\sum_{i=1}^{n^2}\operatorname{Var}(X_i) =\frac{4\gamma}{n^2\beta^2}.
    \end{align*} 
The last bound is summable in $n$. Thus invoking Borel-Cantelli's lemma we have that there exists a random $N$ with $\Pr(7\le N<\infty)=1$ such that 
$$S_{i} \ge i\beta-(N^2/2)\beta \ge -(N^2/2)\beta,\mbox{ for all }1\le i\le N^2,$$
and for all $n\ge N+1$ we have $$ S_{i} \ge (n-1)^2\beta-(n^2/2)\beta \ge (n^2/4)\beta,\mbox{ for all }(n-1)^2+1 \le i\le n^2,$$
where above we used the fact that $n\ge N+1\ge 8$. Thus with probability $1$, we have
\begin{align*}
    \sum_{r=0}^{\infty} e^{-S_r} & =\sum_{r=0}^{N^2} e^{-S_r} +\sum_{n=N+1}^{\infty}\sum_{i=(n-1)^2+1}^{n^2} e^{-S_i} \\ & \le N^2 e^{(N^2/2)\beta}+\sum_{n=N+1}^{\infty}\sum_{i=(n-1)^2+1}^{n^2} e^{-(n^2/4)\beta}  \le N^2 e^{(N^2/2)\beta}+\sum_{n=N+1}^{\infty} n^2e^{-(n^2/4)\beta} <\infty.
\end{align*}
This completes the proof.
\end{proof}
As a corollary, we have the following double-limit result.
\begin{corollary}\label{l:cor1}
Under the setup of Lemma \ref{a1}, almost surely we have
\begin{align*}
     \lim_{k\to \infty}\lim_{n\to \infty} \frac{\sum_{r=k}^{n} e^{-S_r}}{\sum_{r=0}^n e^{-S_r}} =0.
\end{align*}
\end{corollary}
\begin{proof} Note that $\sum_{r=0}^n e^{-S_r}$ is a monotone sequence in $n$ which converges to a random variable that is almost surely finite by Lemma \ref{a1}. Thus, 
    \begin{align*}
        \frac{\sum_{r=k}^{n} e^{-S_r}}{\sum_{r=0}^n e^{-S_r}} = 1-\frac{\sum_{r=0}^{k-1} e^{-S_r}}{\sum_{r=0}^n e^{-S_r}} \stackrel{n\to\infty}{\to} 1-\frac{\sum_{r=0}^{k-1} e^{-S_r}}{\sum_{r=0}^{\infty} e^{-S_r}}.
    \end{align*}
Taking $k\to \infty$ yields the desired result.
\end{proof}

\noindent\textbf{Data availability statement.} Data sharing not applicable to this article as no datasets were
generated or analyzed during the current study.

\noindent\textbf{Statement and Declarations.} The authors declare that they have no known competing financial or non-financial interests that could have appeared to influence the work reported in this paper.

\bibliographystyle{alpha}
\bibliography{bound}
\end{document}